\documentclass{amsart}

\RequirePackage{doi}
\usepackage{hyperref}
\usepackage{cite}

\usepackage[all]{xy}
\usepackage{amsmath}
\usepackage{epsfig}
\usepackage{flafter}
\usepackage{mathtools}
\usepackage{comment}
\usepackage{stmaryrd}
\usepackage{tabularx}

\hypersetup{
    colorlinks=true,    
    linkcolor=blue,          
    citecolor=blue,      
    filecolor=blue,      
    urlcolor=blue           
}
\usepackage{tikz}
\usetikzlibrary{graphs,positioning,arrows,shapes.misc,decorations.pathmorphing}

\tikzset{
    >=stealth,
    every picture/.style={thick},
    graphs/every graph/.style={empty nodes},
}

\tikzstyle{vertex}=[
    draw,
    circle,
    fill=black,
    inner sep=1pt,
    minimum width=5pt,
]
\usepackage[position=top]{subfig}
\usepackage{amssymb}
\usepackage{color}

\calclayout

\usetikzlibrary{decorations.pathmorphing}
\tikzstyle{printersafe}=[decoration={snake,amplitude=0pt}]

\makeatletter
\@namedef{subjclassname@2020}{\textup{2020} Mathematics Subject Classification}
\makeatother

\newcommand{\Aut}{\operatorname{Aut}}
\newcommand{\aff}{\operatorname{aff}}
\newcommand{\gr}{\operatorname{gr}}
\newcommand{\PDiv}{\operatorname{PDiv}}
\newcommand{\WDiv}{\operatorname{WDiv}}
\newcommand{\Pic}{\operatorname{Pic}}
\newcommand{\reg}{\operatorname{reg}}
\newcommand{\Cl}{\operatorname{Cl}}

\newcommand{\id}{\operatorname{id}}
\newcommand{\loc}{\operatorname{loc}}

\newcommand{\supp}{\operatorname{supp}}

\newcommand{\Spec}{\operatorname{Spec}}

\newcommand{\divv}{\operatorname{div}}

\newcommand{\Cox}[1]{\mathcal{R}_{#1}}

\newcommand{\pp}{\mathbb{P}}

\newcommand{\qq}{\mathbb{Q}}
\newcommand{\zz}{\mathbb{Z}}
\newcommand{\nn}{\mathbb{N}}

\newcommand{\cc}{\mathbb{C}}

\usepackage{tikz}
\usetikzlibrary{matrix,arrows,decorations.pathmorphing}

  \newtheorem{introthm}{Theorem}

  \newtheorem{theorem}{Theorem}[section]
  \newtheorem{lemma}[theorem]{Lemma}
  \newtheorem{proposition}[theorem]{Proposition}
  \newtheorem{corollary}[theorem]{Corollary}

  \newtheorem{notation}[theorem]{Notation}
  \newtheorem{definition}[theorem]{Definition}
  \newtheorem{example}[theorem]{Example}

\newtheorem{remark}[theorem]{Remark}

\theoremstyle{remark}

\numberwithin{equation}{section}

\begin{document}

\title[Iteration of Cox rings of klt singularities]{Iteration of Cox rings of klt singularities}

\author[L.~Braun]{Lukas Braun}
\address{Mathematisches Institut, Albert-Ludwigs-Universit\"at Freiburg, Ernst-Zermelo-Strasse 1, 79104 Freiburg im Breisgau, Germany}
\email{lukas.braun@math.uni-freiburg.de}


\author[J.~Moraga]{Joaqu\'in Moraga}
\address{Department of Mathematics, Princeton University, Fine Hall, Washington Road, Princeton, NJ 08544-1000, USA
}
\email{jmoraga@princeton.edu}
\maketitle

\begin{abstract}
In this article, we study the iteration of Cox rings of klt singularities (and Fano varieties) 
from a topological perspective.
Given a klt singularity $(X,\Delta;x)$,
we define the iteration of Cox rings of $(X,\Delta;x)$.
The first result of this article
is that the iteration of Cox rings ${\rm Cox}^{(k)}(X,\Delta;x)$ of a klt singularity 
stabilizes for $k$ large enough. 
The second result is a boundedness one, 
we prove that for a $n$-dimensional klt singularity $(X,\Delta;x)$ the iteration of Cox rings  stabilizes for $k\geq c(n)$,
where $c(n)$ only depends on $n$.
Then, we use Cox rings to establish the existence of a simply connected factorial canonical cover (or scfc cover) of a klt singularity.
The scfc cover generalizes both, the universal cover and the iteration of Cox rings.
We prove that the scfc cover dominates any sequence of quasi-\'etale finite covers and reductive abelian quasi-torsors of the singularity.
We characterize when the iteration of Cox rings is smooth and when the scfc cover is smooth. We also characterize when the spectrum of the iteration coincides with the scfc cover. 
Finally, we give a complete description of the regional fundamental group,
the iteration of Cox rings,
and the scfc cover of klt singularities of complexity one.
Analogous versions of all our theorems are also proved for Fano type morphisms.
To extend the results to this setting,
we show that the Jordan property holds
for the regional fundamental group of Fano type morphisms.
\end{abstract}

\setcounter{tocdepth}{1}
\tableofcontents

\section{Introduction}

The Cox ring ${\rm Cox}(X)$ of an algebraic variety $X$ captures
the geometry of the variety and all the line bundles on it~\cite{ADHL15}.
This ring is also known as universal torsor in arithmetic geometry~\cite{CTS76,CTS77}.
The Cox ring construction generalizes the classic description
of toric varieties as quotients of (big open subsets of) affine spaces~\cite{Cox95}.
Whenever the $\cc$-algebra ${\rm Cox}(X)$
is finitely generated, it controls all birational 
models of $X$ via GIT~\cite{HK00}.
In this case, we say that $X$ is a {\em Mori dream space}, 
since the variety $X$ behaves optimally with respect
to the minimal model program (also known as Mori program).
The use of Cox rings to study varieties 
has become a standard technique in algebraic geometry.
An explicit presentation of the Cox ring
in terms of generators and relations often enlightens 
the geometry of the variety.
In this direction, there are several results on 
weak del Pezzo surfaces~\cite{BP02,HT02}.
Many results about Cox rings have been obtained for 
K3 surfaces~\cite{AHL10,ACL21}.
In higher dimensions, the Cox rings of some Fano manifolds 
have been described explicitly~\cite{DHHKL15,HKL16,HLM19}.
The Cox rings of certain Moduli spaces are considered in~\cite{Cas09,MR17,BM17}. 
The study of Cox rings has also become a central topic
in $\mathbb{T}$-varieties~\cite{AH06,AIPSV12}, especially in the case of {\em complexity one}, i.e., 
$n$-dimensional algebraic varieties with an effective action
of a $(n-1)$-dimensional torus~\cite{HaSu10, AP12, BHHN16}.
In this case, the Cox ring can be described combinatorially.
More generally, horospherical varieties can be described using Cox rings~\cite{LT17,Vez20}.
In many cases, computations of intersection theory can be carried out in the Cox ring of an algebraic variety.
The most general setting in which Cox rings are known to be well-defined is algebraic stacks~\cite{HM15, HMT20}. In the present work, we will mostly deal with the class of integral, noetherian, normal schemes~\cite[Sec 2.3]{HMT20}.
We refer the reader to~\cite{LV09,ADHL15} for a systematic study of Cox rings.

Cox rings have also been used to study singularities.
In this case, the definition of the Cox ring is
often applied to a certain resolution of singularities. 
In~\cite{FGL11}, the authors study the Cox ring of the minimal resolution of a surface Du Val singularity.
In~\cite{Don16}, the author provides two different descriptions of the Cox ring of the minimal resolution of a quotient singularity.
Further results have been obtained towards the computation of the Cox ring of some minimal (or crepant) resolution of singularities~\cite{DG17,DK17,Gra18}.
In~\cite{ABHW18}, a slightly different approach to studying Kawamata log terminal (klt) singularities via Cox rings is proposed by the authors.
Instead of looking at the Cox ring of a resolution of singularities, 
the authors use the definition of Cox ring on the germ itself.
Then, if possible, this process is iterated to simplify the singularity (and possibly, increasing the dimension).
This construction generalizes the presentation 
of surface klt singularities as quotients of factorial canonical singularities by solvable finite groups.
In~\cite{ABHW18}, the iteration of Cox rings is performed for singularities of complexity one. The iteration has at most four steps and can be read off directly from the first Cox ring.
Furthermore, the last variety in this sequence, the so-called \emph{master Cox ring},
is factorial and it can be listed explicitly.
In~\cite{HW18}, the authors characterize all varieties with a torus action of complexity one that admit a finite iteration of Cox rings. In~\cite{Gag19}, it is shown that for spherical varieties, the iteration of Cox rings has at most two steps. More generally, in the works~\cite{Vez20,Vez20a}, Vezier considers the iteration of Cox rings for $G$-varieties of complexity one and determines bounds on the number of iterations.

In order to define an iteration of Cox rings,
we must check that (the spectrum of) the Cox ring ${\rm Cox}(X)$ 
of our variety $X$ is itself a Mori dream space.
It is known that Fano type varieties are a special class
of Mori dream spaces~\cite{BCHM10}.
Furthermore, the Cox ring of a Fano type variety
is an affine Gorenstein canonical quasi-cone~\cite{GOST15, Bro13, Bra19}. In particular, it is an affine model of a klt singularity.
A klt singularity in turn is a local version
of a Fano type variety.
Indeed, a klt singularity is a relative Mori dream space over itself, i.e., when considering the identity as the structure morphism.
Thus, it is natural to iterate the Cox construction for Fano type varieties, or more generally, for klt singularities.
The first author made this observation in~\cite{Bra19}, where he proves
the existence and termination of the iteration of Cox rings
for Fano type varieties and klt quasi-cones.

\begin{introthm}
[Cf.~\cite{Bra19}]
\label{introthm:braun}
Let $X$ be a 
Fano type variety.
Then, for each $k\geq 0$ the $k$-th iteration of Cox rings
${\rm Cox}^{(k)}(X)$ exists.
Furthermore, the iteration stabilizes for $k$ large enough.
\end{introthm}

We recall that the iteration of Cox rings can lead to two different outcomes: it could stop with a factorial master Cox ring or an affine variety which is not a Mori dream space.
Furthermore, it could lead to an infinite sequence of Cox rings.
In this article, we recover Theorem~\ref{introthm:braun} for klt singularities in the general setting.
This means that the iteration of Cox rings always exists for klt singularities
and terminates after finitely many iterations.
We also generalize the concept of Cox rings for log pairs, 
leading to our first result.

\begin{introthm}\label{introthm2-existence-iteration-local}
Let $(X,\Delta;x)$ be a 
Kawamata log terminal singularity.
Then, for each $k\geq 0$ the $k$-th iteration of Cox rings
${\rm Cox}^{(k)}(X,\Delta;x)$ exists.
Furthermore, the iteration stabilizes for $k$ large enough.
\end{introthm}

The main tool used to prove the above theorem is the finiteness of the regional fundamental group of  a klt singularity.

Two natural questions emanate from the two above theorems.
First, we can ask how many times we need to iterate the Cox construction before it stabilizes.
A natural way to study the iteration of Cox rings is to quotient (in each step) by the connected component of the solvable group acting on each model ${\rm Cox}^{(k)}(X,\Delta;x)$.
In this way, we obtain a sequence of finite solvable Galois covers of the starting singularity (or Fano variety). 
This method was initiated in~\cite{Bra19}.
Using the Jordan property for the regional fundamental group of klt singularities~\cite{BFMS20}, we prove that the number of iterations is bounded from above by a constant which only depends on the dimension.
This means that the iteration of Cox rings is controlled by the topology of the variety (or singularity).
The following theorem has a projective and a local version. 
For simplicity of the exposition, we just write the local version in the introduction.

\begin{introthm}\label{introthm3-bounded-iteration-local}\label{introthm:bounded-iteration}
There exists a constant $c(n)$, only depending on $n$, satisfying the following.
Let $(X,\Delta;x)$ be a $n$-dimensional
Kawamata log terminal singularity.
Then, the $k$-th iteration of Cox rings
${\rm Cox}^{(k)}(X,\Delta;x)$ stabilizes
for $k\geq c(n)$.
\end{introthm}

Secondly, we can ask how (if possible) to control
the dimension of the iteration of Cox rings.
For instance, we can ask if there is any invariant of the singularity which can give an upper bound for the dimension of the master Cox ring.
Note that, in general, the iteration of Cox rings could have arbitrarily large dimension.
Indeed, the spectrum of the Cox ring of an affine toric variety of dimension $n$ and Picard rank $\rho$ is isomorphic
to the affine space $\mathbb{A}^{n+\rho}$.
On the other hand, even if the Picard rank of the singularity is bounded, it could happen that the Cox ring itself (or any of the higher iterated Cox rings) has unbounded Picard rank.
Thus, in general, the Picard rank of the initial germ does not control the dimension of the master Cox ring.
This leads to our third result in terms of the second homotopy group of the smooth locus.
The following theorem answers the above question.

\begin{introthm}
\label{introthm4-bounded-dim-it-local}
Let $n$ and $o$ be positive integers.
Let $(X,\Delta;x)$ be a n-dimensional Kawamata log terminal singularity. 
Assume that $\pi_2^{\rm reg}(X,\Delta;x) \otimes \mathbb{Q}$ has rank $o$.
Then, the master Cox ring of $(X,\Delta;x)$ has dimension at most $n+o$.
\end{introthm}

We will prove Theorem~\ref{introthm2-existence-iteration-local} in two different settings, for two different definitions of the iteration of Cox rings. 
We will define the iteration of Cox rings for the (Zariski) local ring and the Henselization of the local ring (i.e., the local ring in the \'etale topology) of a Kawamata log terminal singularity.
The first one will be called the {\em affine iteration}, while the second will be called the {\em Henselian iteration}.

The advantage of the affine iteration is that the outcome of the iteration is an affine klt variety with a distinguished point.
Thus, techniques of affine geometry can be applied to the master Cox ring in this case.
On the other hand, the Henselian iteration captures the local topology of the singularity.
This is the main property that we will use for our next theorem.
We prove that klt singularities admit factorial canonical simply connected covers. 
This cover can be understood as a cover that encompasses all the good properties of the universal cover and the iteration of Cox rings. 

\begin{introthm}
\label{introthm-5-existence-scf-cover}
Let $(X,\Delta;x)$ be a Kawamata log terminal singularity.
Let $X^h$ be the spectrum of the Henselization of the local ring of $X$ at $x$.
There exists a local Henselian ring $R_Y$ so that $Y={\rm Spec}(R_Y)$ satisfies:
\begin{enumerate}
    \item $Y$ is canonical factorial,
    \item $\pi_1^{\rm reg}(Y,y)$ is trivial, 
    \item $Y$ admits the action of a reductive group $G$, and 
    \item we have an isomorphism $Y/G\cong X^h$.
\end{enumerate}
Furthermore, $G$ is an extension of a solvable reductive group and $\pi_1^{\rm reg}(X,\Delta,x)$.
\end{introthm}

Throughout this article, reductive groups are not assumed to be connected. In particular, a solvable reductive group is an extension of a finite solvable group by a torus.

We call the germ $(Y,y)$ constructed in Theorem~\ref{introthm-5-existence-scf-cover} the {\em simply connected factorial canonical cover} of the klt singularity, or {\em scfc} cover for short.
Note that the name of this cover is idiosyncratic since the condition on the regional fundamental group is stronger than being simply connected. However, in the context of singularities, it is natural to consider the fundamental group of the smooth locus instead of the fundamental group of the germ itself.
The scfc cover is a generalization of both; the universal cover and the iteration of Cox rings of a singularity.
Furthermore, it satisfies the universal cover of both aforementioned covers.
Our next result says that the scfc cover of a Kawamata log terminal singularity dominates any sequence of pointed abelian covers and pointed finite covers.

\begin{introthm}
\label{introthm-6-univ-scf-cover}
Let $(X,\Delta;x)$ be a Kawamata log terminal singularity.
Let $X^h$ be the spectrum of the Henselization of the local ring of $X$ at $x$.
Let $(Y,y)$ be the scfc cover of $(X,\Delta;x)$.
Let 
\[
(X^h,x) \leftarrow (X_1,x_1) \leftarrow (X_2,x_2) \leftarrow \dots \leftarrow (X_n,x_n)
\]
be a sequence of pointed finite covers and pointed abelian covers. 
Let $X_n^h$ be the spectrum of the Henselization of the local ring of $X_n$ at $x_n$.
Then, there is a quotient morphism $Y\rightarrow X_n^h$.
\end{introthm}

In view of the above theorem, the scfc cover of a Kawamata log terminal singularity can be regarded as the best singularity that can be obtained from $(X,\Delta;x)$ by taking sequences of finite covers and abelian covers, 
more generally, by taking solvable-finite covers, i.e., covers by finite extensions of solvable reductive groups.
So far, we have four different covers of klt singularities (see Appendix~\ref{appendix}).
The universal cover, 
the Cox ring,
the iteration of Cox rings,
and the scfc cover.
The universal cover of a klt singularity is smooth
if and only if the singularity is the quotient of $\cc^n$ by a finite group acting linearly. 
Furthermore, we know that the Cox ring of a klt singularity is smooth if and only if the 
singularity is formally toric. 
The following theorem characterizes when the iteration of Cox rings of a klt singularity is smooth. 

\begin{introthm}
\label{introthm7-smooth-it}
Let $(X,\Delta;x)$ be a klt singularity.
Then, the following statements are equivalent:
\begin{enumerate}
\item The spectrum of the iteration of Cox rings ${\rm Cox}^{\rm it}(X,\Delta;x)$ is smooth, and
\item $(X,\Delta;x)$ is a finite quasi-\'etale solvable quotient of a toric singularity. 
\end{enumerate}
\end{introthm}

The following theorem characterizes when the scfc cover of a klt singularity is smooth. 

\begin{introthm} 
\label{introthm8-smooth-scfc}
Let $(X,\Delta;x)$ be a klt singularity.
Then, the following statements are equivalent:
\begin{enumerate}
\item The simply connected factorial canonical cover of $(X,\Delta)$ is smooth, and 
\item $(X,\Delta)$ is a finite quasi-\'etale quotient of a projective toric singularity. 
\end{enumerate}
\end{introthm}

In Appendix~\ref{appendix}, we show a diagram with all the natural morphisms among the covers considered in this article.
In this direction, it is also natural to compare the iteration of Cox rings with the scfc cover.
We prove that they coincide as long  as the regional fundamental group of the klt singularity is a solvable group.

\begin{introthm}
\label{introthm9-equal-it-scfc}
Let $(X,\Delta;x)$ be a klt singularity.
Then, the following are equivalent:
\begin{enumerate}
    \item The spectrum of the iteration of Cox rings coincides with the simply connected factorial canonical cover, and
    \item the regional fundamental group $\pi_1^{\rm reg}(X,\Delta;x)$ is solvable.
\end{enumerate}
\end{introthm}

All the theorems in this article are also proved for Fano type varieties.
In many cases, we also prove the statements for Fano type morphisms.
As mentioned above, the boundedness of iterations is a consequence of the Jordan property for the regional fundamental group of klt singularities~\cite{BFMS20}. 
To generalize to the relative setting,
we will need the following relative version of the Jordan property.

\begin{introthm}
\label{introthm10-jordan-relative}
Let $n$ be a positive integer.
There exists a constant $c(n)$, only depending on $n$, satisfying the following.
Let $\phi\colon X \rightarrow Z$ be a projective contraction so that $X$ has dimension $n$.
Let $(X,\Delta)$ be a log pair of Fano type over $Z$.
Let $z\in Z$ be a closed point.
Then, the fundamental group 
$\pi_1^{\rm reg}(X/Z,\Delta;z)$ is finite.
Furthermore, there exists a normal abelian subgroup $A\leqslant \pi_1^{\rm reg}(X/Z,\Delta;z)$ of rank at most $n$ and index at most $c(n)$.
\end{introthm}

Here, the group $\pi_1^{\rm reg}(X/Z,\Delta;z)$
consists of loops over a small punctured analytic neighborhood of a closed point $z\in Z$.
In this direction, we also prove an enhanced version of the Jordan property for klt $\mathbb{T}$-singularities.
Recall that the complexity of a $\mathbb{T}$-variety is the dimension of the variety minus the dimension of the torus
acting on it.
In this direction, we prove the following theorem.

\begin{introthm}\label{introthm11-jordan-t-var}
Let $r$ be a positive integer.
There exists a constant $c(r)$, only depending on $r$,
satisfying the following.
Let $(X,\Delta;x)$ be a $n$-dimensional klt $\mathbb{T}$-singularity
of complexity $r$.
Then, there exists a normal abelian subgroup
$A\leqslant \pi_1^{\rm reg}(X,\Delta;x)$ 
of rank at most $n$ and index at most $c(r)$.
\end{introthm}

In particular, the non-abelian quotient of the regional fundamental group of complexity one klt $\mathbb{T}$-singularities is bounded by a constant which is independent of the dimension.
As a consequence of Theorem~\ref{introthm11-jordan-t-var} and the proof of Theorem~\ref{introthm3-bounded-iteration-local}, we conclude the following statement about the iteration of Cox ring of klt $\mathbb{T}$-singularities.

\begin{introthm}\label{introthm12-it-t-var}
Let $r$ be a positive integer.
There exists a constant $c(r)$, 
only depending on $r$,
satisfying the following.
Let $(X,\Delta;x)$ be a klt $\mathbb{T}$-singularity of complexity $r$.
Then, the $k$-th iteration of Cox rings ${\rm Cox}^{(k)}(X,\Delta;x)$ stabilizes for $k\geq c(r)$.
\end{introthm}

Note that the constant $c(r)$ only depends on the complexity and not on the dimension of the germ.
In Subsection~\ref{subsec:compl-one}, we will culminate the article with an extensive study of the regional fundamental group, the iteration of Cox rings, and the scfc covers of klt $\mathbb{T}$-singularities of complexity one. 
The iteration of Cox rings of these singularities has already been considered in the works~\cite{ABHW18, HW18}.

The present article gives a good understanding of the finite-solvable covers of klt singularities and Fano type varieties.
It is natural to try to extend the above results to general reductive groups. 
However, to do so, a better understanding of semi-simple covers of Fano varieties is required.
It is also interesting to consider the opposite question:
whether the quotient of an affine klt singularity by a reductive group is klt type.
The authors will settle this question in a forthcoming article.

\subsection*{Structure of the paper}
The structure of the paper is as follows: in Section~\ref{sec:prel}, we give some preliminaries about Cox rings, graded-local rings, and the minimal model program.
In Section~\ref{sec:gen-cox}, we turn to define Cox rings in a variety of cases, for morphisms of log pairs, over local rings, and over Henselian rings.
In Section~\ref{sec:bounded}, we turn to prove the existence and boundedness of the iteration of Cox rings for klt singularities. 
In Section~\ref{sec:scfc}, we prove the existence of the simply connected factorial canonical cover of a klt singularity.
In Section~\ref{sec:smoothit}, we give a characterization of Fano type varieties with a smooth iteration of Cox rings
and Fano type varieties with smooth scfc cover.
Finally, in Section~\ref{sec:ex}, we give several examples including a complete classification
of the iteration of Cox rings of klt complexity one $\mathbb{T}$-singularities.

\subsection*{Acknowledgements} 
The authors would like to thank Karl Schwede, Stefano Filipazzi, Christopher Hacon, Burt Totaro, and J\'anos Koll\'ar for many useful comments.

\section{Preliminaries}\label{sec:prel}

Throughout this article, we work over the field of complex numbers $\mathbb{C}$.
The rank of a finite group is the least number of generators. As usual, we may denote by $1$ (resp. $0$) the trivial multiplicative (resp. additive) group.

In this section, we collect some preliminary results and definitions.
In Subsection~\ref{subsec:cox-ring}, we recall the concept of Cox rings and Mori dream spaces.
In Subsection~\ref{subsec:grocal-rings}, we prove some properties about the class groups of gr-local rings.
In Subsection~\ref{subsec:gr-Henselian-rings}, we recall the concept of gr-Henselian rings. 
Then, in Subsection~\ref{subsec:shvs-gr-local}, we define sheaves of gr-local rings.
The Cox sheaves considered in this article will be sheaves of gr-local rings.
In Subsection~\ref{subsec:covers-grocal-rings}, we bring together the concepts of regional fundamental groups and Cox rings.
Finally, in Subsection~\ref{subsec:mmp}, we recollect some notions of singularities of the minimal model program.

\subsection{Cox rings and Mori dream spaces}\label{subsec:cox-ring}
In this subsection, we recall the concept of Cox rings and Mori dream spaces.

\begin{definition}
{\em 
Let $X$ be a normal algebraic variety with free finitely generated class group ${\rm Cl}(X)$.
We can define the {\em Cox ring} of $X$ to be 
\[
{\rm Cox}(X):= 
\bigoplus_{[D]\in {\rm Cl}(X)} 
\Gamma(X,\mathcal{O}_X(D)).
\] 
Here, the multiplication of sections is computed in the field of fractions of $X$.
We say that a normal algebraic variety $X$ is a {\em Mori dream space} (or {\em MDS} for short) if its Cox ring is finitely generated over $\cc$. In this case we denote the affine variety $\overline{X}:=\Spec\, {\rm Cox}(X)$ and call it the {\em total coordinate space} of $X$. We get $X$ back as a good quotient of the big open subset $\hat{X} \subseteq \overline{X}$.
This big open subset $\hat{X}$ is called 
the {\em characteristic space}.
The diagonalizable group (also called a {\em quasi-torus}) $H_X:=\Spec \cc[\Cl(X)]$ is called the {\em characteristic quasi-torus} of $X$.
}
\end{definition}

The name Mori dream space is given to these varieties because they behave optimally with respect to the minimal model program. 
For any divisor $D$ on a Mori dream space $X$, we can run a $D$-MMP that will terminate with either a Mori fiber space or a good minimal model for $D$ (i.e., a model on which its strict transform is a semiample divisor).
Toric varieties are known to be Mori dream spaces.
It is known that the Cox ring of a $n$-dimensional smooth projective toric variety of Picard rank $\rho$ is a polynomial ring in $\rho+n$ variables (see, e.g.,~\cite[Corollary 2.10]{HK00}).
Furthermore, in~\cite[Corollary 1.9]{BCHM10} it is proved that smooth Fano varieties are Mori dream spaces.

The quotient morphism $\hat{X} \xrightarrow{/H_X} X$ restricted to the preimage of the smooth locus of $X$ is a torsor, i.e. a principal $H_X$-bundle. Following~\cite[Definition 1.6.4.1]{ADHL15}, we say that the action of an affine algebraic group $G$ on a variety $Y$ is \emph{strongly stable}, if $Y$ allows an open subset $W$, such that
\begin{enumerate}
    \item  the complement $Y \setminus W$ is of codimension at least two in $Y$,
    \item $G$ acts freely on $W$,
    \item the orbit $G\cdot w$ is closed in $Y$ for every $w \in W$.
\end{enumerate}
In particular, we see that the action of $H_X$ on $\hat{X}$ is strongly stable~\cite[Section 1.6.4]{ADHL15}.

\subsection{Graded-local rings}\label{subsec:grocal-rings}
In this subsection, we recall the concept of gr-local rings and prove some preliminary results about their class groups.

Let $K$ be a finitely generated abelian group and $A$ be a $K$-algebra containing the field of complex numbers $\cc$.
In this subsection, we aim to find a suitable category in which certain generalizations of Cox rings fit in. 
Later on, we define the Cox ring for pair structures on projective varieties and quasi-cones, i.e. affine varieties with a $\cc^*$-action, such that all orbit closures meet in one distinguished point, the vertex. Obviously, quasi-cones share similarities with spectra of local rings. In particular, the Picard group is trivial.
Thus, it is natural to extend the definition of Cox rings to spectra of local rings. This will be done in the next section.
It turns out that the Cox rings of all these objects will be graded-local rings in the sense of~\cite[Definition 1.1.6]{GW78}.
These rings are graded by a finitely generated abelian group $K$, such that the set of graded ideals has a unique maximal element. This ideal does not need to be a maximal ideal in the usual sense. We will see that in our context, such unique maximal element is always a maximal ideal. The equivalent notion of ${\cc}^*$-local rings is considered in~\cite[Definition 1.5.13]{BH93}.

In~\cite[Theorem 2.5]{Hui12} it is proved that being a graded-local ring is equivalent to the degree zero part being a local ring in the classical sense. In particular, the graded maximal ideal is generated by all homogeneous non-units. As the graded maximal ideal is a maximal ideal, it is straightforward to see that in fact $\mathfrak{m}= \mathfrak{m}_0 \oplus \bigoplus_{k \neq 0} A_k$.  We will use the following definition, including the restriction that the ring is finitely generated over the degree zero part.

\begin{definition}{\em 
Let $K$ be a finitely generated abelian group. Let 
\[
A^{(K)}= \bigoplus_{k \in K} A^{(K)}_k
\]
be a $K$-graded noetherian integral domain. Then, we call $A^{(K)}$ a \emph{gr-local} ring if
\begin{enumerate}
    \item The set of graded ideals of $A^{(K)}$ has a unique maximal element $\mathfrak{m}$, which is a maximal ideal in the usual sense, and
    \item the degree-zero part $A^{(K)}_0$ is a local ring with maximal ideal $\mathfrak{m}_0=\mathfrak{m} \cap A^{(K)}_0$ and $A^{(K)}$ is finitely generated as an algebra over $A^{(K)}_0$.
\end{enumerate}}
\end{definition}

\begin{example}{\em 
Let $X$ be a quasi-cone. Then $A:=\mathcal{O}_X(X)$ is a gr-local ring with $A_0=\cc$.
If in addition $X$ is a Mori dream space, then the Cox ring ${\rm Cox}(X)$ has a $(\Cl(X)\times \zz)$-grading that endows it with the structure of a gr-local ring. Note that the ring $A=({\rm Cox}(X)^{(\Cl(X))})_0$ is a gr-local but not a local ring. 
On the other hand, if $X$ is a projective Mori dream space, then ${\rm Cox}(X)^{(\Cl(X))}$ is a gr-local ring. Indeed, in this case $\left({\rm Cox}(X)^{(\Cl(X))}\right)_0$ is just the ground field and thus a local ring.}
\end{example}

\begin{remark}{\em 
We remark that since $A$ is finitely generated over $A_0$, every homogeneous component $A_k$ is a finite $A_0$-module. This can be seen by taking a finite homogeneous set $\{f_i\}_{i \in I}$ of $A_0$-algebra generators of $A$. Then, there are only finitely many monomials in the $f_i$ that do not differ by a monomial lying in $A_0$. }
\end{remark}

Similar to the process of localizing at a prime ideal, we can graded-localize at a graded prime ideal. This process gives us a gr-local ring.

\begin{definition}{\em
Let $A$ be a $K$-graded ring and $\mathfrak{p}$ be a graded prime ideal. Let $S$ be the set of \emph{homogeneous} elements of $A \setminus \mathfrak{p}$. Then 
\[
A_{(\mathfrak{p})}:=S^{-1}A 
\]
is the \emph{gr-localization} of $A$ at $\mathfrak{p}$. It is a graded-local ring. It is not necessarily finitely generated over $\left(A_{(\mathfrak{p})}\right)_0$.
Furthermore, the unique maximal graded ideal is not necessarily maximal in the usual sense. 
}
\end{definition}

It is straightforward to see that the unique graded maximal ideal of $A_{(\mathfrak{p})}$ is maximal if and only if $\mathfrak{p}$ is maximal. 

\begin{example}
{\em
Consider $\cc[x,y]$ with the $\zz$-grading given by 
${\rm deg}(x)=1$ and ${\rm deg}(y)=-1$.
We consider gr-localizations at 
different graded prime ideals $\mathfrak{p}$.
We study them by understanding which scheme points of $\mathbb{A}^2$ define scheme points of $\Spec \cc[x,y]_{(\mathfrak{p})}$. 

When gr-localizing at $\mathfrak{p}=\langle  x,y \rangle$, all closed points on the coordinate axes of $\mathbb{A}^2$ define scheme points of $\Spec \cc[x,y]_{(\mathfrak{p})}$. 
Among the curves, the coordinate axes are the only closures of $\cc^*$-orbits that define points on the quotient.

If instead, we gr-localize at the coordinate axis $\langle  x \rangle$, which is not even graded maximal since it is contained in the graded $\langle x,y \rangle$, only the closed points (different from the origin) on this axis survive. The other axis together with its points vanishes, as all other closures of $\cc^*$-orbits do. Moreover, those curves that do not meet the axis $\langle  x \rangle$ become closed  points. 

If finally, we gr-localize at the orbit $\langle  xy-1 \rangle$, which is graded maximal but not maximal, the surviving closed points are exactly those that lie on this curve. 
}
\end{example}

We remark that localization at a prime $\mathfrak{p}$ factors through  gr-localization at $\mathfrak{p}$. In the following, we collect some useful properties that gr-local rings possess.
First, we note that our definition of gr-local contains the finite-generation property over the degree-zero part $A_0$. Hence, if $A_0$ is essentially of finite type over $\cc$, then so is $A$ (see, e.g.,~\cite[Proposition 1.3.9]{EGA4}).  

The second advantage of the finite-generation property is that the arguments of~\cite[Sec 1.2]{ADHL15} apply. Indeed, we have a surjection $A_0 \otimes \cc[x_1,\ldots,x_n] \twoheadrightarrow A$. In particular, we have an equivalence of categories between gr-local rings and affine schemes of finite type over the spectrum of a local ring with a quasi-torus action.

\begin{example}
{\em
We consider $\cc[x,y]$ with the $\zz$-grading given by the weight $(1,-1)$ and the maximal graded ideal $\mathfrak{m}=\langle x, y \rangle$. The degree zero part is $\cc[xy]$ and the set of homogeneous elements of $\cc[x,y] \setminus \mathfrak{m}$ is just $S:=\cc[xy] \setminus \langle xy \rangle$. On the other hand, graded localizing gives the gr-local ring $\cc[x,y]_{(\mathfrak{m})}$, which has degree-zero elements of the form $\frac{f}{g}$, where $f \in \cc[x,y]$, $g \in S$, and the degree of $f$ equals the degree of $g$. Since the elements of $S$ are of degree zero, both $f$ and $g$ are of degree zero. We conclude that there is an isomorphism 
\[
(\cc[x,y]_{(\mathfrak{m})})_0 \cong \cc[xy]_{\langle xy \rangle}.
\]
This means that we can see the local ring at the origin of the $\cc^*$-quotient $\cc=\Spec \cc[x,y]_0$ as a $\cc^*$-quotient of the gr-local ring $\cc[x,y]_{(\mathfrak{m})}$. Furthermore, this quotient parametrizes orbits in a special way: closed points parametrize closed orbits. Non-closed points parametrize orbits that are not closed, but their closure consists of closures of orbit points. In this, sense it is a good quotient.
}
\end{example}

In the following, we try to formalize this notion of good quotient. The classical one for varieties is not sharp enough here, since it does not stress what is parametrized by non-closed points. 

\begin{definition}{\em 
Let $X$ be an affine scheme over $\cc$. Let an affine algebraic group $G$ act on $X$. A morphism $\varphi \colon X \to Y$ to a scheme $Y$ over $\cc$ is a \emph{good quotient} if:
\begin{enumerate}
    \item $\varphi$ is $G$-invariant, surjective, and affine,
    \item for $U \subseteq Y$ open, the morphism $\mathcal{O}_Y(U) \to \mathcal{O}_X(\varphi^{-1}(U))$ is an isomorphism to $\mathcal{O}_X(\varphi^{-1}(U))^G$,
    \item the image of a $G$-invariant closed subset of $X$ is closed in $Y$, and 
    \item the images of disjoint $G$-invariant closed subsets are disjoint.
\end{enumerate}
Let $x \in X$ be a scheme point with closure $\overline{x}$. We denote by $Gx$ the orbit of $x$. We call the set 
\[
G\overline{x}:=\{g\cdot x' \mid g \in G \text{ and } x' \in \overline{x} \}
\]
the \emph{scheme-orbit} of $x$. 
We say that $\varphi$ is a \emph{good scheme quotient} if in addition to (1)-(4) the following hold:
\begin{enumerate}
    \item[(1')] The image of an orbit $Gx$ (not necessarily closed) with a \emph{closed} scheme-orbit $G\overline{x}$ is a point in $Y$,
     \item[(2')] this point is closed if and only if $Gx$ is closed,
     \item[(3')] if this point is not closed, then it's closure consists of the image $\varphi(G\overline{x})$ of the scheme-orbit, and 
     \item[(4')] in the pre-image of any $y \in Y$ lies exactly one orbit with a closed corresponding scheme-orbit. 
 \end{enumerate}
 We say that a good scheme quotient is a \emph{geometric scheme quotient}, if the pre-image of any $ y \in Y$ is an orbit with a closed corresponding scheme-orbit. 
 }
\end{definition}

In particular, we will see that for a gr-local ring $A$, the morphism $\Spec A \to \Spec A_0$ is a good scheme quotient, which is even geometric if the grading group is finite. We also remark that since we work in characteristic zero, categorical GIT-quotients by reductive groups are {\em universally categorical}, i.e., they behave well under base change, see~\cite[Theorem 1.1]{MFK94}.

\begin{lemma}
Let $K$ be a finite abelian group and $A$ be a $K$-graded gr-local ring. Then $A$ is local.
\end{lemma}

\begin{proof}
We have to show that the unique maximal and graded ideal $\mathfrak{m}$ is the only maximal ideal. Assume that there is another maximal ideal, which is not graded by the uniqueness property. Consider the corresponding closed point $x \in \Spec A$. Let $G \cong K$ be the finite abelian group acting on $\Spec A$. The orbit $Gx=G\overline{x}$ is closed, hence it's image $y'$ under the quotient morphism $\Spec A \to \Spec A_0$ is a closed point. Since $A_0$ is local, the point $y$ coincides with the point  corresponding to the maximal ideal $\mathfrak{m}_0=\mathfrak{m} \cap A_0$. But since $Gx$ is closed, it coincides with the orbit $x_{\mathfrak{m}}$.
This leads to a contradiction.
\end{proof}

In particular, we see that gr-local rings as defined above are also graded-local in the classical sense. Namely only allowing $K$ to be a free abelian group. Indeed, the degree-zero part with respect to the free part of the grading group is local as well, by the above lemma.

We finish this subsection proving that the class group of the spectrum of a gr-local ring is concentrated at the unique maximal graded ideal. In particular, the following holds.

\begin{lemma}
\label{lem:pic0}
Let $A$ be a gr-local ring, $X:=\Spec A$, and  $x\in X$ be the closed point corresponding to the unique graded maximal ideal $\mathfrak{m}$. Then
\[
\Cl(X)\cong \Cl(X,x) \cong \Cl(X_x)
\quad
{\rm and}
\quad
\Pic(X)\cong \Pic(X_x) \cong 0.
\]
\end{lemma}

\begin{proof}
By~\cite[Prop. 7.1]{Sam64}, the class group $\Cl(X)$ is isomorphic to the group of graded divisorial ideals modulo the subgroup of principal graded ideals. Thus, we only have to show that for an ideal $I \subseteq A$, if $IA_\mathfrak{m}$ is principal, then $I$ is already principal. Let $a \in A_{\mathfrak{m}}$ be a generator of $IA_\mathfrak{m}$, which we can assume to be a graded element of $A$. Then, for a graded $x \in I$, there are $p \in A$ and $q \in S=A \setminus \mathfrak{m}$, such that $x=\frac{p}{q}a$ in $A_{\mathfrak{m}}$. So $xq=pa$ holds in $A$. Writing $q_i$  for the graded components of $q$, we know that $q_0$ is nonzero and a unit. Thus, there is a homogeneous component $p_k$ of $p$, such that $x q_0=p_k a$. Hence, $I$ is generated by $a$ in $A$. 
 The argument for triviality of the Picard group is the same as in~\cite[Lemma 5.1]{Mur69}.
\end{proof}

The advantage of the notion of gr-local rings is that not only it will encompass the local Cox rings of singularities, but also it stresses the grading. Note that this provides us with a meaningful notion of finite generation for Cox rings of (spectra of) local rings: the Cox ring should be finitely generated (as an algebra) over the local ring itself.

\begin{remark}{\em 
\label{rem:nongrocalMDS}
If $X$ is an affine Mori dream space, then 
the Cox ring ${\rm Cox}(X)$ may have no grading that makes it a gr-local ring. However, we will see later that the Cox sheaf $\Cox{X}$ is always a sheaf of gr-local rings (see Definition~\ref{def:sheaf-gr-local-rings}).}
\end{remark}

\subsection{Graded-Henselian rings}\label{subsec:gr-Henselian-rings}

In this subsection, we recall the concept of
gr-Henselian rings and prove some preliminary results about their class groups.

The local rings in the Zariski topology are too coarse to capture the local topology at a singularity well. Thus in the following, we consider also the local rings in the \'etale topology, which are Henselian local rings. The resulting Cox rings are gr-local rings with a Henselian degree-zero part. Such rings were studied in~\cite{Cae83} and are called \emph{gr-Henselian} rings.

\begin{definition}[Cf.~\cite{Cae83}]
{\em 
Let $A$ be a gr-local ring. Then, we say that $A$ is \emph{gr-Henselian}, if it satisfies one of the following equivalent conditions:
\begin{enumerate}
    \item $A_0$ is Henselian, and
    \item every graded $A$-algebra is a direct sum of gr-local rings.
\end{enumerate}
}
\end{definition}

In fact,~\cite[Teorem 4.6]{Cae83} contains several more equivalent characterizations analogous to those for Henselian rings. For us, maybe the most important property of gr-Henselian rings is the following.

\begin{theorem}
\label{thm:Cl-gr-Hens}
Let $A$ be an excellent rational $\zz^k$-graded gr-Henselian ring. Let $A_\mathfrak{m}^h$ be the Henselization of the local ring at the unique maximal graded ideal
$\mathfrak{m}$ and $\hat{A}$ be the $\mathfrak{m}$-adic completion. Assume the graded prime ideals $\mathfrak{p}$ of height one in $A$ are in one-to-one-correspondence with the height one prime ideals in $A_0$ via $\mathfrak{p}=\mathfrak{p}_0A$.
Furthermore, we assume the same property holds for the base change $\tilde{A}:=A\otimes_{A_0} \widehat{A_0}$. Then, we have isomorphisms
\[
\Cl(A) \cong \Cl(A_\mathfrak{m}^h) \cong \Cl(\hat{A}).
\]
\end{theorem}

The assumptions on the height one prime ideals are not as restrictive as they may seem. They are   fulfilled for gr-local rings if the morphisms $\Spec(A) \to \Spec(A_0)$ and $\Spec(\tilde{A}) \to \Spec(\widehat{A_0})$ are locally trivial fiber bundles in codimension one.

To prove the theorem, we follow the line of arguments of~\cite[Sec 1 \& 2]{Fl81}, where an analogous result is proved for $\nn$-graded rational rings. Before we can use the results from~\cite{Fl81}, we have to prove the following lemma.

\begin{lemma}
\label{le:componentcompletion}
Let $A$ be a $\zz^K$-graded gr-local ring with maximal graded ideal $\mathfrak{m}=\mathfrak{m}_0 \oplus \bigoplus_{k \neq 0} A_k$. Then the degree $k$ piece of the $\mathfrak{m}$-adic completion $\hat{A}=\prod_{k \in \zz^K} \underleftarrow{\lim}\, A_k/\left(\mathfrak{m}^j\right)_k$ is isomorphic to the $\mathfrak{m}_0$-adic completion of the $A_0$-module $A_k$.
This means that we have isomorphisms
\[
\underleftarrow{\lim}\, A_k/\left(\mathfrak{m}^j\right)_k \cong \underleftarrow{\lim}\, A_k/\left(\mathfrak{m}_0\right)^j\!A_k =: \widehat{A_k}.
\]
\end{lemma}

\begin{proof}
Let $g_{01},\ldots,g_{0n_0},g_{i_11},\ldots, g_{i_1n_{i_1}},\ldots,g_{i_m1},\ldots, g_{i_mn_{i_m}}$ be a finite set of $A$-module generators of $\mathfrak{m}$, where $g_{ij} \in A_i$ for $1\leq j \leq n_i$.
In the following, by degree, we mean the standard degree of a monomial $m(x_{ij})$. Otherwise, we speak of the $\zz^K$-degree. 
Since $\mathfrak{m}^l$ is generated as an $A$-module  by all monomials in the $g_{ij}$ of degree $l$, we know that $(\mathfrak{m}^l)_0$ is generated as an $A_0$-module by all monomials in the $g_{ij}$ of $\zz^K$-degree zero and degree at least $l$. Indeed, a monomial of degree $l$ and nonzero $\zz^K$-degree $k$ has an $A$-coefficient in $A_{-k}$ in order to lie in $A_0$, and expanding it in the $A_0$-module generators of $A_k$ leads to monomials of degree greater than $l$.
On the other hand, we know that there are only finitely many monomials $m_1,\ldots,m_M$ in the $g_{ij}$ such that any other monomial in the $g_{ij}$ is in turn a monomial in these . This follows from standard monomial combinatorics. We set $\mu:=\max(\deg(m_i))_{i=1,\ldots,M}$ and get
$$
(\mathfrak{m}_0)^{\nu \mu} \subseteq \left(\mathfrak{m}^{\nu \mu}\right)_0 \subseteq (\mathfrak{m}_0)^{\mu}
$$
for $\nu \geq 1$. So the claim follows for $k=0$. The argument for the $A_0$-modules $A_k$ is similar. There are only finitely many monomials in the $g_{ij}$ of $\zz^K$-degree $k$, such that all others of $\zz^K$-degree $k$ differ from them by multiplication with a monomial of $\zz^K$-degree $0$. Set $\mu_k$ to be the maximal degree of these finitely many monomials. Then we get
$$
(\mathfrak{m}_0)^{\nu \mu + \mu_k-1} A_k \subseteq \left(\mathfrak{m}^{\nu \mu + \mu_k}\right)_0 \subseteq (\mathfrak{m}_0)^{\mu} A_k
$$
and the claim is proved.
\end{proof}

We need two additional lemmas and use the following definitions. Let $B$ be the $\mathfrak{m}$-adic completion of $A[t_1,t_1^{-1},\cdots,t_k,t_k^{-1}]$. Denote $\hat{A}\llbracket \mathbf{x} \rrbracket:=\hat{A}\llbracket x_1,\ldots,x_k\rrbracket$ and let $p \colon \hat{A} \to \hat{A}\llbracket \mathbf{x} \rrbracket$ and $p_0 \colon \hat{A}\to B$ be the canonical injections. Let $q \colon \hat{A} \to \hat{A}\llbracket  \mathbf{x} \rrbracket$ be the homomorphism defined by
\[
A_{(z_1,\ldots,z_k)} \ni f \mapsto f\cdot (x_1+1)^{z_1}\cdots (x_k+1)^{z_k},
\]
where $(x_i+1)^{-1}=\sum_{j=1}^{\infty} (-x_i)^j$. Further, let $q_0 \colon \hat{A}\to B$ be the homomorphism defined by
\[
A_{(z_1,\ldots,z_k)} \ni f \mapsto f\cdot t_1^{z_1}\cdots t_k^{z_k},
\]
and $g \colon B \to \hat{A}\llbracket  \mathbf{x} \rrbracket$ the $\hat{A}$-homomorphism mapping $t_i$ to $x_i+1$. Observe that the equalities $g \circ p_0=p$ and $g \circ q_0=q$ hold. We prove the following lemma.

\begin{lemma}
\label{le:cl-injective}
The map $g_* \colon \Cl(B) \to \Cl(\hat{A}\llbracket  \mathbf{x} \rrbracket)$ is injective.
 \end{lemma}
 
 \begin{proof}
 Let $\mathfrak{b}$ be a divisorial ideal of $B$ in the kernel of $g_*$. We note that  $\hat{A}\llbracket  \mathbf{x} \rrbracket$ is the $\langle t_1-1,\ldots,t_k-1\rangle$-adic completion of $B$. For any prime $\mathfrak{p} \subseteq \hat{A}$ the ring extension $B_{\mathfrak{p}B} \to \hat{A}\llbracket  \mathbf{x} \rrbracket_{\mathfrak{p}\hat{A}\llbracket  \mathbf{x} \rrbracket}$ is faithfully flat.
 Thus due to principality of $\mathfrak{b}\otimes_B \mathfrak{p}\hat{A}\llbracket  \mathbf{x} \rrbracket$, also $\mathfrak{b}\cdot B_{\mathfrak{p}B}$ is principal for any prime $\mathfrak{p} \subseteq \hat{A}$.
 
 We want to show that  $\mathfrak{b}$ is locally principal, i.e. $\mathfrak{b}_\mathfrak{p}$ is principal for any prime $\mathfrak{p}$. So let $\mathfrak{n} \subseteq B$ be maximal and $\mathfrak{m}=\hat{A} \cap \mathfrak{n}$ be the unique maximal ideal of $\hat{A}$. Now, we have an isomorphism $\hat{A}/\mathfrak{m} \cong B/\mathfrak{n}$ and the local homomorphism $\hat{A}\to B_{\mathfrak{n}}$ of local rings is formally smooth. Then~\cite[II, Corollaire 9.8]{Bou78} implies that $B_\mathfrak{q}$ is parafactorial for any prime $\mathfrak{q} \subseteq \mathfrak{n}$ with $\mathfrak{q} \not\subseteq \mathfrak{m}B$ and $\dim(B_\mathfrak{q})\geq 2$.  Due to normality of $B_\mathfrak{n}$ and by induction on $\dim(B_\mathfrak{q})$, we get that $\mathfrak{b}_\mathfrak{q}$ is principal. Thus $\mathfrak{b}$ is locally principal. But since $B$ is $\mathfrak{m}B$-adically complete and $A/\mathfrak{m}$ is a field, we get
 $$
 \Pic(B) \cong \Pic(B/\mathfrak{m}B) \cong \Pic((A/\mathfrak{m})[t_1,t_1^{-1},\ldots,t_k,t_k^{-1}]) \cong 0.
 $$
So $\mathfrak{b}$ is principal.
 
 \end{proof}

\begin{lemma}
\label{le:cl-equalizer}
Under the assumptions of Theorem~\ref{thm:Cl-gr-Hens}. The sequence
\[
\xymatrix{
0 \ar[r] & \Cl(A) \ar[r] & \Cl(\hat{A}) \ar@<-.5ex>[r]_{p_{0*}} \ar@<.5ex>[r]^{q_{0*}} & \Cl(B)
}
\]
is exact.
 \end{lemma}
 
\begin{proof}
We define
\[
\tilde{A}:= A \otimes_{A_0} \widehat{A_0} = \bigoplus_{k \in \zz^K} \left( A_k \otimes_{A_0} \widehat{A_0} \right) = \bigoplus_{k \in \zz^K} \widehat{A_k}
\cong \bigoplus_{k \in \zz^K}  \underleftarrow{\lim}\, A_k/\left(\mathfrak{m}^j\right)_k ,
\]
where hats denote $\mathfrak{m}_0$-adic completion, the second identity is due to the fact that the $A_k$ are finitely generated $A_0$-modules and the isomorphy is due to Lemma~\ref{le:componentcompletion}.
The $A_0$-algebra-homomorphism $A \to \hat{A}$ factors through $A \to \tilde{A}$ and $\tilde{A} \to \hat{A}$.
By~\cite[Sec 2]{Fl81}, it follows that for any height one prime ideal $\bar{\mathfrak{a}}$ of $\hat{A}$ such that $q_0^*(\bar{\mathfrak{a}})=p_0^*(\bar{\mathfrak{a}})$ in $\Cl(B)$, there is a graded height one ideal $\tilde{\mathfrak{a}} \subseteq \tilde{A}$ such that $\hat{\tilde{\mathfrak{a}}}=\bar{\mathfrak{a}}$.
Then, $\tilde{\mathfrak{a}}_0$ is an ideal of $\widehat{A_0}$ of height one. In particular,  $\tilde{\mathfrak{a}}_0 \tilde{A}=\tilde{\mathfrak{a}}$ by the assumptions of the theorem. Since $A_0$ is rational, by~\cite[Theorem (6.2)]{BF84}, there is a height one prime ideal $\mathfrak{a}_0$ of $A_0$, such that $\widehat{\mathfrak{a}_0}=\tilde{\mathfrak{a}}_0$. Then  $\mathfrak{a}_0A$ is a height one prime of $A$ such that $\mathfrak{a}_0A \otimes_{A_0} \widehat{A_0}=\tilde{\mathfrak{a}}$. So the equalizer of $q_0^*$ and $p_0^*$ indeed equals the image of $\Cl(A)$ in $\Cl(\hat{A})$.
This concludes the proof of the lemma.
\end{proof} 

\begin{proof}[Proof of Theorem ~\ref{thm:Cl-gr-Hens}]
Since $A$ is excellent and rational, the completion $\hat{A}$ has the DCG property. This means that $\pi_*\colon \Cl(\hat{A}\llbracket x \rrbracket) \to \Cl(\hat{A})$ induced by $\pi \colon \hat{A}\llbracket x \rrbracket \to \hat{A}$ mapping $x$ to $0$ is a bijection, see~\cite[p. 128]{Fl81}. But then $\hat{A}\llbracket x_1,\ldots,x_k \rrbracket$ has the DCG property. Moreover, $\omega \colon \hat{A}\llbracket x_1,\ldots,x_k \rrbracket \to  \hat{A}$ mapping all the $x_i$ to $0$ induces a bijection $\omega_*$ between the divisor class groups by induction. But since $\omega \circ p = \omega \circ q$, $p=g \circ p_0$, and $q= g \circ q_0$, by Lemma~\ref{le:cl-injective} and Lemma~\ref{le:cl-equalizer}, we get $p_{0*}=q_{0*}$ and $\Cl(A) \to \Cl(\hat{A})$ is surjective and hence bijective. Since this map factors through $\Cl(A) \to \Cl(A_\mathfrak{m}^h)$, which is injective, bijectivity of the three divisor class groups follows as claimed.
\end{proof}

\subsection{Sheaves of gr-local rings}
\label{subsec:shvs-gr-local}
In this subsection, we define sheaves of gr-local rings on algebraic varieties.

Throughout this subsection, we consider the case that $X$ is only locally a Mori dream space, that is, its Cox sheaf $\Cox{X}$ is locally of finite type in the sense of~\cite[Constr. 1.3.2.1]{ADHL15}. This means that every $x \in X$ has an open affine neighbourhood $U$, such that $\Cox{X}(U)$ is a finitely generated $\cc$-algebra. This makes it possible to define the relative spectrum $\widehat{X} := \Spec_X \Cox{X}$ of the Cox sheaf, the so-called characteristic space of $X$. However, it may happen that the ring of global sections ${\rm Cox}(X)$ is not a finitely generated $\cc$-algebra. 

In this subsection, we define \emph{sheaves of gr-local rings} and we show that a Cox sheaf is locally of finite type in the aforementioned sense if and only if it is a sheaf of gr-local rings. This means, in particular, that this property has to be checked only locally at the singularities whenever the divisor class group $\Cl(X)$ is finitely generated. 

\begin{definition}\label{def:sheaf-gr-local-rings}{\em 
Let $X$ be a normal variety. Let $\mathcal{S}$ be a quasi-coherent sheaf of $\mathcal{O}_X$-modules. If the stalk $\mathcal{S}_x$ of $\mathcal{S}$ at any point of $X$ is a gr-local ring, then we call $\mathcal{S}$ a \emph{sheaf of gr-local rings}.}
\end{definition}

We recall from~\cite[Def 1.3.1.1]{ADHL15}, that the \emph{sheaf of divisorial algebras} associated to a finitely generated subgroup $K \subseteq \WDiv(X)$ is the quasi-coherent sheaf \[
\mathcal{S}:= 
\bigoplus_{D \in K} 
\mathcal{O}_X(D).
\]

\begin{definition}{\em
Let $X$ be a normal variety. Let $\mathcal{S}$ be a quasi-coherent sheaf of $\mathcal{O}_X$-modules. We say that $\mathcal{S}$ is {\em locally of finite type} if for every point $x\in X$ there is an open affine neighborhood $x\in U$ with
$\mathcal{S}(U)$ a finitely generated $\cc$-algebra.
}
\end{definition}

\begin{lemma}
\label{le:fg-local}
Let $X$ be a normal algebraic variety and $\mathcal{S}$ a sheaf of divisorial algebras associated to the finitely generated subgroup $K \subseteq \WDiv(X)$.
Then the stalk $\mathcal{S}_x$ is a finitely generated
$\mathcal{O}_{X,x}$-algebra for any $x \in X$ if and only if $\mathcal{S}$ is locally of finite type.
\end{lemma}

\begin{proof}
First let $\mathcal{S}$ be a sheaf of divisorial algebras locally of finite type and $U \subseteq X$ be affine. It follows that for some $k \in \mathbb{N}$, there is a surjection $\mathcal{O}_X(U)[x_1,\ldots,x_k] \to \mathcal{S}(U)$.
Since surjectivity of $R$-modules is local, this induces a surjection $\mathcal{O}_{X,x}[x_1,\ldots,x_k] \to \mathcal{S}_x$ and thus $\mathcal{S}_x$ is a finitely generated
$\mathcal{O}_{X,x}$-algebra for every $x \in U$.

Now, fix $x \in X$ and assume that $\mathcal{S}_x$ is a finitely generated
$\mathcal{O}_{X,x}$-algebra. We fix a set of generators $D_1,\ldots,D_m$ of $K$. 
Then, there is an open affine neighbourhood $x\in U \subseteq X$, such that $x$ lies in every irreducible component of $D_i\cap U$ for any $i$. 
Let $f_1,\ldots,f_k \in \mathcal{S}_x$ be a finite set of $K$-homogeneous $\mathcal{O}_{X,x}$-algebra-generators of the stalk $\mathcal{S}_x$. 
    By shrinking $U$ if necessary , we can lift these germs to sections  $f_1,\ldots,f_k \in \mathcal{S}(U)$, such that $x$ lies in every irreducible component of $\supp(f_i)$ for any $i$. 

Now let $D \in K$. We have a primary decomposition of the divisorial ideal $\mathcal{S}_D(U)=\mathfrak{q}_1 \cdots \mathfrak{q}_r$, such that the associated primes $\mathfrak{p}_i$ all lie in $\mathfrak{m}_x$. In particular, $\mathrm{sat}_{\mathfrak{m}_x}(\mathcal{S}_D(U))=\mathcal{S}_D(U)$, see e.g.~\cite[Prop. 4.9]{AM69}.
The localization $S_{x,D}$ of $\mathcal{S}_D(U)$ is generated as an $\mathcal{O}_{X,x}$-module by monomials $p_1,\ldots,p_l$ in the $f_i$. In particular, $x$ lies in every irreducible component of $\supp(p_i)$ for any $i$. So the $\mathcal{O}(U)$-module $J:=\sum_{i=1}^{l} \mathcal{O}(U) p_i$ has localization $S_{x,D}=\sum_{i=1}^{l} \mathcal{O}_{U,x} p_i$ and saturation $\mathrm{sat}_{\mathfrak{m}_x}(J)=J$. Thus $J=\mathcal{S}_D(U)$ and $\mathrm{S}(U)$ is generated as an $\mathcal{O}(U)$-algebra by the $f_i$. The proof is finished.  
\end{proof}

\begin{corollary}
\label{cor:sheaf-groc-sheaf-finite-type}
Let $X$ be a normal algebraic variety such that $\Cl(X)$ is finitely generated. Then $\Cox{X}$ is a sheaf of gr-local rings if and only if it is locally of finite type.
\end{corollary}

\begin{example}{\em 
If $X$ is a point, then a sheaf of gr-local rings over $X$ is a gr-local ring.}
\end{example}

\subsection{Coverings of gr-local rings}\label{subsec:covers-grocal-rings}
In this subsection, we bring together the concepts of fundamental group and Cox rings.

There are different notions for the regional fundamental groups of singularities. In the case of a klt singularity, they all agree. Let $ x \in (X,\Delta)$ be a klt singularity. Then the regional fundamental group $\pi_1^{\reg} (X,\Delta;x)$ is the inverse limit of the orbifold fundamental groups $\pi_1^{\reg} (U_{\reg},\Delta;x)$, where $U$ runs through analytic open neighborhoods of $x$. The regional fundamental group is computed by some neighborhood $U$, that can be chosen to be the intersection of $X$ with a small euclidean ball around $x$ in some complex manifold $M \supseteq X$. It equals the fundamental group of the regional link of $x$, the intersection of $X_{\reg}$ with a small euclidean sphere, which is just a deformation retract of $U_{\reg}$.

However, when we work in the algebraic category, we deal with \'etale neighborhoods of local rings. In the case of klt singularities, this makes no difference. This follows from the fact that the regional fundamental group is finite by~\cite[Theorem 1]{Bra20}. In particular,  $\pi_1^{\reg}(X,\Delta;x)$ equals the \'etale fundamental group of the smooth locus of the spectrum of the holomorphic local ring $\mathcal{O}_{X,x}^{\rm hol}$. Since this ring is Henselian, by~\cite[Cor. p. 579]{Elk73}, we have 
\[
\pi_1^{\reg} (X,\Delta;x) \cong \pi_1^{\rm et}(X^h_{x,\reg}, \Delta^h_{\rm reg}) \cong \pi_1^{\rm  et}(\widehat{X}_{x,\reg}, \widehat{\Delta}_{\rm reg}).
\]
Here, $X_x^h$ and $\widehat{X_x}$ denote the spectra of the \'etale and complete local rings. The subscript (or supscrit) reg, means that we consider the regular locus.
Furthermore, the divisor 
$\Delta^h_{\rm reg}$ (resp. $\widehat{\Delta}_{\rm reg}$) is the pull-back of $\Delta$ to $X_{x,{\rm reg}}^h$ 
(resp. $\widehat{X}_{x,{\rm reg}}$).
Since $\pi_1^{\reg} (X,\Delta;x)$ is finite, it is computed by an affine \'etale neighborhood $V_x \to X$ of $x$. 
Moreover, by~\cite[Sec 6]{BF84}, we know that $\Cl(X_x^h)$ and  $\Cl(\widehat{X_x})$ are finitely generated and isomorphic. Thus, we can find an  affine \'etale neighborhood $U_x \to X$ of $x$ that computes both the regional fundamental and the local divisor class group. We will use these facts often throughout the article.

\subsection{Minimal Model Program}\label{subsec:mmp}
In this subsection, we recall the definition of the singularities of the minimal model program.
We also recall some basic constructions as the purely log terminal blow-up.

\begin{definition}
{\em 
A projective morphism $f\colon X\rightarrow Z$ is called a {\em contraction} if $f_*\mathcal{O}_X=\mathcal{O}_Z$.
In particular,
if $X$ is normal and
$X\rightarrow Z$ is a contraction, then
$Z$ is normal as well.
}
\end{definition}

On the other hand, if $g\colon X \to Z$ is an affine morphism, then $X$ is isomorphic to the relative spectrum over $Z$ of the direct image sheaf $f_* \mathcal{O}_X$, i.e., $ X \cong \Spec_Z f_* \mathcal{O}_X$. So if $h= f \circ g \colon X \to Y \to Z$ is an affine morphism $g$ composed with a contraction $f$, then $h_*\mathcal{O}_X$ is a quasi-coherent sheaf of $\mathcal{O}_Z$-modules that is locally of finite type. Morphisms of this kind will become important in the following.

\begin{definition}
\label{def:aff-contraction}
{\em 
A morphism $h\colon X \to Z$ that factors through an affine morphism $g\colon X \to Y$ and a contraction $f \colon Y \to Z$ is called an {\em aff-contraction}.
}
\end{definition}

\begin{example}
{\em
Let $X$ be a projective Mori dream space with structure morphism $\phi \colon X \to {\rm Spec}(\cc)$. Let $\psi \colon \hat{X}=\Spec_X \Cox{X} \to X$ be its characteristic space. Then $h:= \phi \circ \psi$ is an aff-contraction.
}
\end{example}

\begin{definition}
{\em 
Let $X$ be a normal quasi-projective variety.
A {\em log pair} $(X,\Delta)$ consists of $X$ and an effective divisor $\Delta\geq 0$ so that
$K_X+\Delta$ is a $\qq$-Cartier $\qq$-divisor.
}
\end{definition}

\begin{definition}
{\em
Let $(X,\Delta)$ be a log pair.
A {\em prime divisor over $X$} is a prime divisor on a normal quasi-projective variety $Y$ that admits a projective birational morphism to $X$.
This means that there exists a projective birational morphism
$\pi\colon Y\rightarrow X$ and $E\subset Y$ a prime divisor.
The {\em log discrepancy} of $(X,\Delta)$ at $E$ is defined to be
\[
a_E(X,\Delta) :=
1+{\rm coeff}_E(K_Y-\pi^*(K_X+\Delta)). 
\]
A {\rm log resolution} of a log pair $(X,\Delta)$ is a projective birational morphism
$\pi\colon Y\rightarrow X$ so that $Y$ is a regular variety, the exceptional divisor $E$ is purely divisorial, 
and $E_{\rm red}+\pi^{-1}_*\Delta$ has simple normal crossing support.
Any log pair admits a log resolution by Hironaka's resolution of singularities.
}
\end{definition}

\begin{definition}
{\em 
A pair $(X,\Delta)$ is said to be 
{\em Kawamata log terminal} (or {\em klt} for short) if all its log discrepancies are positive.
This means that $a_E(X,\Delta)>0$ for every prime divisor $E$ over $X$.
A pair $(X,\Delta)$ is said to be {\em log canonical} (or {\em lc} for short) if all its log discrepancies are non-negative.
This means that $a_E(X,\Delta)\geq 0$ for every prime divisor $E$ over $X$.
In both cases, it suffices to check all the prime divisors which appear on an arbitrary log resolution of the pair.
A {\em non-klt} center of a pair $(X,\Delta)$ is the image on $X$ of a prime divisor $E$ over $X$ for which $a_E(X,\Delta)\leq 0$.
In particular, if $(X,\Delta)$ is a log canonical pair, a non-klt center is the image on $X$ of a divisor with log discrepancy zero.
}
\end{definition}

\begin{definition}
{\em
A variety $X$ is said to be {\em klt type}
if there exists a boundary $\Delta$ so that $(X,\Delta)$ is a klt pair.
Analogously, we say that a germ $(X,x)$ is {\em klt type} if there exists $\Delta$ through $x$ so that $(X,\Delta)$ is a klt germ.
}
\end{definition}

\begin{definition}
{\em 
A pair $(X,\Delta)$ is called {\em divisorially log terminal} or {\em dlt} if there exists an open subset $U\subset X$ satisfying the following conditions:
\begin{enumerate}
    \item $U$ is smooth and $\Delta|_U$ has simple normal crossing support, 
    \item the coefficients of $\Delta$ are at most one, 
    \item all the non-klt centers of $(X,\Delta)$ intersect $U$ and are given by strata of the divisor $\lfloor \Delta \rfloor$.
\end{enumerate}
A pair $(X,\Delta)$ is said to be {\em purely log terminal} or {\em plt} if it is dlt and it has at most one non-klt center.
}
\end{definition}

\begin{definition}{\em 
Let $X\rightarrow Z$ be a contraction 
and $(X,\Delta)$ be a log pair.
We say that $(X,\Delta)$ is of {\em Fano type} over $Z$
if there exists a boundary
$\Delta'$ on $X$ that is big over $Z$, so that $(X,\Delta+\Delta')$ is klt and 
$K_X+\Delta+\Delta' \sim_{\qq,Z}0$.
}
\end{definition}

\begin{definition}
{\em 
Let $(X,\Delta;x)$ be a klt singularity.
A {\em purely log terminal blow-up} of $(X,\Delta)$ at $x$ (or a {\em plt blow-up} for short) is a projective birational morphism
$\pi\colon Y\rightarrow X$ satisfying the following conditions:
\begin{enumerate}
    \item $\pi$ is an isomorphism on the complement of $x$, 
    \item the pre-image of $x$ on $Y$ is a unique prime divisor $E$,
    \item the pair $(Y,E+\Delta_Y)$ is plt, where $\Delta_Y:=\pi_*^{-1}(\Delta)$, and 
    \item $-E$ is ample over $X$.
\end{enumerate}
In particular, the log pair $(E,\Delta_E)$ obtained by adjunction of $(Y,E+\Delta_Y)$ to $E$ is 
of Fano type.
}
\end{definition}

In this article, we will be concerned with orbifold structures on Fano type varieties and klt singularities. 
Therefore, we will need the following definitios.

\begin{definition}
\label{def:standard-approx}
{\em 
We say that the coefficients of $\Delta$ are {\em standard} if they have the form
$1-\frac{1}{n}$, where $n$ is a positive integer.
Given a boundary $\Delta$ on $X$, we define its {\em standard approximation} to be the effective divisor $\Delta_s$ on $X$ with largest standard coefficients such that $\Delta\geq \Delta_s$.
Note that if $(X,\Delta)$ is of Fano type over $Z$, then $(X,\Delta_s)$ is of Fano type over $Z$ as well.}
\end{definition}

The following is the definition of one of the main kind of covers that we will consider in this article.

\begin{definition}
{\em 
Let $(X,\Delta;x)$ be a klt singularity.
We say that $\phi\colon Y\rightarrow X$ is a {\em finite Galois quasi-\'etale cover} if the following conditions are satisfied:
\begin{enumerate}
    \item There exists a finite group $G$ acting on $Y$,
    \item $X$ is the quotient of $Y$ by $G$, and
    \item the pull-back of $K_X+\Delta$ equals $K_Y+\Delta_Y$, where $\Delta_Y$ is effective.
\end{enumerate}
Note that $Y\rightarrow X$ may not be unramified in codimension one.
However, it is unramified in codimension one when considering $(X^{\rm reg},\Delta^{\rm reg})$ as an orbifold.
This justifies the quasi-\'etale property.
We say that $(Y,y)$ is a {pointed finite Galois quasi-\'etale cover} of $(X,\Delta;x)$ if
$y\in Y$ is a point whose image on $X$ is $x$.
To shorten the notation,
we may say that $Y\rightarrow X$ is a {\em pointed finite cover} of $(X,\Delta;x)$.
}
\end{definition}

\section{Generalized Cox rings}
\label{sec:gen-cox}

In this section, we generalize the concept of Cox rings to different settings and prove some basic properties.
In subsection~\ref{subsec:proj-log-cox}, we will define the Cox ring of a log pair and study its properties.
In subsection~\ref{subsec:rel-log-cox}  and subsection~\ref{subsec:local-cox}, we introduce the relative Cox ring
and the local Cox ring, respectively.
In subsection~\ref{subsec:prop-gen-cox}, we prove some properties of the above generalizations. For instance, we prove that the Cox ring of a relatively Fano type variety admits the structure of a klt type singularity (Theorem~\ref{thm:klt-Cox-ring}).
Finally, in subsection~\ref{subsec:local-Hensel-cox}, we will define the local Henselian Cox ring of a singularity.
This is one of the main objects considered in this article.

\subsection{The Cox ring of a log pair}\label{subsec:proj-log-cox}

In this subsection, we generalize the Cox ring and related notions to log pairs $(X,\Delta)$, where $X$ is a normal algebraic variety and $\Delta$ is an effective divisor on $X$.
We prove some basic properties of the Cox ring of a log pair 
and show that these objects become interesting even for log pair structures on $\pp^1$.

In the case that $\Delta$ has standard coefficients, such pairs can be viewed as geometric orbifolds in the sense of Campana~\cite{Cam11}.
We proceed to define the class group ${\rm Cl}(X,\Delta)$ of a log pair $(X,\Delta)$. 
We will define the class group ${\rm Cl}(X,\Delta)$ to be ${\rm Cl}(X,\Delta_s)$, where $\Delta_s$ is the standard approximation.
Hence, it suffices to define the class group for standard pairs. 

\begin{definition}{\em 
Let $(X,\Delta)$ be a log pair. We denote by $(X^{\rm reg},\Delta^{\rm reg})$ the orbifold smooth locus.
As usual, the divisor $\Delta^{\rm reg}$ denotes the restriction of $\Delta$ to $X^{\rm reg}$.
Then, we have canonical orbifold charts
$V:=\mathbb{A}^n\rightarrow \mathbb{A}^n\cong U\subset X^{\rm reg}$.
These charts are quotients by abelian reflection groups,
ramifying over $\Delta|_U$ for an analytic neighborhood $U$ of any point $x\in X^{\rm reg}$.
For such smooth orbifolds, there is a notion of orbifold Weil divisors and orbifold Picard group ${\rm Pic}^{\rm orb}$ (see, e.g.,~\cite[Sec 4.4.3]{BG08}).
Then, we can define 
$$
\Cl(X,\Delta):=\Pic^{\rm orb} \left(X^{\reg},\Delta^{\reg}\right).
$$
The group ${\rm Cl}(X,\Delta)$ is essentially the group of orbifold Weil divisors ${\rm WDiv}(X,\Delta)$ quotient by linear equivalence of $\qq$-divisors.
}
\end{definition}

Now, we define the sheaves of sections $\mathcal{O}_{(X,\Delta)} (D)$ for orbifold Weil divisors $D$ on pairs $(X,\Delta)$.

\begin{definition}
{\em 
Let $(X,\Delta)$ be a log pair and let $D \in \WDiv(X,\Delta)$. Then, we define the sheaf $\mathcal{O}_{(X,\Delta)}(D)$ by 
\[
\Gamma(U,\mathcal{O}_{(X,\Delta)}(D)) := \langle D' \in \WDiv^{\mathrm{eff}}(U,\Delta) \mid D- D' \in \PDiv(U) \rangle
\]
for any open $U \subseteq X$. In particular,  $\mathcal{O}_{(X,\Delta)}(D)$ is a coherent sheaf of $\mathcal{O}_{X}$-modules for any $D \in \WDiv(X,\Delta)$.
When $X=U$, we may write
$\Gamma(X,\Delta,\mathcal{O}_X(D))$
or simply
$\Gamma(X,\Delta,D)$.
}
\end{definition}

Proceeding as in~\cite[Sec 3.1,3.2]{ADHL15}, we first define the sheaves of divisorial algebras for  subgroups $K \subseteq \Cl(X)$, before defining Cox sheaves and Cox rings.

\begin{definition}{\em 
Let $(X,\Delta)$ be a log pair. Let $N \subseteq \WDiv(X,\Delta)$ be a subgroup. Then the sheaf of divisorial algebras associated to $N$ is
\[
\mathcal{S}^{(N)}:=\bigoplus_{D \in N} \mathcal{S}^{(N)}_D
\text{ where    }
\mathcal{S}^{(N)}_D:=\mathcal{O}_{(X,\Delta)}(D).
\]
}
\end{definition}

Now, if $\Cl(X,\Delta)$ is torsion free, we can define the Cox sheaf to be the sheaf of divisorial algebras associated to any $N \subseteq \WDiv(X,\Delta)$ such that $K \to \Cl(X,\Delta)$ is an isomorphism.
If $\Cl(X,\Delta)$ has torsion, we proceed similarly to~\cite[Constr. 1.4.2.1]{ADHL15} in the case of ordinary Cox rings. The idea is to take the sheaf of divisorial algebras $N \subseteq \WDiv(X,\Delta)$ projecting onto $\Cl(X,\Delta)$, and then quotient by a certain ideal sheaf identifying homogeneous components $\mathcal{S}^{(N)}_D$ and $\mathcal{S}^{(N)}_{D'}$ whenever $D$ and $D'$ are linearly equivalent.

\begin{definition}
\label{def:logCox}
{\em 
Let $(X,\Delta)$ be a log pair with finitely generated log divisor class group  $\Cl(X,\Delta)$. Let $N \subseteq \WDiv(X,\Delta)$ be a finitely generated subgroup. 
Assume that 
\[
c: N \to \Cl(X,\Delta),
\qquad 
D \mapsto [D]
\]
is onto and denote its kernel by $N^0$. Let $\mathcal{S}$ be the sheaf of divisorial algebras associated to $N$. Let $\chi: N^0 \to \cc(X,\Delta)^*$ be a group homomorphism  yielding
\begin{equation}
\label{chi-eq}
\divv(\chi(E))=E     
\end{equation}
for all $E \in N^0$. Denote by $\mathcal{I}$ the sheaf of ideals of $\mathcal{S}$ locally generated by the sections $1-\chi(E)$, where $E$ runs through $N^0$. We define the \emph{log Cox sheaf} of $(X,\Delta)$ to be the quotient sheaf $\Cox{(X,\Delta)}:=\mathcal{S}/\mathcal{I}$, graded by
\[
\Cox{(X,\Delta)}:=\bigoplus_{[D] \in \Cl(X,\Delta)} (\Cox{(X,\Delta)})_{[D]},
\text{ where  }
(\Cox{(X,\Delta)})_{[D]}:=\pi \left( \bigoplus_{D' \in c^{-1}([D])} S_{D'} \right),
\]
and $\pi:\mathcal{S} \to \Cox{(X,\Delta)}$ is the projection.
The ring of global sections 
\[
{\rm Cox}(X,\Delta):=\Gamma(X,\Cox{(X,\Delta)})
\]
of this sheaf is called the \emph{log Cox ring} of $(X,\Delta)$.
In what follows, we may need to consider the Cox ring with respect to a finitely generated subgroup $N\leqslant {\rm WDiv}(X,\Delta)$
which may not surject onto 
${\rm Cl}(X,\Delta)$.
Analogously, in this case we have a homomorphism
$N\rightarrow {\rm Cl}(X,\Delta)$ with kernel $N^0$
and we choose a group
homomorphism
$\chi\colon N^0\rightarrow
\mathbb{C}(X,\Delta)^*$ satisfying the equality~\eqref{chi-eq}.
In this case, we denote the Cox ring by
\[
{\rm Cox}(X,\Delta)_{N,\chi}.
\]
}
\end{definition}

\begin{remark}{\em 
It is clear from the construction that $N^0$ is always a subgroup of $\PDiv(X)$, so $\chi$ is a group homomorphism $\chi: N^0 \to \cc(X)^*$ to the field of rational functions on $X$. Thus, the assertions from~\cite[Sec 1.4.2]{ADHL15} hold. In particular, if $\Gamma(X,\mathcal{O}^*)=\cc^*$, then the above definition of the log Cox sheaf and log Cox ring does not depend on the choice of $N$ and $\chi$ up to isomorphism,  see~\cite[Prop. 1.4.2.2]{ADHL15}. Note that the requirement $\Gamma(X,\mathcal{O}^*)=\cc^*$ is fulfilled for projective varieties and quasi-cones.}
\end{remark}

\begin{proposition}\label{prop:fg-log-cox}
Let $(X,\Delta)$ be a log pair.
The Cox ring ${\rm Cox}(X,\Delta)$ is finitely generated 
if and only if ${\rm Cox}(X)$ is finitely generated.
\end{proposition}

\begin{proof}
Note that we have an inclusion of groups
${\rm Cl}(X)\leqslant {\rm Cl}(X,\Delta)$ of finite index.
Furthermore, we have a monomorphism of rings
${\rm Cox}(X)\hookrightarrow {\rm Cox}(X,\Delta)$
obtained by coarsening the grading.
By~\cite[Corollary 1.2.5]{ADHL15}, we conclude that $\mathcal{R}_{(X,\Delta)}(D_1,\dots,D_k)$ is finitely generated over $\cc$ if and only if
$\mathcal{R}_{X}(m_1D_1,\dots,m_kD_k)$ is finitely generated over $\cc$.
\end{proof}

\begin{corollary}
\label{cor:CoxlogCoxfg}
Let $X$ be a Mori dream space.
For any log pair structure $(X,\Delta)$
the Cox ring $\mathcal{R}(X,\Delta)$ is finitely generated.
\end{corollary}

The following proposition says that the only case in which the Cox ring of a log pair $(X,\Delta)$ may be non-isomorphic to the Cox ring of $X$
is when there is at least one coefficient of $\Delta$ which is equal to or larger than one half.

\begin{proposition}
Let $(X,\Delta)$ be a log pair so that 
${\rm coeff}_P(\Delta)<\frac{1}{2}$ for every prime divisor $P$ on $X$.
Then ${\rm Cox}(X,\Delta)\cong {\rm Cox}(X)$.
\end{proposition}

\begin{proof}
Note that ${\rm coeff}_P(\Delta)<\frac{1}{2}$ if and only if
$\Delta_s$, the standard approximation of $\Delta$, equals the zero divisor.
The above condition is equivalent to 
${\rm WDiv}(X,\Delta)\cong {\rm WDiv}(X)$.
Furthermore, any section of a orbifold Weil divisor of $(X,\Delta)$
is just a section of a Weil divisor on $X$.
Hence, we have that $\mathcal{R}_{(X,\Delta)} \cong\mathcal{R}_{X}$.
This implies the desired isomorphism.
\end{proof}

We are interested in the universal abelian covering space that the log Cox ring provides us. 
We can also study other abelian covers of $X$.
In analogy to the case of the ordinary Cox ring, they should correspond to quotients of ${\rm Cox}(X,\Delta)$ by subgroups of $\Cl(X,\Delta)$, see~\cite[Them 4.2.1.4]{ADHL15}.  We explore the interplay in the following example.

\begin{example}{\em 
Consider the $D_4$-cone singularity $X$ given by the equation $\{x_3^2+x_1^2x_2+x_2^2x_1=0\}$ in $\mathbb{A}^3$. Then 
\[
\Cl(X)= \langle D_1, D_2 \mid 2D_1=2D_2=0 \rangle \cong (\zz/2\zz)^2,
\]
where $D_1=V(x_1)$, $D_2=V(x_2)$. Then, the Cox ring ${\rm Cox}(X)$ is the $A_1$-singularity $Y$ given by $y_1^2+y_2^2+y_3^2$, where $(\zz/2\zz)^2$ acts via
\[
(a,b)\cdot (y_1,y_2,y_3) := ((-1)^{a}y_1,(-1)^{b}y_2,(-1)^{a+b}y_3).
\]
The generating invariants for this action are 
\[
x_1:=y_1^2, 
\quad
x_2:=y_2^2,
\quad
x_3:=y_1y_2y_3,
\text{ and}
\quad
x_4:=y_3^2.
\]
They satisfy the relation $x_3^2-x_1x_2x_4=0$.
Furthermore, the relation of the $A_1$-singularity gives us $x_1+x_2+x_4=0$. Eliminating $x_4$ gives us back our initial relation. Now, consider pair structures $(X,\Delta)$, with $\Delta=(1-\frac{1}{m_1})D_1 + (1-\frac{1}{m_2})D_2$. We have that 
\[
\Cl(X,\Delta)=\left\langle \frac{1}{m_1}D_1,\frac{1}{m_2}D_2 \mid 2   D_1=2D_2=0 \right\rangle.
\]
Also note that for the ordinary Cox cover $\pi:Y \to X$, we have 
\[
\pi^{-1}(D_1)=V(y_2+iy_3) \cup V(y_2-iy_3),
\quad
\pi^{-1}(D_2)=V(y_1+iy_3) \cup V(y_1-iy_3).
\]
Since $\pi:Y \to X$ does not ramify over divisors, we have 
\[
\pi^*(\Delta)=\left(1-\frac{1}{m_1}\right)(V(y_2+iy_3) + V(y_2-iy_3)) + \left(1-\frac{1}{m_2} \right)(V(y_1+iy_3) + V(y_1-iy_3)).
\]
That is, when $m_1$ and $m_2$ are different from one, the pull-back does not have normal crossings.}
\end{example}

To finish this subsection, we show the Cox ring of a log Fano structure on $\pp^1$. In this case, the standard approximation has at most three non-trivial coefficients.
In the case that there are two non-trivial coefficients,
the Cox ring is isomorphic to $\mathbb{A}^2$ with a
characteristic quasi-torus action.
In the case that there are three non-trivial coefficients,
the Cox ring may not be isomorphic to $\mathbb{A}^2$.

\begin{example}{\em 
Let $\Delta$ be an effective divisor on $\pp^1$ so that
$-(K_{\pp^1}+\Delta)$ is ample.
Assume that $\Delta_s$ has two non-trivial coefficients.
Then $\Delta_s=\left(1-\frac{1}{n}\right)p+\left(1-\frac{1}{m}\right)q$ for some positive integers $n$ and $m$.
In this case, we have that
${\rm Cl}(\pp^1,\Delta)=\langle \frac{1}{n}p,\frac{1}{m}q\mid  p=q \rangle$.
The above group is isomorphic to $\zz\oplus \zz/\gcd(n,m)\zz$.
Let $g=\gcd(n,m)$.
We conclude that the Cox ring is isomorphic to $\mathbb{A}^2$ with the Picard action
\[
t\cdot (x,y) \mapsto \left(t^{\frac{m}{g}}x,t^{\frac{n}{g}}y \right)
\]
and
\[
\mu\cdot (x,y)\mapsto 
\left(\mu^{\frac{n}{g}}x,\mu^{\frac{-m}{g}}y\right),
\]
where $\mu$ is a $g$-th root of unity.
}
\end{example}

\begin{example}
{\em
Let $\Delta$ be an effective divisor on $\pp^1$ so that
$-(K_{\pp^1}+\Delta)$ is ample.
Assume that $\Delta_s$ has three non-trivial coefficients.
In this case, the coefficients of $\Delta_s$ correspond to platonic triples (see, e.g.,~\cite{LS13,LLM19}).
In this case, we have that
${\rm Cl}(\pp^1,\Delta)=\langle \frac{1}{n}p,\frac{1}{m}q,\frac{1}{s}r \mid p=q=r\rangle$.
Let $g=\gcd(ms,ns,nm)$.
We may assume that the points $p,q$ and $r$ are 
$0,\{\infty\}$ and $1$, respectively.
In this case, the class group is isomorphic to
$\zz \oplus T_{n,m,s}$,
where $T_{n,m,s}$ is the roots system of the fork Dynkin diagram
with three branches of length $n,m$ and $s$.
The Cox ring is isomorphic to
\[
\cc[x,y,z]/\langle x^n+y^m+z^s\rangle.
\]
The characteristic quasi-torus action is given by 
\[
t\cdot (x,y,z) =\left(
t^{\frac{ms}{g}}x,
t^{\frac{ns}{g}}y,
t^{\frac{nm}{g}}z
\right)
\]
and $T_{n,m,s}$ acts on $(x,y,z)$ 
in the usual way (see, e.g.,~\cite{Muk04}).
}
\end{example}

\begin{remark}{\em 
By taking all the possible Cox rings of log Fano pairs on $\pp^1$
and quotient by the finite part of the characteristic quasi-torus action, we obtain back surface klt singularities. These singularities are quotients of smooth points by finite groups.
For the classification of surface klt singularities see, e.g.,~\cite{Ale93}.
}
\end{remark}

\subsection{The relative Cox ring of a log pair}\label{subsec:rel-log-cox}
In this subsection, we define the relative Cox ring of a log pair,
prove some basic properties, and give some examples.

\begin{definition}\label{def:rel-cox}{\em 
Let $(X,\Delta)$ be a log pair and $\phi\colon X\rightarrow Z$ be a contraction.
We define the \emph{relative log Cox sheaf of $X/Z$} to be the 
direct image sheaf 
\[
\Cox{(X/Z,\Delta)}:=\phi_* \Cox{(X,\Delta)},
\]
where $\Cox{(X,\Delta)}$ is the log Cox sheaf of $(X,\Delta)$ as in Definition~\ref{def:logCox}. 
If $\Cox{(X/Z,\Delta)}$ is a sheaf of finitely generated $\mathcal{O}_Z$-algebras, we say that $X \to Z$ is a \emph{relative Mori dream space} for the log pair $(X,\Delta)$.
The \emph{relative affine log Cox ring} is defined to be
\[
{\rm Cox}^{\rm aff}(X/Z,\Delta):=\Gamma(\mathcal{R}_{(X/Z,\Delta)},Z).
\]
We write ${\rm aff}$ on top of the relative Cox ring to stress that, in this case, we are working with an affine base $Z$.
Later on, we will be interested in the local behaviour around some special point of the base.
}
\end{definition}

More generally, we can make the above definitions if $h \colon X \to Z$ is an aff-contraction, see Definition~\ref{def:aff-contraction}.

\begin{remark}\label{rem:gen-stalks}{\em 
If $Z$ is a point, we can identify the relative log Cox sheaf $\Cox{(X/Z,\Delta)}$  with the log Cox ring ${\rm Cox}(X,\Delta)$. More generally, when $Z$ is affine, then $\Cox{(X/Z,\Delta)}$ is a sheaf of finitely generated $\mathcal{O}_Z$-algebras if and only if the algebra of global sections ${\rm Cox}^{\rm aff}(X/Z,\Delta)$ is finitely generated over $\mathcal{O}_Z(Z)$ and thus over $\cc$ by the same argument as in~\cite[Prop. 4.3.1.3]{ADHL15}.

Again more generally, if $\Cox{(X/Z,\Delta)}$ is a sheaf of finitely generated $\mathcal{O}_Z$-algebras, then any fiber $X_z:=\phi^{-1}(z)$ of $\phi: X \to Z$ has an open neighbourhood $X_z \subseteq U \subseteq X$, such that $\Cox{(X,\Delta)}(U)$ is a finitely generated $\mathcal{O}_X(U)$-algebra.}
\end{remark}

\subsection{The local Cox ring}\label{subsec:local-cox}
In this subsection, we define the local Cox ring for germs $(X,\Delta;x)$, where $(X,\Delta)$ is a  pair and $x \in X$ is a closed point. More generally, when $\phi: X \to Z $ is a contraction, we define the relative local Cox ring for closed points $z \in Z$.

Here, it makes sense to consider different local models depending on the needs. A priori, we consider points on algebraic varieties $X$. Since we can realize $\Cl(X,\Delta,x)$ as a subgroup of $\Cl(X,\Delta)$. The approach is to  define the local Cox ring at $x \in X$ to be 
$$
\bigoplus_{[D] \in \Cl(X,\Delta,x)} \Gamma(X,\Delta,\mathcal{O}_X(D)).
$$
This definition amounts to choosing a subgroup $N$ of the orbifold Weil divisors of $(X,\Delta)$ surjecting onto $\Cl(X,\Delta,x)$ with kernel $N^0$ and a character $\chi: N^0 \to \cc(X)^*$. Note that by~\cite[Theorem 2.3]{HMT20}, the set of isomorphism classes of Cox rings defined in this way is in bijection to
\[
\mathrm{Ext}^{1}(\Cl(X,\Delta,x),\mathcal{O}(X)^{*}).
\]
This construction only makes sense if $X$ is affine, so we will assume this in the following. Moreover, we assume that the group $N$ consists of Weil divisors going through $x$ and we fix a character $\chi: N^0 \to \cc(X)^*$. Then, we can define the affine local Cox ring 
(or aff-local Cox ring for short) as above. If it is finitely generated over $\mathcal{O}(X)$, then its spectrum is an affine scheme of finite type.

We denote by $X_x$ the spectrum of the local ring of $X$ at $x$. We have $\Cl(X_x,\Delta_x) \cong \Cl(X,\Delta,x)$ and we can identify uniquely the group $N$ with a subgroup of $\WDiv(X_x,\Delta_x)$.
Here, $\Delta_x$ is the pull-back of $\Delta$ to $X_x$.
Moreover, since  $X$ and $X_x$ are birational, we can use the character $\chi$ from above in order to define the Cox ring
$$
\bigoplus_{[D] \in \Cl(X_x)} \Gamma(X_x,\Delta_x,\mathcal{O}_{X_x}(D)).
$$
This is a gr-local ring, finitely generated over the degree-zero part $\mathcal{O}_{X,x}$, which is why we call it the \emph{gr-local Cox ring} of $x \in X$. By localizing at the unique graded maximal ideal, we get a local ring.

\begin{definition}
{\em
Let $X$ be an affine variety, $(X,\Delta)$ a log pair, and $x \in X$ a closed point. Fix a subgroup $N \subseteq \WDiv(X,\Delta)$ of orbifold Weil divisors going trough $x$ such that the induced homomorphism $\varphi\colon N \to \Cl(X,x)$ is surjective. Fix a character $\chi \colon \ker(\varphi) \to \cc(X)^*$. Let $\mathcal{S}$ be the sheaf of divisorial algebras on $X$ associated to $N$ and $\mathcal{I}$ the ideal subsheaf generated by sections $1-\chi(E)$, where $E \in \ker(\varphi)$. Then, we define the \emph{aff-local Cox ring} of $x \in (X,\Delta)$ to be
\[
{\rm Cox}(X,\Delta;x)^{\rm aff}_{N,\chi}:= \bigoplus_{[D] \in \Cl(X,\Delta,x)} \frac{ \bigoplus_{D' \in \varphi^{-1}([D])} \mathcal{S}_{D'}(X)}{\mathcal{I}(X)}.
\]
Similarly, where $X_x:=\Spec \mathcal{O}_{X,x}$, we define the \emph{gr-local Cox ring} of $x \in (X,\Delta)$ to be
\[
{\rm Cox}(X,\Delta;x)^{\gr}_{N,\chi}={\rm Cox}(X_x,\Delta_x)_{N,\chi}:= \bigoplus_{[D] \in \Cl(X_x,\Delta_x)} \frac{ \bigoplus_{D' \in \varphi^{-1}([D])} \mathcal{S}_{D',x}}{\mathcal{I}_{x}}.
\]
Finally, we define the \emph{local Cox ring} of $x \in (X,\Delta)$ to be the localization
\[
{\rm Cox}(X,\Delta;x)^{\loc}_{N,\chi}:=\left({\rm Cox}(X,\Delta;x)^{\gr}_{N,\chi}\right)_{\mathfrak{m}}
\]
at the unique homogeneous maximal ideal of the gr-local Cox ring. We denote the  spectra of these rings by 
\[
\overline{X}^{\aff}_{N,\chi}, \overline{X}^{\gr}_{N,\chi}, \text{ and } \overline{X}^{\loc}_{N,\chi}
\]
respectively.  The isomorphism class of the Cox rings just defined depends on the choice of $N$ and $\chi$, but having made such a choice, we will usually omit them in the notation. 
}
\end{definition}

In particular, ${\rm Cox}(X,\Delta;x)^{\gr}$ is the stalk of the quotient  sheaf $\mathcal{S}/\mathcal{I}$ at $x$ or, equivalently, the gr-localization at the unique pre-image of $x \in X$ in $\overline{X}^{\aff}$. Thus, since localization factors through gr-localization, we have the following commutative diagram.
\[
\xymatrix{ 
\overline{X}^{\loc} \ar[rr] \ar@/^1pc/[rrrr] \ar[dd] && \overline{X}^{\gr} \ar[rr]  \ar[dd] \ar[rrddd] && \overline{X}^{\aff}  \ar[dd] \ar[rrddd]
\ \\
\ \\
 \overline{X}^{\loc}_{\rm fin} \ar[rr]^{\rm \id} && \overline{X}^{\gr}_{\rm fin} \ar[rr]|!{[uu];[drr]}\hole \ar[rrd] && \overline{X}^{\aff}_{\rm fin} \ar[rrd] \\
&&&& X_x \ar[rr]  && X. }
\]

In particular, there is still a morphism $\overline{X}^{\loc} \to X_x$, but it may not be a quotient by the characteristic quasi-torus (at least in a strict sense).
For instance, for an open orbit with the unique fixed point of $\overline{X}^{\gr}$ in its closure, localization will remove all closed points and only keep the generic point of the orbit. 

Note that $\Cl(\overline{X}^{\gr})\cong \Cl(\overline{X}^{\loc})$, since the class group of the gr-local ring ${\rm Cox}(X,\Delta;x)^{\gr}$ is concentrated at the unique graded maximal ideal. This essentially means that we can iterate Cox rings in a unique way. Since ${\rm Cox}(\overline{X}^{\loc})$ can be obtained from  ${\rm Cox}(\overline{X}^{\gr})$ via base change.
The iteration of Cox rings is defined in~\ref{def:cox-iteration}.

\begin{definition}
\label{def:rel-loc-cox}
{\em 
Let $(X,\Delta)$ be a log pair and $\phi\colon X\rightarrow Z$ be a contraction. Let $z\in Z$ be a closed point.
Let $Z_z$ be the spectrum of the local ring $\mathcal{O}_{Z,z}$.
Let $X_z\rightarrow Z_z$ be the projective morphism obtained by the base change $Z_z\rightarrow Z$.
We denote by $\Delta_z$ the pull-back of $\Delta$ to $X_z$.
Analogously to the local case,
we can define the 
{\em relative gr-local Cox ring}
at $z\in Z$ to be 
\[
{\rm Cox}(X/Z,\Delta;z)^{\rm gr}:=
\bigoplus_{[D] \in \Cl(X_z,\Delta_z)} \frac{ \bigoplus_{D' \in \varphi^{-1}([D])} \mathcal{S}_{D',\phi^{-1}(z)}}{\mathcal{I}_{\phi^{-1}(z)}}.
\]
Note that the relative gr-local Cox ring comes with a natural maximal graded ideal $\mathfrak{m}$, i.e., 
the ideal generated by homogeneous regular functions in the Cox ring which correspond to Weil divisors on $X$ that intersect the fiber $\phi^{-1}(z)$ non-trivially.
The {\em relative local Cox ring}
of $(X/Z,\Delta)$ at $z$ is then defined to be
\[
{\rm Cox}(X/Z,\Delta;z)^{\rm loc}:=
\left(
{\rm Cox}(X/Z,\Delta;z)^{\rm gr}
\right)_{\mathfrak{m}}.
\]
The {\em local Cox ring} comes as the special case where $X \to X$ is the identity and $x \in X$:
\[
{\rm Cox}(X,\Delta;x)^{\rm loc} \cong {\rm Cox}(X_x/X_x,\Delta)^{\rm loc}.
\]
}
\end{definition}

\subsection{Properties of generalized Cox rings}
\label{subsec:prop-gen-cox}

In this subsection, we prove some properties of the Cox rings defined in the previous subsections.
First, we prove two general statements concerning relative Mori dream spaces, then we focus on the case of klt pairs.

\begin{proposition}
\label{prop:rel-log-Cox-sheaf-groc}
Let $X \to Z$ be a relative Mori dream space.
%, where $Z$ is either projective, a quasi-cone or the spectrum of a local ring essentially of finite type.
Then $\Cox{(X/Z,\Delta)}$ is a sheaf of gr-local rings.
In particular, if $Z$ is the spectrum of a local ring essentially of finite type, then the local Cox ring ${\rm Cox}^{\rm loc}(X/Z,\Delta;z)$ is a local ring essentially of finite type.
\end{proposition}

\begin{proof}
Since $X \to Z$ is a relative Mori dream space, we know that the stalks $\left(\Cox{(X/Z,\Delta)}\right)_z$ are graded rings, with zero graded piece $\mathcal{O}_{Z,z}$, which is a local ring. Finite generation over $\mathcal{O}_{Z,z}$ follows as in Corollary~\ref{cor:sheaf-groc-sheaf-finite-type}. The last assertion follows from the definition.
\end{proof}

The following set of statements shows that klt singularities and weakly Fano pairs behave optimally with respect to the Cox construction. This is known for the classical Cox ring of weakly Fano pairs and klt quasi-cones (see, e.g.,~\cite{GOST15}).

\begin{theorem}\label{thm:relative-fano}
Let $(X,\Delta)$ be of Fano type over $Z$, where $Z$ is either projective, a quasi-cone or the spectrum of a local ring essentially of finite type. Then $X\rightarrow Z$ is a relative Mori dream space for the log pair $(X,\Delta)$.
\end{theorem} 

\begin{proof}
By Remark~\ref{rem:gen-stalks} and Lemma~\ref{le:fg-local}, it suffices to check that for every point
$z\in Z$ the stalk $\left(\Cox{(X/Z,\Delta)}\right)_z$ is a finitely generated $\mathcal{O}_{Z,z}$-algebra.

We denote by $\phi_X \colon X\rightarrow Z$ the contraction morphism.
By~\cite[Corollary 1.4.3]{BCHM10}, we can take a small $\mathbb{Q}$-factorialization $Y\rightarrow X$ of $X$.
Note that $\pi\colon Y\rightarrow Z$ is still of Fano type over $Z$
(see, for instance~\cite[Lemma 3.1]{GOST15}).
Let $K_Y+\Delta_Y=\pi^*(K_X+\Delta)$.
We denote by $\phi_Y\colon Y\rightarrow Z$ the contraction morphism.
By~\cite[Corollary 1.3.2]{BCHM10} and Lemma~\ref{le:fg-local}, we know that 
$\left(\Cox{(Y/Z)}\right)_z$ is a finitely generated algebra over $\mathcal{O}_{Z,z}$.
Since $\pi$ is small, we conclude that
$\left(\Cox{(X/Z)}\right)_z$ is a finitely generated algebra over $\mathcal{O}_{Z,z}$.
By Proposition~\ref{prop:fg-log-cox}, we deduce that
$\left(\Cox{(X/Z,\Delta)}\right)_z$
is a finitely generated algebra over $\mathcal{O}_{Z,z}$.
Hence, $\Cox{(X/Z,\Delta)}$ is a sheaf 
of finitely generated $\mathcal{O}_Z$-algebras.
Then, $X\rightarrow Z$ is a relative Mori dream space for
the log pair $(X,\Delta)$.
\end{proof}

\begin{corollary}
\label{cor:klt-mds}
Let $x \in (X,\Delta)$ be a klt singularity.
 Then $(X_x,\Delta_x)$ is a Mori dream space at $x\in X$, that is 
 ${\rm Cox}^{\rm gr}(X,\Delta;x)$ 
 is a gr-local ring, finitely generated as an algebra over $\mathcal{O}_{X,x}$.
\end{corollary}

\begin{proof}
This follows by setting $X=Z=X_x$ in the statement of Theorem~\ref{thm:relative-fano}.
\end{proof}

\begin{corollary}
\label{cor:klt-mds-glob}
Let $(X,\Delta)$ be a klt pair with finitely generated log divisor class group $\Cl(X,\Delta)$. Then the Cox sheaf is a sheaf of grocal rings. In particular, it is a sheaf of finitely generated $\mathcal{O}_X$-algebras.
\end{corollary}

\begin{proof}
Since $(X,\Delta)$ is klt, Corollary~\ref{cor:klt-mds} tells us that
for any $x \in X$, the gr-local Cox ring ${\rm Cox}^{\rm gr}(X,\Delta;x)$ is finitely generated over $\mathcal{O}_{X,x}$.

By~\cite[Remark 1.3.1.4]{ADHL15}, we have a surjective homomorphism
\[
{\rm Cox}^{\rm gr}(X,\Delta;x)[x_1^{\pm 1},\ldots,x_k^{\pm 1}] \to \left(\Cox{(X,\Delta;x)}\right)_{x},
\]
where $\Cl(X,\Delta,x)=\Cl(X,\Delta)/\Cl_x(X,\Delta)$ and the subgroup $\Cl_x(X,\Delta)$ has $k$ generators.
Hence, we have that 
\[
\left({\rm Cox}^{\rm loc}(X,\Delta;x))\right)_{x}
\]
is a gr-local ring, finitely generated as an algebra over $\mathcal{O}_{X,x}$. The last assertion follows from  Lemma~\ref{le:fg-local}.
\end{proof}

In view of~\cite[Proposition 4.3.1.4]{ADHL15}, we have the stronger result that affine klt varieties with finitely generated class group are Mori dream spaces.

\begin{corollary}
\label{cor:klt-aff-mds}
Let $(X,\Delta)$ be an affine klt pair with finitely generated log divisor class group $\Cl(X,\Delta)$. Then the Cox ring ${\rm Cox}^{\rm aff}(X,\Delta)$ is finitely generated.
\end{corollary}

\begin{proof}
This follows directly from Corollary~\ref{cor:klt-mds-glob} and~\cite[Prop. 4.3.1.3]{ADHL15}.
\end{proof}

\begin{lemma}\label{lem:one-dim-cone}
Let $X$ be a $\mathbb{Q}$-factorial $\mathbb{T}$-variety.
Let $X\rightarrow Z$ be a projective contraction.
Let $(X,\Delta)$ be a log pair which is of Fano type over $Z$.
Let $D$ be a $\mathbb{T}$-invariant $\mathbb{Q}$-divisor on $X$. 
Assume that for each prime $P\subset X$ we have that 
\[
{\rm coeff}_P(\Delta) \geq 
1-\frac{1}{i_P(D)},
\]
where $i_P(D)$ is the Cartier index of $D$ at the generic point of $P$.
Then, the spectrum $Y$ of the ring 
\[
\bigoplus_{m\in \zz}
H^0(X/Z,\mathcal{O}_X(mD)) 
\]
is klt type.
\end{lemma}

\begin{proof}
Note that if neither $-D$ or $D$ is effective over the base $Z$, 
then there is nothing to prove.
Without loss of generality, we may assume that $D$ is effective.
We run a $D$-MMP over the base
$X\dashrightarrow X'$.
This $D$-MMP terminates since $X$ is of Fano type over the base.
Furthermore, $X'$ is also of Fano type over the base.
Since $\mathbb{T}$ is connected, this MMP is $\mathbb{T}$-equivariant.
The induced divisor $D'$ on $X'$ is still $\mathbb{T}$-invariant.
Let $X''$ be the ample model of $D'$ over $X$.
Hence, $X''$ is of Fano type over $Z$, being the image of a Fano type variety over $Z$.
Let $X^{(3)}$ be a small $\mathbb{Q}$-factorialization of $X''$.
Replacing $X$ by $X^{(3)}$,
we may assume that $D$ is ample over $Z$.

Let $\phi\colon \tilde{X}\rightarrow X$ be the relative spectrum of the divisorial sheaf 
$\bigoplus_{m\in \zz} \mathcal{O}_X(mD)$.
Hence, we have a projection morphism 
$r\colon \tilde{X}\rightarrow Y$ which contracts at most one horizontal divisor over $X$.
We denote such divisor (if it exists) by $F$.
Note that $(X,\Delta)$ is $\mathbb{Q}$-complemented over $Z$,
hence $(X,\Delta_s)$ is $\mathbb{Q}$-complemented over $Z$ as well.
Hence, we can find $\Delta'$ so that
$(X,\Delta_s+\Delta')$ is klt and log Calabi-Yau over $Z$.
Since $\Delta'$ is big over the base,
we can find a general ample divisor $A$ 
and $E\geq 0$ so that
$\Delta' \sim_{\qq,Z} A+E$.
We can assume that
$(X,\Delta_s+A+E)$ is a klt pair which is trivial over $Z$.
Let $\epsilon>0$ be small enough so that
$A-\epsilon D$ is an ample divisor over $Z$.
We can write 
$A-\epsilon D \sim_{\qq,Z} \Gamma \geq 0$ general enough so that the pair
$(X,\Delta_s+E+\Gamma)$ remains klt.
Note that we have a $\qq$-linear equivalence over the base
\[
-(K_X+\Delta_s + E +\Gamma) \sim_{\qq,Z} \epsilon D.
\]
Hence, we have that 
$\phi^*(K_X+\Delta_s+E+\Gamma)=K_{\tilde{X}}+\Gamma_{\tilde{X}}$ satisfies that
$(\tilde{X},\Gamma_{\tilde{X}}+(1-\epsilon)F)$ is a log pair which is klt and $\qq$-trivial over $Y$.
Define $\Gamma:=r_*\Gamma_{\tilde{X}}$.
Then, we have that
$(Y,\Gamma)$ is crepant equivalent to 
$(\tilde{X},\Gamma_{\tilde{X}}+(1-\epsilon)F)$
so it is a klt pair.
\end{proof}

\begin{theorem}\label{thm:klt-Cox-ring}
Let $\phi\colon X \rightarrow Z$ be a contraction.
Assume that $(X,\Delta)$ is a log pair which is of Fano type over $Z$.
Let $K\leqslant {\rm WDiv}(X,\Delta)$ be a finitely generated subgroup.
Consider $\pi\colon K\rightarrow {\rm Cl}(X/Z,\Delta)$ the induced homomorphism and $K_0$ its kernel.
Let $\chi\colon K_0\rightarrow \mathbb{C}(X,\Delta)^*$ be a character.
Then, the spectrum of the Cox ring ${\rm Cox}(X/Z,\Delta)_{K,\chi}$ is 
klt type. 
\end{theorem}

\begin{proof}
Observe that replacing $X$ with a small $\mathbb{Q}$-factorialization does not change the Cox ring, so we may assume that $X$ is $\mathbb{Q}$-factorial.
Let $D_1,\dots,D_s,D_{s+1},\dots,D_r$ be a finite set of Weil divisors so that 
$\langle D_1,\dots,D_s \rangle$ maps isomorphically to
$\pi(N)_{\rm free}$
and 
$\langle D_{s+1},\dots,D_r \rangle$ surjects onto 
$\pi(N)_{\rm tor}$.
We denote the spectrum of the Cox ring ${\rm Cox}(X/Z,\Delta)_{N,\chi}$ by $Y'$.
Note that we have a natural split
$\mathbb{T}\cong \mathbb{T}_0 \times A$,
where $A$ is a finite abelian group
and $\mathbb{T}_0$ is a torus.
Let $Y$ be the quotient of $Y'$ by $A$ and $X'$ the quotient of $Y'$ by $\mathbb{T}_0$.
Then, we have a commutative diagram as follows
\[
\xymatrix{ 
Y'\ar[r]^-{/A}\ar[d]_-{/\mathbb{T}_0} & Y\ar[d]^-{/\mathbb{T}_0} \\ 
X' \ar[r]^-{/A} & X.}
\]
We denote the finite quasi-\'etale Galois morphism
$X'\rightarrow X$ by $p$.
We have a natural isomorphism
\[
{\rm Cox}(X/Z,\Delta)_{N,\chi} 
\cong 
\bigoplus_{(m_1,\dots,m_s)\in \zz^s} 
H^0(X'/Z, \mathcal{O}_{X'}(m_1p^*(D_1)+
\dots +m_sp^*(D_s)). 
% simeq {\rm Cox}(X'/Z,\Delta').
% indeed, it is not isomorphic to the Cox of X'
\] 
The above isomorphism is induced by
the isomorphism
\[
p_*\mathcal{O}_{X'} 
\cong 
\bigoplus_{D\in \pi(K)_{\rm tor}} \mathcal{O}_X(D).
\]
The finite morphism
$X'\rightarrow X$ ramifies with multiplicity at most $m$ at prime divisors of $X$ with coefficient at least $1-\frac{1}{m}$.
Then, the log pull-back of $K_X+\Delta$
is a klt pair $K_{X'}+\Delta'$.
Since $(X,\Delta)$ is of Fano type over $Z$, we can find a boundary $B$ on $X$ so that $K_X+\Delta+B\sim_{\mathbb{Q},Z} 0$ is klt and $B$ is big over $Z$.
We may assume that $B$ contains no component of the branch locus of $p$.
Then, the pull-back $B':=p^*(B)$ satisfies that $K_{X'}+\Delta'+B'_{\mathbb{Q},Z} 0$ is klt and $B'$ is big over $Z$.
Hence, $(X',\Delta')$ is of Fano type over $Z$.
Hence, it suffices to prove the statement for $\pi(K)$ free.
Thus, we may replace $X$ with $X'$
and assume $s=r$.

We reduce to the case in which each 
$D_i \sim_{\mathbb{Q},Z} K_X+B_i$, 
where $(X,B_i)$ is klt and $B_i\geq A$ for some fixed effective divisor $A$
ample over $Z$.
Without loss of generality, we may assume that each $D_i$ is effective.
Furthermore, we may assume that $(X,B'+D_i)$ is klt, where 
$B':=\Delta+B$ is big over $Z$.
For this purpose, it suffices to replace $D_i$ with $D_i/n_i$ with $n_i$ large enough.
Note that $K_X+B'+D_i\sim_{\mathbb{Q},Z} D_i$.
Since $B'$ is big over $Z$, we can write $B'\sim_{\mathbb{Q},Z} A+E$
where $A$ is ample over $Z$ and $E$ is effective.
By choosing a very general section of $A$, we may replace $B'$ with $A+E$
and set $B_i=D_i+A+E$.

In this step, we reduce to the case in which there is a single divisor $D_1$.
For each $D_i$, we can find $k_i>0$ so that $k_iD_i$ is Cartier.
Consider the orbifold projective bundle
\[
X_1:=\mathbb{P}_X(\mathcal{O}_X(D_1)\oplus \dots \oplus \mathcal{O}_X(D_s)) 
\]
over $X$,
and the projective bundle
\[
X_2:=\mathbb{P}_X (\mathcal{O}_X(k_1D_1)\oplus \dots \oplus \mathcal{O}_X(k_sD_s))
\]
over $X$.
Note that we have a finite morphism
$X_1\rightarrow X_2$.
We denote by $\pi_1\colon X_1\rightarrow X$ and 
$\pi_2\colon X_2\rightarrow X$ the corresponding morphisms.
We claim that $X_2$ is of Fano type over $Z$.
Let $H_1,\dots,H_{s+1}$ be the hyperplane sections of $X_2$ over $X$
and $H:=H_1+\dots+H_{s_1}$.
Let $\Delta_{X_2}=\pi_2^*(B')$.
By inversion of adjunction, we conclude that the pair
$(X_2,H+\Delta_{X_2})$ is dlt and
$K_{X_2}+H+\Delta_{X_2}$ is $\mathbb{Q}$-trivial over $Z$.
Furthermore, the boundary
$H+\Delta_{X_2}$ is big over $Z$.
For $\epsilon>0$ small enough, the pair
$\pi_2^*(A)+\epsilon H$ is ample over $Z$.
Let $A_{X_2} \sim_{\mathbb{Q},Z} \pi_2^*(A)+\epsilon H$ be a general effective divisor.
We conclude that 
\[
K_{X_2}+(1-\epsilon)H 
+\Delta_{X_2} + (1-\epsilon)\pi_2^*(A)+
\pi_2^*(H) + A_{X_2} \sim_{\mathbb{Q},Z} 0
\]
is klt and its boundary is effective.
We conclude that $X_2$ is of Fano type over $Z$.
By taking $\epsilon$ small enough, we can make sure that the log pull-back of the above pair to $X_1$ is a  klt pair.
Thus, we conclude that $X_1$ is of Fano type over $Z$ as well.
Note that the section ring of the tautological $\mathbb{Q}$-line bundle of $X_1$ coincides with the multi-section ring generated by the $D_i$'s on $X$. However, the grading given by the ring of sections of the tautological line bundle is coarser. Thus, replacing $X$ with $X_1$, we reduced the statement to the $\mathbb{T}$-equivariant case with $s=1$. Then, the statement follows from Lemma~\ref{lem:one-dim-cone}.
\end{proof}

\begin{corollary}\label{cor:loc-pot-klt}
Let $(X,\Delta)$ be a log pair
and $\phi\colon X\rightarrow Z$ be a contraction.
Assume that $(X,\Delta)$ is of Fano type over $Z$.
Let $z\in Z$ be a closed point.
Then, the spectrum of the Cox ring
${\rm Cox}^{\rm aff}(X/Z,\Delta)$
(resp. 
${\rm Cox}^{\rm gr}(X/Z,\Delta;z)$
or 
${\rm Cox}^{\rm loc}(X/Z,\Delta;z)$)
is klt type.
\end{corollary}

\begin{proof}
The statement for ${\rm Cox}^{\rm aff}(X/Z,\Delta;z)$ is a direct consequence of Theorem~\ref{thm:klt-Cox-ring}.
It suffices to show the statement for 
${\rm Cox}^{\rm gr}(X/Z,\Delta;z)$.
Indeed, the spectrum of the localization at the maximal graded ideal will have klt type singularities provided that the spectrum of
${\rm Cox}^{\rm gr}(X/Z,\Delta;z)$ satifies that property.
Let $W'_1,\dots,W'_k$ be Weil divisors on $X_z$ 
whose classes generate
${\rm Cl}(X_z/Z_z,\Delta_z)$.
Up to shrinking $Z$ around $z$, we may find Weil divisors $W_1,\dots,W_k$ on $X$ whose pull-backs to $X_z$ coincide with the $W'_i$'s.
Up to shrinking $Z$ around $z$ again, 
we may assume that the group $K$ generated by the $W_i$'s in ${\rm Cl}(X/Z,\Delta)$ 
is isomorphic to the group
$K'$ generated by the $W_i$'s in 
${\rm Cl}(X_z/Z_z,\Delta_z)$.
We can consider the Cox ring construction
with respect to the subgroup $K\leqslant {\rm Cl}(X/Z,\Delta)$.
We denote this ring by 
\begin{equation}
\label{eq:ring}
{\rm Cox}^{\rm aff}(X/Z,\Delta)(W_1,\dots,W_k).
\end{equation} 
By Theorem~\ref{thm:klt-Cox-ring}, 
we know that the spectrum of the ring~\eqref{eq:ring}
is klt type.
Note that the spectrum of the relative grocal Cox ring 
${\rm Cox}^{\rm gr}(X/Z,\Delta;z)$ is induced
by the base change $Z_z\rightarrow Z$ from 
the spectrum of the ring~\eqref{eq:ring}.
Hence, we conclude that the spectrum of
${\rm Cox}^{\rm gr}(X/Z,\Delta;z)$ is
klt type.
\end{proof}

\subsection{The local Henselian Cox ring}\label{subsec:local-Hensel-cox}
In this section, we define the local Henselian Cox ring for germs $(X,\Delta;x)$, where $(X,\Delta)$ is a pair and $x \in X$ is a closed point. More generally, when $\phi: X \to Z $ is a contraction, we define the relative local Henselian Cox ring for closed points $z \in Z$.

\begin{definition}{\em 
Let $(X,\Delta)$ be a log pair and $\phi\colon X\rightarrow Z$ be a contraction. Let $z\in Z$ be a closed point.
As usual, we denote by $Z^h$ the spectrum of the Henselization of the local ring of $Z$ at $z$.
We obtain a morphism by base change $X^h\rightarrow Z^h$.
The {\em relative gr-Henselian Cox ring} at $z\in Z$ is defined to be
\[
{\rm Cox}^{\rm gr\text{-}h}(X/Z,\Delta;z):={\rm Cox}(X^h/Z^h,\Delta^h),
\]
where the right side is the relative Cox ring of the associated morphism as defined in Definition~\ref{def:rel-cox}.
Then the {\em relative local Henselian Cox ring} at $z\in Z$ is defined to be
\[
{\rm Cox}^h(X/Z,\Delta;z) := 
\left({\rm Cox}^{\rm gr\text{-}h}(X/Z,\Delta;z)_{\mathfrak{m}}\right)^h.
\]
Here, the localization at the maximal ideal 
$\mathfrak{m}$ defined by homogeneous regular functions on the Cox ring which correspond to Weil divisors on $X^h$ which intersect $\phi^{-1}(z)$ non-trivially, followed by Henselization.
The boundary $\Delta^h$ is the boundary induced by $\Delta$ on $X^h$.
We can also define the {\em gr-Henselian} and {\em local Henselian Cox ring},
which come as the special cases where $X\to X$ is the identity and $x\in X$ is a closed point.
We denote them respectively by
\[
{\rm Cox}^{\rm gr\text{-}h}(X,\Delta;x) %:= {\rm Cox}(X^h/X^h,\Delta^h).
\quad \mathrm{and}
\quad
{\rm Cox}^h(X,\Delta;x).
\]
}
\end{definition}

The following example shows that the Henselian local Cox ring often
differs from the local Cox ring.
The former captures the local topology of the singularity,
while the latter does not (see subsection~\ref{subsec:covers-grocal-rings}).

\begin{example}{\em
This example shows that the Henselian local Cox ring
can be different from the local Cox ring.
Consider the factorial quasi-cone threefold singularity $X$ defined by
\[
\{ (x,y,z) \mid x^2+y^3+z^3w=0\}.
\] 
There is a singular stratum $C$ given by 
$x=y=z=0$. Around a general point $c$ of $C$, \'etale locally,  $X$ is isomorphic to $\cc$ times the $E_6$-singularity. Thus the regional fundamental group at $c$ is the binary tetrahedral group, which has abelianization $\zz/3\zz$. 
However, the class group ${\rm Cl}(X)$ and thus ${\rm Cl}(X_c)$ is trivial.
Since the regional fundamental group of the singularity is not perfect,
then ${\rm Cl}(X_c^h)$ is non-trivial, 
so the local Henselian Cox ring at $c$ is non-trivial
while the local Cox ring is $X$ itself.
Similar examples are given in~\cite[Ex 5.5]{BF84}.}
\end{example}

The following theorem shows that the relative local Henselian Cox ring is well-behaved for Fano
type morphisms.

\begin{theorem}\label{thm:hen-cox}
Let $X\rightarrow Z$ be a Fano type morphism,
where $(Z,z)$ is the spectrum of a local ring essentially of finite type over $\cc$.
Let $X^h\rightarrow Z^h$ be the base change to the Henselization of the local ring. 
Then, the following statements hold:
\begin{enumerate}
\item The class group
${\rm Cl}(X^h/Z^h)$ is finitely generated,
\item the relative gr-Henselian Cox ring
${\rm Cox}^{\rm gr\text{-}h}(X/Z,\Delta;z)$
is finitely generated over $\mathcal{O}_{Z^h}$, and 
\item the spectra 
${\rm Spec}({\rm Cox}^{\rm gr\text{-}h}(X/Z,\Delta;z))$
and
${\rm Spec}({\rm Cox}^h(X/Z,\Delta;z))$
are klt type.
\end{enumerate}
\end{theorem}

\begin{proof}
Note that $(Z,z)$ is a klt type singularity.
Indeed, since $X\rightarrow Z$ is a Fano type morphism, we can find a boundary $B$ on $X$ so that
$(X,B)$ is klt and $K_X+B\sim_{\mathbb{Q},Z}0$.
By the canonical bundle formula, 
we can find a boundary $B_Z$ on $Z$ so that
$(Z,B_Z)$ is klt.

We prove the first statement.
The subgroup of ${\rm Cl}(X^h/Z^h)$ generated by the
class of the effective Weil divisors contracted by $X^h\rightarrow Z^h$ and the class group of the generic fiber
is finitely generated.
Hence, it suffices to show that ${\rm Cl}(Z^h)$ is finitely generated.
Let $Y_0\rightarrow Z$ be a purely log terminal blow-up of the klt type singularity $(Z,z)$.
By base change, we obtain a plt blow-up
$Y_0^h\rightarrow Z^h$ of the local Henselian klt singularity. 
Let $E$ be the exceptional divisor.
Then, we have that 
\[
{\rm rank}_{\mathbb{Q}} ( {\rm Cl}(Z^h)_{\mathbb{Q}}) 
\leq {\rm rank}_{\mathbb{Q}} ({\rm Cl}_{\mathbb{Q}} (E)+1).
\]
On the other hand, the torsion subgroup of ${\rm Cl}(Z^h)$ 
is finite. Indeed, its order is bounded by the order of the regional fundamental group of $Z$ at $z$, which is finite by~\cite[Theorem 2]{Bra20}.
We conclude that ${\rm Cl}(Z^h)$ is finitely generated, so
${\rm Cl}(X^h/Z^h)$ is finitely generated as claimed.
This proves the first statement.

We prove the second statement.
Since ${\rm Cl}(X^h/Z^h)$ is finitely generated, we can find a finite set of Weil divisors
$W_1,\dots,W_r$ on $X^h$ which generate this group.
Recall that  $Z^h\rightarrow Z$ is a colimit of \'etale morphisms.
Hence, there exists a pointed \'etale cover $(Z',z')\rightarrow (Z,z)$ a base change 
$X'\rightarrow Z'$ and divisors $W'_1,\dots,W'_r$ on $X'$ which pull-back to $W_1,\dots,W_r$ respectively.
Since $Z'\rightarrow Z$ is of finite type, we conclude that $Z'$ is essentially of finite type over $\cc$.
Hence, $X'\rightarrow Z'$ is a projective morphism to the spectrum of a local ring essentially of finite type over $\cc$.
We can find a projective morphism $X''\rightarrow Z''$ of Fano type over a pointed affine algebraic variety $(Z'',z'')$ so that the base change of $X''\rightarrow Z''$ to the localization of $Z''$ at $z''$ is isomorphic to $X'\rightarrow Z'$.
Let $W^{''}_1,\dots,W^{''}_r$ be Weil divisors on $X''$ which restrict to the divisors $W'_1,\dots,W'_r$.
By Theorem~\ref{thm:relative-fano}, we conclude that the multigraded ring
\begin{equation}\label{eq:cox-ring-partial}
{\rm Cox}(X''/Z'',\Delta)(W^{''}_1,\dots,W^{''}_r)
\end{equation} 
is finitely generated over
$\cc$.
By faithfully flat base change, 
we conclude that the ring 
\[
{\rm Cox}^{\rm gr\text{-}h}(X/Z,\Delta;z)
\] 
is finitely generated over 
$\mathcal{O}_{Z^h}$.
This proves the second statement.

By the proof
of Corollary~\ref{cor:loc-pot-klt}, 
we have that the spectrum of the ring~\eqref{eq:cox-ring-partial} is klt type.
Hence, the same statement holds 
when we take the base change with respect to $Z^h\rightarrow Z''$.
We conclude that 
the spectrum of
${\rm Cox}^{\rm gr\text{-}h}(X/Z,\Delta;z)$
is klt type. Since the spectrum of ${\rm Cox}^{h}(X/Z,\Delta;z)$ is obtained from this ring by localization and Henselization, it is klt type as well.
\end{proof}

\begin{remark}{\em 
Finite generatedness of the class group $\Cl(\mathcal{O}_{X,x}^h)$ of the \'etale local ring more generally holds for rational singularities~\cite[Theorem 6.1]{BF84}. Moreover, it is isomorphic to the class group of the completion, i.e., $\Cl(\mathcal{O}_{X,x}^h) \cong \Cl(\widehat{\mathcal{O}_{X,x}})$, by~\cite[Theorem 6.2]{BF84}.
This complements the statement of Theorem~\ref{thm:Cl-gr-Hens}, that for gr-Henselian rational rings, all local class groups are isomorphic.}
\end{remark} 

\begin{remark} 
{\em 
We remark at this point that due to the above considerations, it would also be possible to define Cox rings and iteration of Cox rings for complete local rings. We omit the complete local case in order not to overload the notation. One may feel free to pass to a completion or also to base change to a {\em gr-complete Cox ring} anytime. Graded rings over complete local rings are also considered in~\cite{Cae83}.
}
\end{remark}

\section{Boundedness of iteration of Cox rings}
\label{sec:bounded}

In this section, we aim to prove that the iteration of the Cox ring of a relatively log Fano variety is bounded in terms of the dimension.
In subsection~\ref{subsec:iteration-cox-ring}, we will define the iteration of Cox.
We prove that for a Fano type morphism, the iteration stops after finitely many steps.
In particular, the iteration stabilizes for klt singularities.
In subsection~\ref{subsec:finite-rel-fund}, we prove the Jordan property for the relative regional fundamental group of a relative Fano type variety.
Finally, in subsection~\ref{subsec:bounded-iteration}, we use the Jordan property to prove the boundedness of iteration of Cox rings. 
This means that there exists an upper bound for the number of iterations which only depends on the dimension.

\subsection{Iteration of Cox rings for relative Mori dream spaces}
\label{subsec:iteration-cox-ring}

In this subsection, we define the iteration of Cox rings and generalize some results from~\cite{Bra19} to the case of relative Mori dream spaces. 
The setting is the following. Let $(X,\Delta)$ be a log pair and $\phi \colon X \to Z$ be a contraction, so that $(X,\Delta)$ becomes a relative Mori dream space over $Z$. Here, $Z$ will either be affine, the spectrum of a local ring essentially of finite type, or the Henselization of such a ring. 
We denote by $\mathbb{T}_X$ the characteristic quasi-torus of $X$ over $Z$, which is a direct product of a torus $\mathbb{T}_0$ and a finite abelian group $A$. The next crucial statement is a generalization of~\cite[Lemma 1]{Bra19}:

\begin{lemma}
\label{le:CoxCox}
Let $\phi\colon X \to Z$ be a contraction
and let $(X,\Delta)$ be a log pair.
Assume that $(X,\Delta)$ is a relative Mori dream space over $Z$.
Denote by $\overline{X}:=\Spec {\rm Cox}(X/Z,\Delta)_{N,\chi}$ the characteristic space of the relative log Cox ring with respect to $N \subseteq \WDiv(X,\Delta)$ and $\chi$. Denote by $X_1$ the finite Galois cover of $X$ corresponding to the abelian group $A$, by $\Delta_1$ the log-pullback of $\Delta$ to $X_1$,  and by $Y$ the quotient of $\overline{X}$ by $A$. Then, the following statements hold:
\begin{enumerate}
    \item $Y$ is $\qq$-factorial over $Z$ and there exists a boundary $\Delta_Y$ on $Y$, such that the aff-contraction $Y \to Z$ is a relative Mori dream space for $(Y,\Delta_Y)$ with characteristic quasi-torus $A$. In particular, there are $N_Y \subseteq \WDiv(Y,\Delta)$ and $\chi_Y$, such that
    \[
    {\rm Cox}(X/Z,\Delta)_{N,\chi} \cong {\rm Cox}(Y/Z,\Delta_Y)_{N_Y,\chi_Y}.
    \]
    \item There exists a boundary $\overline{\Delta}$ on $\overline{X}$, so that $(X_1,\Delta_1)$ is a relative Mori dream space over $Z$ if and only if  $(\overline{X},\overline{\Delta})$ is a relative Mori dream space over $Z$. If this is the case, then $\mathbb{T}_{X_1} \cong \mathbb{T}_{\overline{X}} \times \mathbb{T}_0$ and there exist $N_{X_1}$, $\chi_{X_1}$, $N_{\overline{X}}$, and $\chi_{\overline{X}}$, such that 
    \[
    {\rm Cox}(X_1/Z,\Delta_1)_{N_{X_1}, \chi_{X_1}} \cong {\rm Cox}(\overline{X}/Z,\overline{\Delta})_{N_{\overline{X}}, \chi_{\overline{X}}}.
    \]
\end{enumerate}
In particular, if {\rm (2)} holds, we have a commutative diagram, where dashed arrows denote good quasi-torus quotients of big open subsets
\[
\xymatrix@R=30pt@C=30pt{ 
\overline{X}_1=\overline{\overline{X}} \ar@{-->}[dr]^{/\mathbb{T}_{\overline{X}}} \ar@{-->}[ddr]_{/\mathbb{T}_{X_1}}\\
& \overline{X}=\overline{Y} \ar[r]^{/\mathbb{T}_Y} \ar@{-->}[d]^{/\mathbb{T}_{0}} \ar@{-->}[dr]^{/\mathbb{T}_X}& Y\ar@{-->}[d] \\ 
&X_1 \ar[dr] \ar[r] & X \ar[d] \\
& & Z.
}
\]

\end{lemma}

\begin{corollary}
\label{cor:CoxCoxFano}
Under the assumptions of Lemma~\ref{le:CoxCox}.
Assume that  $(X,\Delta)$ is of Fano type over $Z$. Then:
\begin{enumerate}
    \item $(X_1,\Delta_1)$ is of Fano type over $Z$.
    \item $\overline{X}$ has Gorenstein canonical singularities.
    \item There is a boundary $\overline{\Delta}$  on $\overline{X}$, such that $(\overline{X},\overline{\Delta})$ is a relative Mori dream space over $Z$ and its Cox ring coincides with the Cox ring of $(X_1,\Delta_1)$.
\end{enumerate}
\end{corollary}

\begin{proof}
The first item follows from the proof of Theorem~\ref{thm:klt-Cox-ring}. The second item follows by the same considerations as in the proof of~\cite[Theorem 1]{Bra19}, with the following two differences. Firstly, $Y$ is in general only $\qq$-factorial over $Z$, so we only get Gorensteinnness locally. Secondly, Cox rings are not unique but involve a choice of subgroup $N \subseteq \WDiv$ and $\chi$. Thus, for some index one cover $\tilde{Y} \to Y$ of $K_Y$ and some choice of $N_Y \subseteq \WDiv(Y,\Delta)$ and $\chi_Y$, the Cox construction $\overline{Y} \to Y$ factors through $\tilde{Y} \to Y$. Indeed, the index one cover is cyclic quasi-\'etale and thus a quotient presentation in the sense of~\cite[Sec 4.2.1]{ADHL15}. Thus, $\overline{Y}$ is Gorenstein. Since it is klt type by Theorem~\ref{thm:klt-Cox-ring}, it is canonical.
The third item follows from Lemma~\ref{le:CoxCox}, {\rm (2)} and the fact that the relative Fano type $(X_1,\Delta_1)$ is a  relative Mori dream space over $Z$.
\end{proof}

\begin{remark}
{\em Note that in order to ensure that $\overline{X}$ is Mori Dream, in general it does not suffice that it is Gorenstein canonical. The essential property is that $(X_1,\Delta_1)$ is of Fano type relative over the base.
}
\end{remark}

\begin{proof}[Proof of Lemma~\ref{le:CoxCox}]
We start by defining the boundaries on $Y$ and $\overline{X}$. Since $\overline{X} \to X_1$ and $Y \to X$ are locally trivial torus bundles in codimension one, we can uniquely pullback Weil divisors by first restricting to the smooth locus, pulling back via usual pullback of Cartier divisors and finally taking the closure, (see, e.g.,~\cite[Rem 1.3.4.1]{ADHL15}). Hence, we can define $\Delta_Y$ and $\overline{\Delta}$ to be the pullbacks of $\Delta$ and $\Delta_1$, respectively. Moreover, the divisor
$\overline{\Delta}$ is the unique divisor so that $K_{\overline{X}}+\overline{\Delta}$
is the log pull-back of $K_Y+\Delta_Y$.

Now, we argue that ${\rm Cox}(X/Z,\Delta)_{N,\chi}$ together with the coarsened $A$-grading is the Cox ring of $(Y,\Delta_Y)$ over $Z$. 
First, since ${\rm Cox}(X/Z,\Delta)_{N,\chi}$ is the Cox ring of $(X,\Delta)$ over $Z$, it is factorially $\Cl(X/Z,\Delta)$-graded by~\cite[Theorem 1.5.3.7]{ADHL15}. 
Thus by~\cite[Theorem 1.5]{Bech12}, it is also factorially $A$-graded. Since $\overline{X} \to X$ is the characteristic space of $(X,\Delta)$ over $Z$, the characteristic quasi-torus $\mathbb{T}_X$ acts log-strongly stably on $(\overline{X},\overline{\Delta})$, thus the subgroup $A$ acts log-strongly stably as well.
Altogether, by Theorem~1.6.4.3 and Corollary~1.6.4.4 of~\cite{ADHL15}, we get that $\Cl(Y/Z,\Delta_Y) \cong A$ and ${\rm Cox}(X/Z,\Delta)_{N,\chi}$ is a Cox ring for $(Y,\Delta_Y)$ over $Z$. The choice of $N_Y \subseteq \WDiv(Y,\Delta)$ and $\chi_Y$ is as follows: we can identify $N \subseteq \WDiv(X,\Delta)$ with a subgroup of $\WDiv^{\mathbb{T}_0}(Y,\Delta_Y)$. For $N_Y$, we take the subgroup mapping to the torsion part of $\Cl(X/Z,\Delta)$, while $\chi_Y$ is the restriction of $\chi$ to this subgroup concatenated with the inclusion $\cc(X,\Delta)^* \hookrightarrow \cc(Y,\Delta_Y)^*$. So the first item from the Lemma follows.

We prove the second item. We already defined the boundary $\overline{\Delta}$ on $\overline{X}$. First assume $(\overline{X},\overline{\Delta})$ is a relative Mori dream space over $Z$. Then, for a choice $N_{\overline{X}} \subseteq \WDiv(\overline{X},\overline{\Delta})$, we can assume that $N_{\overline{X}}$ is a subgroup of $\WDiv^{\mathbb{T}_0}(\overline{X},\overline{\Delta})$. Moreover, it is the direct product of a subgroup mapping isomorphically to the free part of $\Cl(\overline{X}/Z,\overline{\Delta})$ and a subgroup mapping to its torsion part.
Secondly, for a choice of $\chi_{\overline{X}}$, we can assume that $\chi_{\overline{X}}$ maps to $\cc(X_1,\Delta_1)^* \subseteq \cc(\overline{X},\overline{\Delta})^*$.

Then, the Cox ring  ${\rm Cox}(\overline{X}/Z,\overline{\Delta})_{N_{\overline{X}}, \chi_{\overline{X}}}$ is factorially $\Cl(\overline{X}/Z,\overline{\Delta})$-graded. Invoking~\cite[Theorem 1.5]{Bech12} as above, we see that it is also factorially $\Cl(\overline{X}/Z,\overline{\Delta}) \times \zz^{\dim(\mathbb{T}_0)}$-graded. Moreover, the action of $\mathbb{T}_{\overline{X},0} \times \mathbb{T}_0$ on $\overline{\overline{X}}$ is strongly stable, since the actions of $\mathbb{T}_{\overline{X},0}$  and $\mathbb{T}_0$ on $\overline{\overline{X}}$ and  $\overline{X}$, respectively, are so. Thus ${\rm Cox}(\overline{X}/Z,\overline{\Delta})_{N_{\overline{X}}, \chi_{\overline{X}}}$ is indeed a Cox ring for $(X_1,\Delta_1)$ over $Z$.
In particular, \[
\Cl(X_1/Z,\Delta_1) \cong \Cl(\overline{X}/Z,\overline{\Delta}) \times \zz^{\dim(\mathbb{T}_0)}.
\]
The choice of $N_{X_1} \subseteq \Cl(X_1/Z,\Delta_1)$ and $\chi_{X_1}$ is as follows: for $N_{X_1}$ take the direct product of $N_{\overline{X}}$ as a subgroup of $\WDiv(X_1/Z,\Delta_1)$ and an arbitrary subgroup mapping isomorphically to the $\zz^{\dim(\mathbb{T}_0)}$-part of $\Cl(X_1/Z,\Delta_1)$. Hence, we can identify the kernel of $N_{X_1} \to \Cl(X_1/Z,\Delta_1)$ with the kernel of $N_{\overline{X}} \to  \Cl(\overline{X}/Z,\overline{\Delta})$.Thus, $\chi_{\overline{X}}$ from above can be taken to define $\chi_{X_1}$.

The arguing in the other direction, i.e., when $(X_1,\Delta_1)$ is a Mori dream space over $Z$, is analogous to the proof of the first item.
This concludes the proof.
\end{proof}

\begin{definition}
\label{def:cox-iteration}
{\em
Let $\phi \colon X \to Z$ be a contraction (or an aff-contraction) and $(X,\Delta)$ a relative Mori dream space over $Z$. 
We denote $\mathbb{T}^1:=\mathbb{T}_X$ with torus part $\mathbb{T}^1_0$ and finite abelian part $A^1$ respectively.
We define 
\[
{\rm Cox}^{(1)} (X/Z,\Delta):={\rm Cox} (X/Z,\Delta), 
\text{ }
\overline{X}_1:=\overline{X}=\Spec  {\rm Cox} (X/Z,\Delta),
\text{ and }
\overline{\Delta}_1:=\overline{\Delta}.
\] 
We iteratively define ${\rm Cox}^{(i)} (X/Z,\Delta)$ as follows. 
Assume $(\overline{X}_{i-1},\overline{\Delta}_{i-1})$ is a relative Mori dream space over $Z$. Then, we set
\[
{\rm Cox}^{(i)} (X/Z,\Delta):={\rm Cox} (\overline{X}_{i-1}/Z,\overline{\Delta}_{i-1}),
\text{ }
\overline{X}_{i}:= \Spec {\rm Cox}^{(i)} (X/Z,\Delta),
\text{ and }
\mathbb{T}^{i}:=\mathbb{T}_{\overline{X}_{i-1}}=\mathbb{T}^{i}_{0} \times A^{i}.
\]
We let $\overline{\Delta}_{i}$ be the log-pullback of $\overline{\Delta}_{i-1}$. Then, we call the (possibly infinite) chain
\[
\xymatrix@R=30pt@C=30pt{ 
\cdots \ar@{-->}[r] & (\overline{X}_3,\overline{\Delta}_3)\ar[rrrd] \ar@{-->}[r]^{/\mathbb{T}^3} & (\overline{X}_2,\overline{\Delta}_2)\ar[rrd] \ar@{-->}[r]^{/\mathbb{T}^2} & (\overline{X}_1,\overline{\Delta}_1) \ar[rd] \ar@{-->}[r]^{/\mathbb{T}^1} & (X,\Delta) \ar[d] \\
&&&&
Z
}
\]
the {\em iteration of Cox rings of $(X,\Delta)$ over $Z$}. If $\Cl(\overline{X}_{i}/Z,\overline{\Delta}_{i})$ is trivial for some $i\geq 1$, we say that $(X,\Delta)$ has finite iteration of Cox rings over $Z$.
If the iteration stabilizes for some $k$, i.e., the ring is eventually factorial,
then we denote by ${\rm Cox}^{\rm it}(X/Z,\Delta)$ the isomorphism class over $Z$ of this ring.
The ring ${\rm Cox}^{\rm it}(X/Z,\Delta)$ is called the {\em iteration of Cox rings} or the {\em master Cox ring}.
}
\end{definition}

\begin{remark}
{\em
In the case that $Z$ is local, essentially of finite type or Henselian, in the above definition we iterate the gr-local or gr-Henselian Cox rings. In each step, we can also localize (and Henselize respectively) $\overline{X}_{i-1}$ at the unique graded maximal ideal and take the Cox ring of such spectrum. However, by Lemma~\ref{lem:pic0} and Theorem~\ref{thm:Cl-gr-Hens}, the class groups of these spaces agree. Hence, the iteration defined in such way is compatible to the iteration defined above. Indeed, the localization (and Henselization respectively) of $\overline{X}_{i}$ will always yield the same spectrum.
}
\end{remark}

\begin{remark}
\label{rem:fin-Cox-covers}
{\em
By Lemma~\ref{le:CoxCox}, $(\overline{X},\overline{\Delta})$ is Mori Dream over $Z$ if and only if $(X_1,\Delta_1)$ is so. Thus, the iteration of Cox rings induces a chain of finite abelian Galois covers $(X_i,\Delta_i) \xrightarrow{/A_i} (X_{i-1},\Delta_{i-1})$, where $\Delta_i$ is the log-pullback of $\Delta_{i-1}$. In particular, the characteristic quasi-tori satisfy 
\[
\mathbb{T}_{X_i} \cong \mathbb{T}^{i} \times \mathbb{T}_{X_{i-1}}^0.
\] 
We get the following commutative diagram:
\[
\xymatrix@R=25pt@C=20pt{
\ddots \ar@{-->}[dr] \ar@{-->}[dddr] \ar@{-->}[ddd] \\
&(\overline{X}_{2},\overline{\Delta}_{2}) \ar@{-->}[dd] \ar@{-->}[dr] \ar@{-->}[ddr]\\
&& (\overline{X}_{1},\overline{\Delta}_{1}) \ar@{-->}[d] \ar@{-->}[dr]\\ 
\cdots \ar[r] \ar[drrr] &(X_2,\Delta_2) \ar[drr] \ar[r] &(X_1,\Delta_1) \ar[dr] \ar[r] & (X,\Delta) \ar[d] \\
&& & Z.
}
\]
In particular, Corollary~\ref{cor:CoxCoxFano} tells us that if $(X,\Delta)$ is of Fano type over $Z$, so is $(X_i,\Delta_i)$ for any $i\geq 1$. Hence, the $i$-th iterated Cox ring and $(\overline{X}_{i},\overline{\Delta}_{i})$ is defined for any $i\geq 1$. The question remains if the iteration stabilizes or not. And, if yes, if there is any bound on the number of iteration steps. We answer these questions in Section~\ref{subsec:bounded-iteration}.
}
\end{remark}

We finish the present subsection by showing that the actions of the characteristic quasi-tori can be lifted to the iterated total coordinate spaces. Moreover, they induce an action of a solvable reductive group. This generalizes observations made in~\cite{ABHW18, Bra19}. 
We start with the following lemma slightly generalizing~\cite[Theorem 5.1]{AG10}. This lemma covers the lifting of automorphisms to the Cox ring of  relative Mori dream spaces, affine over the base.

\begin{lemma}
\label{le:aff-lift-aut}
Let $(X,\Delta)$ be a relative Mori dream space, affine over $Z$ and $\Aut_Z(X)$ the automorphism group of $X$ over $Z$. Denote by $\Aut^\mathbb{T}(\overline{X})$ the normalizer of the characteristic quasi-torus $\mathbb{T}$ in the automorphism group $\Aut_Z(\overline{X})$ of $\overline{X}:= \Spec {\rm Cox}(X/Z,\Delta)_{N,\chi}$. Then there is a short exact sequence:
    \[
    \xymatrix{
    1 \ar[r] & \mathbb{T} \ar[r] &  \Aut^\mathbb{T}(\overline{X}) \ar[r] & \Aut_Z(X) \ar[r] & 1.
    }
    \]
The above short exact sequence is called a lifting of ${\rm Aut}_Z(X)$ to ${\rm Aut}_Z(\overline{X})$.
\end{lemma}

\begin{proof}
The proof is analogous to the one of~\cite[Theorem 5.1]{AG10}. We need to argue the surjectivity of the map $\Aut^\mathbb{T}(\overline{X}) \to \Aut_Z(X)$ since we have non-isomorphic Cox rings depending on the choice of $N \subseteq \WDiv(X,\Delta)$ and $\chi$. An automorphism $\psi \in \Aut_Z(X)$ induces an automorphism $\psi$ of $\WDiv(X,\Delta)$. Observe that $\psi$ maps the kernel of the surjective map $  N \to \Cl(X/Z,\Delta)$ to the kernel of the surjective map $ \varphi \colon \psi(N) \to \Cl(X/Z,\Delta)$ and induces a character $\psi^*(\chi)\colon \ker(\varphi) \to \cc(X)^*$. Hence, ${\rm Cox}(X/Z,\Delta)_{N,\chi}$ and ${\rm Cox}(X/Z,\Delta)_{\psi(N),\psi^*(\chi)}$ are isomorphic. Moreover, fixing an isomorphism $\tau$ between them, $\psi^* \circ \tau$ is an element of $\Aut^\mathbb{T}(\overline{X})$ mapping to $\psi$.
This finishes the proof of the lemma.
\end{proof}

The following proposition explains how to lift automorphisms to the Cox ring of a relative Mori dream space that is projective over the base. Here, we have to distinguish between automorphisms on the total coordinated space $\overline{X}$ and its big open subset $\hat{X}$. Following~\cite[Sec 4.2.4]{ADHL15}, we denote by ${\rm Bir}_{2,Z}(X)$ the \emph{weak automorphisms} of $X$ over $Z$, namely birational maps $X \dashrightarrow X$ being regular isomorphisms  in codimension one over $Z$.

\begin{proposition}
\label{prop:proj-lift-aut}
Let $(X,\Delta)$ be a relative Mori dream space, projective over $Z$.
Let $\Aut_Z(X)$ be the automorphism group of $X$ over $Z$. 
%Denote by $\Aut^0_Z(X)$ the identity component and by $\pi\Aut_Z(X)$ the group of components of  $\Aut_Z(X)$. 
Then, the following statements hold:
\begin{enumerate}
    \item $\Aut_Z(X)$ is a linear algebraic group.
     \item There is a divisor $L$ on $X$, ample over $Z$, with relative section ring $R_L$ and spectrum $\tilde{X}:=\Spec R_L$, such that there is a short exact sequence 
    \[
    \xymatrix{
    1 \ar[r] & \cc^* \ar[r] &  \Aut^{\cc^*}(\tilde{X}) \ar[r] & \Aut_Z(X) \ar[r] & 1.
    }
    \]
    \item If $(X,\Delta)$ is of Fano type over $Z$, then for $L:=-(K_X+\Delta)$, the spectrum $\tilde{X}$ of the section ring $R_L$ allows an action of these $Z$-automorphisms $\Aut_Z(X,\Delta)$ leaving $\Delta$ invariant, as a subgroup of  $\Aut_Z(\tilde{X})$. In particular, there is a split short exact sequence
    \[
    \xymatrix{
    1 \ar[r] & \cc^* \ar[r] &  \Aut^{\cc^*}(\tilde{X}) \ar[r] & \Aut_Z(X,\Delta) \ar[r] & 1.
    }
      \]
    \item There is a commutative diagram with exact sequences as rows and vertical inclusions of finite index:
    \[
    \xymatrix{
    1 \ar[r] & \mathbb{T} \ar[r] &  \Aut^\mathbb{T}(\overline{X}) \ar[r] & {\rm Bir}_{2,Z}(X) \ar[r] & 1 \\
    1 \ar[r] & \mathbb{T} \ar[r] \ar@{=}[u] &  \Aut^\mathbb{T}(\hat{X}) \ar[r] \ar@{^{(}->}[u] & \Aut_Z(X) \ar@{^{(}->}[u] \ar[r] & 1.
    }
    \]
  \end{enumerate}
\end{proposition}

\begin{proof}
Since $(X,\Delta)$ is relative Mori Dream over $Z$, the cone of nef divisors relative to $Z$ is rational polyhedral. Its rays are permuted by the group of components of $\Aut_Z(X)$. Taking the sum of the ray generators, we get the relatively ample $\Aut_Z(X)$-invariant class $[L]$. By~\cite[Theorem 2.16]{Bri19}, $\Aut_Z(X)$ is a linear algebraic group. 

Now, $\Aut_Z(X)$ may not stabilize the divisor $L$, but only its class.
We argue as in the proof of Lemma~\ref{le:aff-lift-aut}. Since $L$ is ample, its relative stable base locus is empty (see, e.g.,~\cite[Sec 2.3]{Bri19}). Thus, we obtain an isomorphism $\Aut^{\cc^*}(L) \cong \Aut^{\cc^*}(\tilde{X})$. Hence, we get the short exact sequence of the second item. The statement about $L:=-(K_X+\Delta)$ in the case of a Fano type follows since $K_X$ is invariant under the action of $\Aut_Z(X)$. The proof of the last item is analogous to the one of~\cite[Theorem 4.2.4.1]{ADHL15}.
\end{proof}

In what follows, 
we aim to lift the whole characteristic quasi-torus action to the iterated Cox rings. In this way, we will produce an action of a solvable reductive group. The derived normal series of this solvable reductive group reflects the iteration of Cox rings. 
We denote the $k$-th derived subgroup of a group $G$ by $\mathcal{D}^G_k:=[\mathcal{D}^G_{k-1},\mathcal{D}^G_{k-1}]$, where $\mathcal{D}^G_0:=G$.

An important property of Cox rings is that their spectra dominate all \emph{quotient presentations} in the sense of~\cite[Sec 4.2.1]{ADHL15}. This means, all good quasi-torus quotients $Y \to X = Y / H$, such that the action of $H$ is strongly stable. In analogy to the notion of quasi-\'etale covers, we call them \emph{abelian quasi-torsors} in the followng. In the classical setting, where $X$ is assumed to have only constant invertible global functions, the quasi-torsors are assumed to have only constant invertible $H$-homogeneous functions. Thus in our setting, we have two major differences: firstly, a priori, we have invertible non-constant functions on $X$, which means that we have non-isomorphic Cox rings depending on the choice of $N \subseteq \WDiv(X)$ and the character $\chi$. This means that an abelian quasi-torsor $Y \to  X$ is dominated by $\overline{X}_{N,\chi} \to X$ for some choice
of $N$ and $\chi$. Secondly, what we have to impose is not that invertible $H$-homogeneous functions are constant, but that they descend to $X$. This property is fulfilled e.g. if there is at least one maximal homogeneous ideal in $\mathcal{O}(Y)$. This is the case in all situations relevant for us, e.g. finite coverings of relative Fano type $(X,\Delta)$ or (iterated) Cox rings. However, even if $X$ has only constant invertible functions, there may be non-constant non-homogeneous invertible functions in the Cox ring (see, e.g.,~\cite[Ex 1.4.4.2]{ADHL15}). The precise definition of an abelian quasi-torsor in our setting is the following.

\begin{definition}
\label{def:quot-pres}{\em 
Let $(X,\Delta)$ be a relative Mori dream space  over $Z$. Let $Y \to X= Y/H$ be a good quotient by a quasi-torus $H$. We call $\varphi\colon Y \to X$ an {\em
abelian quasi-torsor}, if the following are satisfied.
\begin{enumerate}
    \item Let $H_0$ be the identity component and $A$ be the group of components of $H$. Then the finite abelian cover $Y':=Y/H_0 \xrightarrow{/A} X$ is log quasi-\'etale over $(X,\Delta)$ with log-pullback $\Delta_{Y'}$ of $\Delta$.
    \item There are big open subsets $U_{Y'} \subseteq Y'$ and $U_Y :=\varphi^{-1}(U_{Y'}) \subseteq Y$ such that the restriction 
    \[
    \left. \varphi \right|_{U_Y}\colon U_Y \xrightarrow{/H_0} U_{Y'}
    \]
    is an \'etale locally trivial $H_0$-bundle. In particular, the action of $H_0$ on $Y$ is strongly-stable. 
    \item Global invertible homogeneous functions on $Y$ descend to $X$ via the induced homomorphism  $\mathcal{O}(Y)^H \cong \mathcal{O}(X) \hookrightarrow \mathcal{O}(Y)$.
\end{enumerate}
In the case that $H$ is a torus, then we may say that it is a {\em torus quasi-torsor}.
Whenever the quasi-torus $H$ is clear from the context,
we may just say that $Y\rightarrow X$ is a {\em quasi-torsor}.

Let $(X,\Delta;x)$ be a klt singularity.
We say that $(Y,y)$ is a {\em pointed abelian quasi-torsor} of $(X,\Delta;x)$ if there
exists an abelian quasi-torsor $Y\rightarrow X$ so that the image of $y$ in $X$ equals $x$. To shorten notation, we may say that
$Y\rightarrow X$ is an {\em abelian pointed cover}.
Observe that if $Y\rightarrow X$ is an abelian pointed cover, 
then the corresponding finite morphism
$Y'\rightarrow X$ is a finite pointed cover.

}
\end{definition}

\begin{proposition}
\label{prop:quot-pres}
Let $(X,\Delta)$ be a relative Mori dream space  over $Z$. Let $Y \to X= Y/H$ be a quasi-torsor. Then, there exists
\begin{itemize}
\item a monomorphism $\mathbb{X}(H)\rightarrow {\rm Cl}(X/Z,\Delta)$, 
\item a subgroup $N_Y \leqslant N\leqslant {\rm WDiv}(X)$, 
\item surjections $\varphi \colon N \to \Cl(X/Z,\Delta)$ and $\left. \varphi \right|_{N_Y} \colon N_Y \to \mathbb{X}(H)$, and 
\item a character $\chi: \ker(\varphi) \to \cc(X)^*$
\end{itemize}
such that the following statements are satisfied:
\begin{enumerate}
    \item $Y \cong \Spec_X \mathcal{S}^{(N_Y)} / \mathcal{I}$, where $\mathcal{I}$ is the ideal subsheaf of $\mathcal{S}^{(N_Y)}$ locally generated by sections $1-\chi(E)$,
    where $E$ runs in $\ker(\left. \varphi \right|_{N_Y})$.
    \item There is a commutative diagram
         \[
    \xymatrix{
 \hat{X}_{N,\chi} \ar[rr]^{/ H'} \ar[dr]^{/\mathbb{T}_X} & & Y \ar[dl]^{/H} \\
 & X
    }
    \]
    where the quasi-torus $H'$ is defined by the exact sequence
    $1\rightarrow H'\rightarrow \mathbb{T}_X \rightarrow H \rightarrow 1$.
\end{enumerate}
\end{proposition}

\begin{proof}
The proof is analogous to the one of~\cite[Theorem 4.2.1.4]{ADHL15}, with the two differences mentioned above. In particular, invoking~\cite[Prop. 1.6.4.5]{ADHL15} and the notation therein, we get the following. Let $M:=\mathbb{X}(H)$ and $E(Y)$ be the multiplicative group of non-zero $M$-homogeneous rational functions on $Y$ with $E(Y)_{w}$ those of degree $w \in M$. Then we have the following diagram of group homomorphisms
\[
    \xymatrix{
E(Y) \ar[rr]^{f \mapsto \divv(f)} && \WDiv(Y)^H \ar@/_/[rr]_{q_*} && \WDiv(X) \ar@/_/[ll]_{q^*}.    }
    \]
As $Y \to X$ is an \'etale locally trivial $H$-bundle in codimension one, the homomorphisms $q_*$ and $q^*$ are inverse to each other.  As in~\cite[Prop. 1.6.4.5]{ADHL15}, but using item (3) of Definition~\ref{def:quot-pres}, the homomorphism $E(Y) \to \WDiv(X)$ induces a monomorphism $M \to \Cl(X/Z,\Delta)$. Thus, we can choose a subgroup $N_Y$ of $ \WDiv(X) \cong \WDiv(Y)^H  $ surjecting onto $M$ and enlarge $N_Y \subseteq N$ such that $\varphi \colon N \to \Cl(X/Z,\Delta)$ is onto. Choosing a character $\chi \colon \ker(\varphi) \to \cc(X)^*$ yields the desired statements together with the rest of the proof of~\cite[Theorem 4.2.1.4]{ADHL15}.
\end{proof}

\begin{corollary}
\label{cor:Cox-it-solv-cover}
Let $(X,\Delta)$ be a relative Mori dream space over $Z$. Assume the $k$-th iterated Cox ring ${\rm Cox}^{(k)} (X/Z,\Delta)$ exists and is of finite type over $Z$. Then $\overline{X}_k$ allows an action of a solvable reductive group $S$ with maximal torus $\mathbb{T}:=\mathbb{T}_{X_{k}}$ and an $S$-invariant big open subset $\hat{X}^k$, such that:
\begin{enumerate}
    \item $X_k \cong \hat{X}^k / \mathbb{T}$ and $X \cong  \hat{X}^k / S $.
    \item $\hat{X}^j \cong  \hat{X}^k / \mathcal{D}_j^S$ and $\mathbb{T}^{j}\cong \mathcal{D}_{j-1}^S/\mathcal{D}_j^S$ for $j \leq k$.
    \item For the finite solvable group $S_{\rm fin}:=S/\mathbb{T}$ and the finite covers $X_j$, the assertions hold analogously, i.e.,
    \[
  X_j\cong X_k / \mathcal{D}_j^{S_{\rm fin}}, \text{ and} \quad A^{j}\cong \mathcal{D}_{j-1}^{S_{\rm fin}}/\mathcal{D}_j^{S_{\rm fin}}.
    \]
\end{enumerate}
\end{corollary}

\begin{proof}
The arguing is analogous to the proof of~\cite[Theorem 1.6]{ABHW18} in the case that $X$ is affine over $Z$, where we use Proposition~\ref{prop:quot-pres} instead of~\cite[Prop. 3.5]{AG10}. If $X$ is projective over $Z$, then we choose a divisor $L$ on $X$ ample over $Z$. By the same argument as in the proof of Lemma~\ref{le:CoxCox} (1), we have ${\rm Cox}(X/Z,\Delta)\cong {\rm Cox}(\tilde{X}/Z,\tilde{\Delta)}$ and thus we can reduce to the relatively affine case. 
\end{proof}

\begin{remark}
{\em
In Definition~\ref{def:quot-pres} (2), it is essential, that not only $U_{Y'} \subseteq Y'$ but also $U_Y \subseteq Y$ is a big open subset. Otherwise $\varphi \colon Y \to Y/H$ may contract divisors. In particular, the existence of the monomorphism  $\mathbb{X}(H) \to \Cl(X)$ from Proposition~\ref{prop:quot-pres} would not hold true in this more general setting. As an example, consider the blowup of $Y^n:={\rm Bl}_0(\mathbb{A}^n) \to \mathbb{A}^n$ at the origin, which has relative Cox ring ${\rm Cox}(Y^n/\mathbb{A}^n) \cong \cc[x_1,\ldots,x_{n+1}]$. Then the induced $\cc^*$-quotient $\mathbb{A}^{n+1} \to \mathbb{A}^n$, given by the weights $(1,\ldots,1,-1)$, is not a quotient presentation. Indeed, the divisor $\{x_{n+1}=0\}$ maps to the origin.
Observe that we have an infinite sequence of $\cc^*$-quotients 
$\mathbb{A}^1 \leftarrow \mathbb{A}^2 \leftarrow \mathbb{A}^3 \leftarrow \dots$. 
This sequence does not contradict Theorem~\ref{introthm-6-univ-scf-cover},
because this covers are not pointed abelian covers in the sense of Definition~\ref{def:quot-pres}.
}
 \end{remark}

\subsection{Regional fundamental group of a relative Fano type variety}\label{subsec:finite-rel-fund}
In this subsection, we prove that the regional fundamental group
of a relative Fano type variety is finite
and it satisfies the Jordan property.

\begin{definition}
{\em 
Let $\phi \colon X \rightarrow Z$ be a projective contraction.
Let $(X,\Delta)$ be a log pair.
Let $z\in Z$ be a closed point.
We define the fundamental group
\[ 
\pi_1^{\rm reg}(X/Z,\Delta;z)
\]
to be the inverse limit of the fundamental groups
\[
\pi_1^{\rm reg}(\phi^{-1}(U), \Delta_U),
\]
where the limit runs through all the open sets $U$ on $Z$ which contain $z$.
In the above, by abuse of notation, we let  $\Delta_U$ to be
the restriction of $\Delta$ to $\phi^{-1}(U)^{\rm reg}$.
}
\end{definition}

\begin{theorem}\label{thm:rel-finiteness}
Let $n$ be a positive integer.
There exists a constant $c(n)$, only depending on $n$, satisfying the following.
Let $\phi\colon X \rightarrow Z$ be a projective contraction so that $X$ has dimension $n$.
Let $(X,\Delta)$ be a log pair of Fano type over $Z$.
Let $z\in Z$ be a closed point.
Then, the fundamental group 
$\pi_1^{\rm reg}(X/Z,\Delta;z)$ is finite.
Furthermore, there exists a normal abelian subgroup $A\leqslant \pi_1^{\rm reg}(X/Z,\Delta;z)$ of rank at most $n$ and index at most $c(n)$.
\end{theorem}

\begin{proof}
The divisor $-(K_X+\Delta)$ is nef and big over $Z$. Hence, it is semiample and big over $Z$, given that $X$ is a relative Mori dream space over $Z$.
Let $X'$ be an ample model of $-(K_X+\Delta)$ over $Z$.
Let $\Delta'$ be the push-forward of $\Delta$ to $X'$.
Then, we have that $(X',\Delta')$ has klt singularities and 
$-(K_{X'}+\Delta')$ is ample over $Z$.
We will prove the statement for
$\pi_1^{\rm reg}(X'/Z,\Delta';z)$. 
Let $Y$ be the orbifold cone with respect to the $\mathbb{Q}$-polarization $-(K_{X'}+\Delta')$ over $Z$, i.e., 
\[
Y:={\rm Spec}
\left( 
\bigoplus_{m\geq 0} 
H^0\left(
X'/Z, \mathcal{O}_X(-m(K_{X'}+\Delta')) 
\right) 
\right).
\]
Note that $\dim(Y)=\dim(X)+1=n+1$.
We have a rational map $\pi\colon Y\dashrightarrow X'$,
which is defined outside a codimension two subset of $Y$.
Let $\Delta_Y$ be the effective divisor so that
$\pi^*(K_{X'}+\Delta')=K_Y+\Delta_Y$.
Then, the pair $(Y,\Delta_Y)$ is klt.
The blow-up of $Y$ at the vertex of the torus action 
is a variety $\tilde{Y}$ which admits a good quotient to $X'$.
The exceptional locus of $\tilde{Y}\rightarrow Y$ is isomorphic to $X'$ and its image in $Y$ is isomorphic to $Z$.
Hence, under this isomorphism, we can consider an embedding $Z\hookrightarrow Y$.
So, we can consider $z\in Y$.
By~\cite[Theorem 1]{Bra19}, we know that the regional fundamental group $\pi_1^{\rm reg}(Y,\Delta_Y;z)$ is finite.
By~\cite[Theorem 2]{BFMS20}, we know that there exists an abelian normal subgroup $A_Y$ of
$\pi_1^{\rm reg}(Y,\Delta_Y;z)$ of rank at most $n+1$ and index at most $c(n+1)$.
Here, $c(n+1)$ is a constant which only depends on $n+1$, hence it only depends on $n$.

Let $U_Z$ be an arbitrary open neighborhood of $z$ in $Z$.
Let $\phi'\colon X'\rightarrow Z$ the associated projective morphism 
and define $U_{X'}:={\phi'}^{-1}(U_Z)$.
We define $U_Y:=\pi^{-1}(U_{X'})$.
For every such $U_Z$, 
we have a short exact sequence:
\[
1\rightarrow \zz_m 
\rightarrow \pi_1^{\rm reg}(U_Y,\Delta_{U_Y}) 
\rightarrow \pi_1^{\rm reg}(U_{X'},\Delta_{U_{X'}})
\rightarrow 1.
\]
Here, $m$ may depend on the chosen neighborhood.
As usual $\Delta_{U_Y}$ (resp. $\Delta_{U_{X'}}$)
is the restriction of $\Delta_Y$ (resp. $\Delta_{X'}$)
to the open set $U_Y$ (resp. $U_{X'})$.
We claim that for certain neighborhood $U$ of $z$ in $Z$,
there is an isomorphism
\begin{equation}\label{eq:iso} 
\pi_1^{\rm reg}(U_Y,\Delta_{U_Y})\cong
\pi_1^{\rm reg}(Y,\Delta_Y;z).
\end{equation} 
Let $U_0$ be an open neighborhood of $z$ in $Y$ which computes the regional fundamental group of the pair $(Y,\Delta_Y)$ at $z$, i.e., there is an isomorphism
\[
\pi_1^{\rm reg}(Y,\Delta_Y;z) 
\cong
\pi_1^{\rm reg}(U_0,\Delta_{U_0}).
\]
Let $U_{Z,0}$ be the inverse image of $U_0$ under
the embedding $Z\hookrightarrow X$.
We define $U_{X',0}:={\phi'}^{-1}(U_{Z,0})$
and $U_{Y,0}:=\pi^{-1}(U_{X',0})$.
Note that $U_{Y,0}$ is homotopic to an analytic open subset
which is contained in $U_0$.
Thus, we conclude that $U_{Y,0}$ satisfies the isomorphism in equation~\eqref{eq:iso}.
Hence, we have an exact sequence
\[
1\rightarrow \zz_m \rightarrow 
\pi_1^{\rm reg}(Y,\Delta_Y;z) \rightarrow \pi_1^{\rm reg}(U_{X',0}, \Delta_{U_{X',0}}) \rightarrow 1. 
\]
Passing to the inverse limit, we have an exact sequence
\[
1\rightarrow \zz_m \rightarrow 
\pi_1^{\rm reg}(Y,\Delta_Y;z) \rightarrow 
\pi_1^{\rm reg}(X'/Z,\Delta;z) \rightarrow 1.
\]
Hence, we conclude that
$\pi_1^{\rm reg}(X'/Z,\Delta;z)$ is finite and satisfies the Jordan property of rank $n+1$, i.e., 
it contains a normal abelian subgroup of rank at most $n+1$ and index at most $c(n)$.
We denote such group by $A_{X'}$.

We claim that $\pi_1^{\rm reg}(X'/Z,\Delta;z)$ actually satisfies the Jordan property of rank $n$.
Let $H'\geq 0$ be an effective divisor which is general in the $\mathbb{Q}$-linear system of $-(K_{X'}+\Delta')$ relative to $Z$.
We may assume that all the coefficients of $H'$ are less than one half.
Then, the pair $(X',\Delta'+H')$ is klt and $\mathbb{Q}$-trivial over the base.
Let $K_Z+H_Z$ be the pair obtained by the canonical bundle formula on $Z$, i.e., 
we have that
\[
K_{X'}+B'+H' \sim_{\mathbb{Q}} {\phi'}^*(K_Z+H_Z).
\]
Since all the coefficients of $H'$ are less than one half, 
there is a natural isomorphism
\[
\pi_1^{\rm reg}(X'/Z,\Delta'+H';z)\rightarrow
\pi_1^{\rm reg}(X/Z,\Delta';z).
\]
Hence, the group on the left hand side contains an abelian subgroup of rank
at most $n+1$ and index at most $c(n)$.
On the other hand, there is an exact sequence
\[
\pi_1^{\rm reg}(F,\Delta_F+H_F) \rightarrow
\pi_1^{\rm reg}(X'/Z,\Delta'+H';z)\rightarrow 
\pi_1^{\rm reg}(Z,H_Z;z)\rightarrow 1.
\]
Let $A_Z$ be the homomorphic image of $A_{X'}$ in the regional fundamental group of $(Z,H_Z)$ at $z$.
Let $A_F$ be the kernel of the surjection $A_{X'}\rightarrow A_Z$.
Hence, we have an exact sequence
\[
A_F \rightarrow A_{X'}\rightarrow A_Z\rightarrow 1.
\]
By~\cite[Theorem 2]{BFMS20}, we know that $A_F$ admits a subgroup $H_F$ of index at most $c(f)$ and rank at most $f$, while
$A_Z$ admits a subgroup $H_Z$ of index at most $c(z)$ and rank at most $z$. Here, $f$ and $z$ are the dimension of $F$ and $Z$ respectively.
We conclude that $A_{X'}$ admits a subgroup $H_{X'}$ of rank at most $\dim(X')=z+f$ and index at most $c(z)+c(f)$.
Note that $c(z)+c(f)$ is bounded by a constant in terms of $\dim(X')$.

Finally, we prove that $\pi_1^{\rm reg}(X/Z,\Delta;z)$ satisfies the Jordan property of rank $n$ with respect to the dimension.
It suffices to prove that there is a surjection
\begin{equation}\label{claim:surj}
\pi_1^{\rm reg}(X'/Z,\Delta';z)
\rightarrow 
\pi_1^{\rm reg}(X/Z,\Delta;z).
\end{equation} 
Indeed, let $U$ be an arbitrary open neighborhood of $z$ in $Z$.
Let $U_X$ be its pre-image in $X$ and $U_{X'}$ be its pre-image in $X'$.
Then, there is a natural surjection
\[
\pi_1^{\rm reg}(U_{X'},\Delta_{U_{X'}})
\rightarrow 
\pi_1^{\rm reg}(U_X,\Delta_{U_X}).
\]
Indeed, the image of the exceptional locus of
$U_{X}\rightarrow U_{X'}$ is a union of closed subsets which are
either contained in the singular locus of $U_X$ or
codimension two subsets of the smooth locus. 
Since $\pi_1^{\rm reg}(X'/Z,\Delta';z)$ surjects into each 
of the fundamental groups $\pi_1^{\rm reg}(U_{X'},\Delta_{U_{X'}})$, 
we conclude that for each $U$ there is a surjective homomorphism
\[
\pi_1^{\rm reg}(X'/Z,\Delta;z) 
\rightarrow 
\pi_1^{\rm reg}(U_X,\Delta_{U_X}).
\]
Taking inverse limit, we conclude that the surjection~\eqref{claim:surj} holds.
Hence, $\pi_1^{\rm reg}(X/Z,\Delta;z)$ is finite and contains a normal abelian subgroup of rank at most $n$ and index at most $c(n)$.
This completes the proof.
\end{proof}

\begin{corollary}
\label{cor:abelianization}
Let $(X,\Delta)$ be of relative Fano type over $Z$. 
Let $A:= \mathbb{T}_X /  \mathbb{T}_X^0 \cong \Cl(X/Z,\Delta;z)_{\rm tor}$ be the group of components of the characteristic quasi-torus of $(X,\Delta)$ over $Z$ at $z$. Then $A$ is the abelianization of $\pi_1^{\rm reg}(X/Z,\Delta;z)$, i.e.,
\[
A \cong \pi_1^{\rm reg}(X/Z,\Delta;z) / [\pi_1^{\rm reg}(X/Z,\Delta;z),\pi_1^{\rm reg}(X/Z,\Delta;z)],
\]
where $[\pi_1^{\rm reg}(X/Z,\Delta),\pi_1^{\rm reg}(X/Z,\Delta)]$ is the commutator subgroup of $\pi_1^{\rm reg}(X/Z,\Delta)$.
\end{corollary}

\begin{proof}
By Theorem~\ref{thm:rel-finiteness}, we know that  $\pi_1^{\rm reg}(X/Z,\Delta;z)$ is finite and the \'etale fundamental group of $X_{\rm reg}^h$. Here, $X^h \to Z^h$ is the base change to the Henselization of $Z$ at $z$. Thus, we can assume $Z$ is local Henselian. 
Then $G:=\pi_1^{\rm reg}(X/Z,\Delta;z)$ induces a finite log quasi-\'etale Galois cover of $(X,\Delta)$, which by abuse of notation, we denote by 
\[
(\tilde{X},\tilde{\Delta}) \xrightarrow{/G} (X,\Delta).
\]
Then $[G,G]$ is a normal subgroup of $G$ and induces an abelian log quasi-\'etale Galois cover 
\[
(Y,\Delta_Y) \xrightarrow{/(G/[G,G])} (X,\Delta),
\]
which is a quasi-torsor in the sense of Definition~\ref{def:quot-pres}. Indeed, $G/[G,G]$ is finite and $(Y,\Delta_Y)$ is of Fano type over $Z$. In particular, invertible functions of $Y$ and $X$ are invertible functions of $Z$. Thus by Proposition~\ref{prop:quot-pres}, $G/[G,G]$ is a subgroup of $\Cl(X/Z,\Delta)$. In particular, it is a subgroup of $A$. 

But since $A$ induces a log quasi-\'etale abelian Galois cover of $(X,\Delta)$ as well, we have $A=G/[G,G]$. Otherwise, there would be a normal subgroup of $G$ smaller than $[G,G]$ with abelian quotient, which is a contradiction. 
This finishes the proof of the corollary.
\end{proof}

\subsection{Boundedness of the iteration of Cox rings}
\label{subsec:bounded-iteration}

In this subsection, we prove the main theorem of this article,
the boundedness of the iteration of Cox rings for Fano type varieties.

\begin{theorem}\label{thm:bounded-iteration}
There exists a constant $k(n)$, only depending on $n$, 
satisfying the following.
Let $\phi\colon X \rightarrow Z$ be a projective contraction so that $X$ has dimension $n$. Let $(X,\Delta)$ be a log pair of Fano type over $Z$.
Then, ${\rm Cox}^{(k)}(X/Z, \Delta)$ stabilizes for $k\geq k(n)$.
\end{theorem}

\begin{proof}
Firstly, we prove that the iteration of Cox rings
\[
{\rm Cox}^{(k)}(X/Z,\Delta)
\]
stabilizes for $k$ large enough.
It suffices to show that $\Cl(\overline{X}_k/Z,\overline{\Delta}_k)$ is torsion-free for some $k \in \mathbb{N}$, since then $(\overline{X}_{k+1},\overline{\Delta}_{k+1})$ is factorial over $Z$, see, e.g.,~\cite[1.4.1.5]{ADHL15}. By Remark~\ref{rem:fin-Cox-covers}, torsion-freeness of $\Cl(\overline{X}_k/Z,\overline{\Delta}_k)$ is equivalent to torsion-freeness of the divisor class group of the finite abelian covering space $(X_k/Z,\Delta_k)$. 

Now, we assume that 
\[
\Cl(X_k/Z,\Delta_k)_{\rm tor} \not\cong 1
\]
for any $k \in \mathbb{N}$. Then by Corollary~\ref{cor:Cox-it-solv-cover} (3), there is an infinite chain of log quasi-\'etale finite solvable Galois covers 
\[
(X_k,\Delta_k) \xrightarrow{/S_{k}} (X,\Delta),
\]
where $\left|S_{k+1} \right| > \left| S_k \right| $. This is a contradiction to the finiteness of $\pi_1^{\rm \reg}(X/Z,\Delta)$. Thus ${\rm Cox}^{(k)}(X/Z,\Delta)$ stabilizes 
for $k$ large enough.

Now, we show that such bound at which the iteration of Cox rings stabilizes admits an upper bound which only depends on the dimension of $X$. As before, we denote $\mathcal{D}_0:=\pi_1^{\rm \reg}(X/Z,\Delta)$ and inductively we define $\mathcal{D}_i:=[\mathcal{D}_{i-1},\mathcal{D}_{i-1}]$. By Theorem~\ref{thm:rel-finiteness}, we know that there is an exact sequence $1 \to A_0 \to \mathcal{D}_0 \to N_0 \to 1$, where $A_0$ is an abelian normal subgroup  of rank at most $n$ and $N_0$ has order at most $c(n)$. We denote $A_{i}:=\mathcal{D}_i \cap A_0$. Now, we have a commutative diagram
\[
    \xymatrix{
    & 1 \ar[d] & 1 \ar[d] & 1 \ar[d] \\
    1 \ar[r] & A_{i+1} \ar[r] \ar[d] & A_i \ar[r] \ar[d] & B_i \ar[r] \ar[d] & 1 \\
    1 \ar[r] & \mathcal{D}_{i+1} \ar[r] \ar[d] & \mathcal{D}_{i} \ar[r] \ar[d] & \Cl(X_i/Z,\Delta_i)_{\rm tor} \ar[r] \ar[d] & 1 \\
    1 \ar[r] & N_{i+1} \ar[r] \ar[d] & N_{i} \ar[r] \ar[d] & M_i \ar[r] \ar[d] & 1 \\
    & 1 & 1 & 1
    }
\]
with exact rows and columns for each $i \geq 0$. In particular, we get a chain of normal subgroups $\cdots \trianglelefteq N_2 \trianglelefteq N_1 \trianglelefteq N_0$ with $M_i:=N_{i+1}/N_i$. If in such a chain, no two consecutive $M_i$ are trivial, then the length $k$ of the chain is bounded by $2\log_2(c(n))$. Hence, we know that there is some $j  \leq 2\log_2(c(n))+1$ with $M_j\cong M_{j+1}\cong 1$. 
We have a commutative diagram
\[
    \xymatrix{
    & 1 \ar[d] & 1 \ar[d] & 1 \ar[d] \\
    1 \ar[r] & A_{j+2} \ar[r] \ar[d] & A_j \ar[r] \ar[d] & C \ar[r] \ar[d] & 1 \\
    1 \ar[r] & \mathcal{D}_{j+2} \ar[r] \ar[d] & \mathcal{D}_{j} \ar[r] \ar[d] & S \ar[r] \ar[d] & 1 \\
    1 \ar[r] & N_{j+2} \ar[r] \ar[d] & N_{j} \ar[r] \ar[d] & L \ar[r] \ar[d] & 1 \\
    & 1 & 1 & 1
    }
\] 
with exact rows and columns similar to the one from above. But here $L$ is trivial, since $N_{j+2}\cong N_{j+1} \cong N_j$. The group $C$ is abelian, since $A_j$ and $A_{j+2}$ are abelian. So $S \cong C$ is abelian. Thus $\mathcal{D}_{j+2}$ equals $\mathcal{D}_{j+1}$, the derived subgroup of $\mathcal{D}_j$. But then $\Cl(X_{j+1}/Z,\Delta_{X_{j+1}})_{\rm tor}\cong \mathcal{D}_{j+1}/\mathcal{D}_{j+2}$ is trivial and the iteration of Cox rings stabilizes for $k \geq j+2$. Since $j  \leq 2\log_2(c(n))+1$, we can set $k(n):= 2\log_2(c(n))+3$. Since $c(n)$ only depends on $n$, the proof is finished.
\end{proof}

\section{Simply connected factorial canonical cover}\label{sec:scfc}

In this section, we aim to prove the existence of a simply connected factorial canonical cover for klt singularities.
In Subsection~\ref{subsec:existence-scfc}, we prove the existence of the scfc cover.
In Subsection~\ref{subsec:universality-scfc}, we prove that the scfc cover dominates any sequence of finite covers and abelian covers.
In Subsection~\ref{subsec:upper-bound-dim}, we give an upper bound for the dimension of the iteration of Cox rings of the singularity.

\subsection{Existence of the simply connected factorial canonical cover}\label{subsec:existence-scfc}

In this subsection, we prove the existence of a simply connected factorial canonical cover for a klt singularity. The following proposition is the cornerstone for the construction.

\begin{proposition}
\label{prop:fiber-prod-fincov-and-quotpres}
Let $(X,\Delta)$ be a log pair.
Let $\phi\colon X\rightarrow Z$ be an aff-contraction
where $Z$ is local Henselian.
Assume that $(X,\Delta)$ is of relative Fano type over $Z$. Let $(X',\Delta') \to (X,\Delta)$ be a
torus quasi-torsor in the sense of Definition~\ref{def:quot-pres}.
Let $\mathbb{T}$ be the acting torus on $X'$.
Let $(\tilde{X},\tilde{\Delta}) \to (X,\Delta)$ be the finite log quasi-\'etale Galois cover associated to $\pi:=\pi_1^{\rm reg}(X/Z,\Delta)$.  

Define $\tilde{X}':= \tilde{X} \times_{X} X'$. Then, there is a commutative diagram
\[
\xymatrix{
(\tilde{X}',\tilde{\Delta}') \ar[r]^{/\pi} \ar[d]^{/\mathbb{T}} & (X',\Delta') \ar[d]^{/\mathbb{T}} \\
(\tilde{X},\tilde{\Delta}) \ar[r]^{/\pi} & (X,\Delta)
}
\]
where $\pi_1^{\rm reg}(X'/Z,\Delta') \cong \pi_1^{\rm reg}(X/Z,\Delta)$ and $\tilde{X}'$ allows a $\mathbb{T} \times \pi_1^{\rm reg}(X'/Z,\Delta')$-action
satisfying the following conditions:
\begin{enumerate}
    \item $\tilde{X}' \to X'$ is the finite log quasi-\'etale Galois cover associated to $\pi_1^{\rm reg}(X'/Z,\Delta')$ and $\tilde{\Delta}'$ is the log-pullback of $\Delta'$.
    \item $(\tilde{X}',\tilde{\Delta}') \to (\tilde{X},\tilde{\Delta})$ is a $\mathbb{T}$-quasi-torsor.
    \item The $\mathbb{T}\times \pi_1^{\rm reg}(X'/Z,\Delta')$-action is log-free in codimension one.
    \item $\mathbb{T}$ is the same torus as in the quasi-torsor $X'\rightarrow X$.
\end{enumerate}
\end{proposition}

\begin{proof}
First, we note that both $\tilde{X} \to X$ is log quasi-\'etale and 
and $X' \to X$ is  \'etale locally trivial over the smooth locus. All homotopy groups are considered relatively over $Z$. Thus, we have an exact sequence of homotopy groups 
\[
\xymatrix@C=15pt{
 \pi_2^{\rm reg}(\mathbb{T}) \ar[r] & \pi_2^{\rm reg}(X' \setminus  \Delta') \ar[r] & \pi_2^{\rm reg}(X \setminus  \Delta) \ar[r] & 
\pi_1^{\rm reg}(\mathbb{T}) \ar[r] & \pi_1^{\rm reg}(X' \setminus  \Delta') \ar[r] & \pi_1^{\rm reg}(X \setminus  \Delta) \ar[r] &
\pi_0^{\rm reg}(\mathbb{T}).
}
\]
Note that
$\pi_2^{\rm reg}(\mathbb{T})\cong \pi_0^{\rm reg}(\mathbb{T}) \cong 1$
and 
$\pi_1^{\rm reg}(\mathbb{T})\cong
\zz^{\dim(\mathbb{T})}$.
Hence, we have the following exact sequence
\[
\xymatrix@C=15pt{
 1 \ar[r] & \pi_2^{\rm reg}(X' \setminus  \Delta') \ar[r] & \pi_2^{\rm reg}(X \setminus  \Delta) \ar[r] & 
\zz^{\dim(\mathbb{T})} \ar[r] & \pi_1^{\rm reg}(X' \setminus  \Delta') \ar[r] & \pi_1^{\rm reg}(X \setminus  \Delta) \ar[r] &
1.
}
\]
Above, by abuse of notation, we denote by 
$X\setminus \Delta$ the complement of the support of $\Delta$.
Now, we recall that the orbifold fundamental group $\pi_1^{\rm \reg}(X,\Delta)$ is by definition  $\pi_1^{\rm \reg}(X \setminus \Delta)/ \langle \gamma_i^{m_i} \rangle$, where $\gamma_i$ is a small loop around a general component of $\Delta_i$ and $\frac{m_i-1}{m_i}$ is the coefficient of $\Delta_i$ in the standard approximation of $\Delta$, cf. Definition~\ref{def:standard-approx}. But since $X' \to X$ is  \'etale locally trivial over the smooth locus, a small loop around a general point of $\Delta_i$ is as well a small loop around a general point of the pullback $\Delta_i'$. Of course the coefficient $\Delta_i'$ in the standard approximation of $\Delta'$ is $\frac{m_i-1}{m_i}$ again. 

By the above exact sequence, setting $A :={\rm im}(\zz^{\dim(\mathbb{T})})$, we have 
$
\pi_1^{\rm reg}(X \setminus \Delta) = \pi_1^{\rm reg}(X',\Delta')/A
$.
So, with the above considerations, we get
\[
\pi_1^{\rm \reg}(X,\Delta):=\pi_1^{\rm reg}(X \setminus \Delta) / \langle \gamma_i^{m_i} \rangle = \left(\pi_1^{\rm reg}(X',\Delta')/A \right)/ \langle \gamma_i^{m_i} \rangle = \left(\pi_1^{\rm reg}(X',\Delta')/\langle \gamma_i^{m_i} \rangle \right)/ A  = \pi_1^{\rm \reg}(X',\Delta')/A.
\]
Since $\pi_1^{\rm \reg}(X'/Z,\Delta')$ is finite by Theorem~\ref{thm:rel-finiteness}, in fact $A$ is a finite abelian group.
Recall that arbitrary base change preserves \'etaleness, finiteness and GIT-quotients.
Hence,  the cover $\tilde{X}' \to X'$ is  indeed a log quasi-\'etale finite Galois cover with Galois group $\pi_1^{\rm reg}(X,\Delta)$. 
So $\tilde{X}' \to X'$ is the Galois cover associated to the normal subgroup $A$. In particular, there is an $A$-quasi-torsor 
\[
Y \xrightarrow{/A} \tilde{X}'.
\]
With the same arguments as above, $\tilde{X}' \to \tilde{X}$ is a $\mathbb{T}$-quasi-torsor. 

Since $\mathbb{T}$ is connected, by~\cite[Theorem 4.2.3.2]{ADHL15} we can lift the $\mathbb{T}$-action on $\tilde{X}'$ to a $\mathbb{T} \times A$ action on $Y$, such that $Y \to \tilde{X}$ becomes a $\mathbb{T} \times A$-quasi-torsor. 

Furthermore, $\Cl(\tilde{X}/Z,\tilde{\Delta})$ is torsion free. Otherwise by Lemma~\ref{le:CoxCox}, the torsion part would induce a log quasi-\'etale Galois cover of $(\tilde{X},\tilde{\Delta})$, which contradicts the log-simply-connectedness of the smooth locus. By Proposition~\ref{prop:quot-pres}, this means there is a monomorphism from $A$ to a torsion free group, which implies that $A$ is trivial.     
\end{proof}

Now, we turn to prove the existence of the scfc cover.

\begin{theorem}
\label{thm:scfc-cover}
Let $(X,\Delta)$ be a log pair.
Let 
$\phi\colon X\rightarrow Z$ be an aff-contraction 
so that $Z$ is local Henselian.
Assume that $(X,\Delta)$ is relatively Fano over $Z$.
Let $(\tilde{X},\tilde{\Delta}) \to (X,\Delta)$ be the finite log quasi-\'etale Galois cover associated to $\pi:=\pi_1^{\rm reg}(X/Z,\Delta)$. Denote by $(\widetilde{\overline{X}},\widetilde{\overline{\Delta}})$ the total coordinate space of $(\tilde{X},\tilde{\Delta})$ over $Z$. Then we have a commutative diagram
\[
\xymatrix{
(\widetilde{\overline{X}},\widetilde{\overline{\Delta}}) \ar[dr]^{/G} \ar[d]^{/\mathbb{T}}  \\
(\tilde{X},\tilde{\Delta}) \ar[r]^{/\pi} & (X,\Delta),
}
\]
where the following conditions hold:
\begin{enumerate}
    \item The characteristic quasi-torus $\mathbb{T}$ is connected, i.e., a  torus.
    \item $G$ is a reductive group acting freely in log-codimension one on $(\widetilde{\overline{X}},\widetilde{\overline{\Delta}})$ and fitting in the short exact sequence
    \[
\xymatrix{
1 \ar[r] & \mathbb{T} \ar[r] & G \ar[r] & \pi \ar[r] & 1.
}
\]
\item $(\widetilde{\overline{X}},\widetilde{\overline{\Delta}})$ is factorial over $Z$ and has canonical singularities.
\item $(\widetilde{\overline{X}},\widetilde{\overline{\Delta}})$ is log-simply connected in codimension one. 
\end{enumerate}
\end{theorem}

\begin{proof}
By Proposition~\ref{prop:fiber-prod-fincov-and-quotpres}, we know that $\Cl(\tilde{X}/Z,\tilde{\Delta})$ is torsion free. 
 Thus $\mathbb{T}$ is a torus, and  $(\widetilde{\overline{X}},\widetilde{\overline{\Delta}})$ is factorial over $Z$. In particular, it is locally factorial, and since it is klt type by Theorem~\ref{thm:hen-cox}, it has canonical singularities, yielding items (1) and (3). In fact, by Corollary~\ref{cor:CoxCoxFano}, we already know that the Cox ring of a relative Fano pair is Gorenstein and has canonical singularities, while in general, it is of course not factorial. 
 Proposition~\ref{prop:fiber-prod-fincov-and-quotpres} shows that 
 \[
 \pi_1^{\reg}(\widetilde{\overline{X}}/Z,\widetilde{\overline{\Delta}}) \cong \pi_1^{\reg}(\tilde{X}/Z,\tilde{\Delta}) \cong 1,
 \]
 yielding (4). Lastly, (2) follows from Proposition~\ref{prop:proj-lift-aut}. 
\end{proof}

\begin{remark}{\em 
The construction of the scfc cover, as usual, depends on the choice of a subgroup $N \subseteq {\rm CL}(\tilde{X}/Z,\tilde{\Delta})$ and character $\chi$.
}
\end{remark}

\subsection{Universality of the simply connected factorial canonical cover}\label{subsec:universality-scfc}

In this subsection, we prove a universality property for the scfc cover.
This means that the scfc cover dominates any sequence of finite covers and abelian quasi-torsors overf the singularity.

\begin{theorem}\label{thm-univ-scfc}
Let $(X,\Delta)$ be a log pair.
Let $\phi\colon X \rightarrow Z$ be an aff-contraction, 
where $X$ is of dimension $n$ and $Z$ is local Henselian.
Assume that $(X,\Delta)$ is of relative Fano type over $Z$. Let
\[
\xymatrix{
\cdots \ar[r] & (X_{(2)},\Delta_{(2)}) \ar[r] & (X_{(1)},\Delta_{(1)}) \ar[r] & (X_{(0)},\Delta_{(0)}) :=(X,\Delta)
}
\]
be a (possibly infinite) sequence of finite log quasi-\'etale covers and abelian quasi-torsors. Then $(X_{(j)},\Delta_{(j)})$ stabilizes after finitely many steps and the scfc covers of $(X_{(j)},\Delta_{(j)})$ coincide  for all $j \geq 0$. If in addition all abelian (finite or quasi-torsor-) covers in the sequence are given by the respective Cox rings, then there is a constant $j(n)$ only depending on $n$, such that $(X_{(j)},\Delta_{(j)})$ stabilizes for $j \geq j(n)$.
\end{theorem}

\begin{remark}
{\em 
As mentioned before, the construction of the scfc cover depends on the choice of a subgroup $N \subseteq \Cl(\tilde{X}/Z,\tilde{\Delta})$ and a character $\chi$. So the equality of the scfc covers of $(X_{(j)},\Delta_{(j)})$ means that for any $j_1 \neq j_2 \geq 0$, any scfc cover of $(X_{(j_1)},\Delta_{(j_1)})$ is also a scfc cover of  $(X_{(j_2)},\Delta_{(j_2)})$ and vice-versa.
}
\end{remark}

\begin{proof}[Proof of Theorem~\ref{thm-univ-scfc}]
We show by induction that all the scfc covers coincide. So let first $(X_{(j+1)},\Delta_{(j+1)}) \to (X_{(j)},\Delta_{(j)})$ be finite. Then the covers $(\tilde{X}_{(j+1)},\tilde{\Delta}_{(j+1)})$ and $(\tilde{X}_{(j)},\tilde{\Delta}_{(j)})$ associated to the respective regional fundamental groups obviously coincide, so the scfc covers coincide as well. 

Now let $(X_{(j+1)},\Delta_{(j+1)}) \to (X_{(j)},\Delta_{(j)})$ be a $H$-quotient presentation. Then taking the quotient by the identity component $H^0$ yields a commutative diagram 
\[
\xymatrix@C=50pt{
(X_{(j+1)},\Delta_{(j+1)}) \ar[d]^{/H^0} \ar[dr]^{/H} \\
(X_{(j+1)}',\Delta_{(j+1)}') \ar[r]^{/(H/H^0)} & (X_{(j)},\Delta_{(j)}),
}
\]
where $(X_{(j+1)}',\Delta_{(j+1)}') \to (X_{(j)},\Delta_{(j)})$ is finite log quasi-\'etale abelian. Thus the scfc covers of $(X_{(j+1)}',\Delta_{(j+1)}')$ and $(X_{(j)},\Delta_{(j)})$ coincide. By Proposition~\ref{prop:fiber-prod-fincov-and-quotpres}, we can extend the diagram in the following way:
\[
\xymatrix@C=50pt{
(\tilde{X}_{(j+1)},\tilde{\Delta}_{(j+1)}) \ar[r]^{/\pi} \ar[d]^{/H^0} & (X_{(j+1)},\Delta_{(j+1)}) \ar[d]^{/H^0} \ar[dr]^{/H} \\
(\tilde{X}_{(j)},\tilde{\Delta}_{(j)}) \ar[r]^{/\pi} & (X_{(j+1)}',\Delta_{(j+1)}') \ar[r]^{/(H/H^0)} & (X_{(j)},\Delta_{(j)}).
}
\]
Observe that $(\tilde{X}_{(j+1)},\tilde{\Delta}_{(j+1)}) \to (\tilde{X}_{(j)},\tilde{\Delta}_{(j)})$ is a quasi-torsor. We conclude that the scfc cover of $(X_{(j+1)},\Delta_{(j+1)})$ and the scfc cover of $(X_{(j)},\Delta_{(j)})$ coincide. 

In order to show that the sequence stabilizes, we show the following claim by induction.\\

\textbf{Claim:} The first $k$ covers in the sequence induce a sequence of $k$ finite covers \[
(X'_{(j+1)},\Delta'_{(j+1)})\rightarrow (X'_{(j)},\Delta'_{(j)}),
\]
where 
\[
(X_{(j)},\Delta_{(j)}) \rightarrow (X'_{(j)},\Delta'_{(j)})
\]
is a torus-quasi-torsor for every $j$.\\

\begin{proof}[Proof of the Claim]
For $k=1$, either the cover is finite or it is a $H$-quasi-torsor.
In the latter case, we take the finite cover given by the group of components $H/H^0$. Assume the claim is proven for $k$. Then, as we have seen before, $(X_{(k+1)},\Delta_{(k+1)}) \to (X_{(k)},\Delta_{(k)})$ yields a finite Galois cover $(X_{(k+1)}^*,\Delta_{(k+1)}^*) \to (X_{(k)},\Delta_{(k)})$. Due to Proposition~\ref{prop:fiber-prod-fincov-and-quotpres}, we know that the regional fundamental groups of $(X_{(k)},\Delta_{(k)})$ and the $k$-th finite cover $(X_{(k)}',\Delta_{(k)}')$ coincide. So $(X_{(k+1)}^*,\Delta_{(k+1)}^*) \to (X_{(k)},\Delta_{(k)})$ induces a finite Galois cover $(X_{(k+1)}',\Delta_{(k+1)}') \to (X_{(k)}',\Delta_{(k)}')$. Here, 
\[
(X_{(k+1)}^*,\Delta_{(k+1)}^*) \to (X_{(k+1)}'\Delta'_{(k+1)})
\]
is a $\mathbb{T}_1$-quasi-torsor. We can lift the action of $\mathbb{T}_1$ to the $\mathbb{T}_2$-quasi-torsor $(X_{(k+1)},\Delta_{(k+1)}) \to (X_{(k+1)}^*\Delta_{(k+1)}^*)$ such that 
\[
(X_{(k+1)},\Delta_{(k+1)}) \to (X_{(k+1)}'\Delta_{(k+1)}')
\]
is a $(\mathbb{T}_1 \times \mathbb{T}_2)$-quasi-torsor. Thus, we showed the existence of the sequence of $k$ induced finite covers. 
This finishes the proof of the claim.
\end{proof}

By the claim, 
we have already that the scfc  cover $(\widetilde{\overline{X}},\widetilde{\overline{\Delta}})$  of $(X,\Delta)$ dominates the original sequence. It follows that at most 
\[
\kappa:=\dim(\widetilde{\overline{X}}) - \dim(X)
\]
of the original covers can induce a trivial finite cover. Thus, by finiteness of $\pi_1^{\reg}(X/Z,\Delta)$, the original sequence stabilizes for $j$ large enough. In the case that all quasi-torsors in the original sequence are Cox covers, the bound $j(n)$ on the number of nontrivial covers follows as in the proof of Theorem~\ref{thm:bounded-iteration}. 
\end{proof}

\begin{remark}
{\em
If we do not assume the quasi-torsors to be Cox covers, there is no bound depending only on the dimension. This already happens in dimension two. We can construct sequences of arbitrary length of nontrivial abelian quasi-\'etale covers over two-dimensional $A_n$-singularities, if we do not fix $n$.
}
\end{remark}

\subsection{Upper bound for the dimension of the iteration of Cox rings}\label{subsec:upper-bound-dim}

In this subsection, we give an upper bound for the dimension of the iteration of Cox rings in terms of homotopy groups. 
For orbifolds (and more general, orbispaces), similarly to the fundamental group, one may define higher homotops groups $\pi_k$, and for orbispace fibrations, these groups satisfy the same long exact sequence as ordinary homotopy groups, cf.~\cite[Theorem 4.5]{Chen01}. The precise statement is the following.

\begin{theorem}
\label{thm:dim-bound-2-homotopy}
Let $(X,\Delta)$ be a log pair.
Let 
$\phi\colon X\rightarrow Z$ be an aff-contraction 
so that $Z$ is local Henselian.
Assume that $(X,\Delta)$ is relatively Fano over $Z$.
Let $(\widetilde{\overline{X}},\widetilde{\overline{\Delta}})$ be the scfc cover of  $(X,\Delta)$. Then   
\[
\dim(\widetilde{\overline{X}},\widetilde{\overline{\Delta}}) \leq \dim (X) +  {\rm rk}( \pi_2^{\reg}(X/Z, \Delta)\otimes \mathbb{Q}).
\]
\end{theorem}

\begin{proof}
As in the proof of Proposition~\ref{prop:fiber-prod-fincov-and-quotpres}, we use the fact that $(\widetilde{\overline{X}},\widetilde{\overline{\Delta}}) \xrightarrow{/G} (X,\Delta)$ is an \'etale locally trivial $G$-bundle over $(X_{\reg}, \Delta_{\reg})$.
Thus, we have an exact sequence of orbifold homotopy groups 
\[
\xymatrix@C=15pt{
 \pi_2^{\rm reg}(G) \ar[r] & \pi_2^{\rm reg}(\widetilde{\overline{X}}/Z, \widetilde{\overline{\Delta}}) \ar[r] & \pi_2^{\rm reg}(X/Z,  \Delta) \ar[r] & 
\pi_1^{\rm reg}(G) \ar[r] & \pi_1^{\rm reg}(\widetilde{\overline{X}}/Z,  \widetilde{\overline{\Delta}}) \ar[r] & \pi_1^{\rm reg}(X/Z,  \Delta) \ar[r] &
\pi_0^{\rm reg}(G).
}
\]
Note that the orbifold structure on the fiber $G$ is trivial and, moreover, 
$\pi_2^{\rm reg}(G)\cong 1$, $\pi_1^{\rm reg}(G) \cong
\zz^{\dim(\widetilde{\overline{X}}) - \dim(X)}$, and $\pi_0^{\rm reg}(G) \cong \pi_1^{\rm reg}(X, \Delta)$.
Since $\qq$ is a flat $\zz$-module, tensoring the above exact sequence with $\qq$ yields an exact sequence of $\qq$-vector spaces
\[
\xymatrix@C=15pt{
 \pi_2^{\rm reg}(X,  \Delta) \otimes \qq  \ar[r] & 
\qq^{\dim(\widetilde{\overline{X}}) - \dim(X) } \ar[r] &  0
}
\]
which finishes the proof.
\end{proof}

\begin{remark}{\em 
If we consider the scfc cover with respect only to $X$, without the orbifold structure, then the second homotopy group in Theorem~\ref{thm:dim-bound-2-homotopy} is the ordinary second homotopy group $\pi_2^{\rm reg}(X/Z)$ of the smooth locus.
}
\end{remark}

\section{Fano type varieties with smooth iteration of Cox rings}\label{sec:smoothit}

Throughout this paper, 
we introduced some special covers of Fano type varieties and klt singularities.
The aim of this section 
is to explain when such coverings are smooth.
In Subsection~\ref{subsec:smooth-iteration}, we give a characterization of Fano type varieties
with smooth iteration of Cox rings.
In Subsection~\ref{subsec:smoothness-scfc}, we give a characterization of Fano type varieties
with smooth simply connected factorial canonical cover.
Analogous theorems hold for klt singularities.
In such case, we will just enunciate the results without a proof since these are verbatim from the projective case.
Finally, in Subsection~\ref{subsec:iteration-vs-scfc}, we will give a characterization of Fano type varieties for which the spectrum of the iteration coincides with the scfc cover.

\subsection{Smoothness of the iteration of Cox rings}\label{subsec:smooth-iteration}

In this subsection, we give a characterization of the smoothness of the iteration of Cox rings.

\begin{theorem}\label{thm:smooth-it-proj} 
Let $(X,\Delta)$ be a Fano type pair.
Then, the following statements are equivalent:
\begin{enumerate}
\item The spectrum of the iteration of Cox rings ${\rm Cox}^{\rm it}(X,\Delta)$ is smooth, and
\item $(X,\Delta)$ is a finite quasi-\'etale solvable quotient of a projective toric pair. 
\end{enumerate}
Furthermore, if any of the above conditions holds, then we have that
\begin{enumerate}
    \item[(1')] The simply connected factorial canonical cover coincides with the spectrum of the iteration of Cox rings, and
    \item[(2')] $(X,\Delta)$ is a finite quotient of a projective toric pair with torsion-free class group.
\end{enumerate}
\end{theorem}

\begin{proof}
Assume that $(X,\Delta)$ is a finite quasi-\'etale solvable quotient of a projective toric pair $(T,\Delta_T)$.
Then, we can find:
\begin{itemize}
    \item a sequence of finite abelian groups $A_1,\dots,A_n$, 
    \item projective Fano type pairs $(X_i,\Delta_i)$ with $(X_n,\Delta_n)=(T,\Delta_T)$, and 
    \item $A_i$ acts on $(X_i, \Delta_i)$ so that
    $(X_{i-1},\Delta_{i-1})$ is the quotient by this action.
\end{itemize}
Moreover, we may assume $(X_0,\Delta_0)=(X,\Delta)$.
Let ${\rm Cox}^{(k)}(X,\Delta)$ be the $k$-th iteration of the Cox ring of $(X,\Delta)$.
We denote by $\mathbb{T}_k$ the connected component of the reductive solvable group acting
on the $k$-th iteration ${\rm Cox}^{(k)}(X,\Delta)$.
Since the Cox ring dominates all quasi-\'etale finite abelian covers, we have a diagram as follows:
\[
 \xymatrix{
 (T,\Delta_T)\ar[d]^-{/A_n} & (Y_n,B_n)\ar[d]\ar[l] & {\rm Cox}^{(n)}(X,\Delta)\ar[d]\ar[l]_-{/ \mathbb{T}_n} \\ 
 (X_{n-1},\Delta_{n-1})\ar[d]^-{/A_{n-1}} & (Y_{n-1},B_{n-1})\ar[d]\ar[l] & {\rm Cox}^{(n-1)}(X,\Delta)\ar[d]\ar[l]_-{/ \mathbb{T}_{n-1}} \\ 
 \vdots \ar[d]^-{/A_2} & \vdots\ar[l]\ar[d] & \vdots\ar[d]\ar[l] \\ 
(X_{1},\Delta_{1})\ar[d]^-{/A_1} & (Y_{1}, B_{1})\ar[l] & {\rm Cox}^{(1)}(X,\Delta)\ar[l]_-{/ \mathbb{T}_{1}} \\ 
(X,\Delta). & & 
 }
\]
In the above diagram, each map 
$(Y_i,B_i)\rightarrow (X_i,\Delta_i)$
is a finite quotient.
Since $T$ is toric
and $Y_n\rightarrow T$ only branches along the support of $\Delta_T$,
we conclude that $Y_n\rightarrow T$ does not branch along the torus.
Hence, $(Y_n,B_n)$ is a projective toric pair.
Thus, we have that 
${\rm Cox}(Y_n,B_n)\cong \mathbb{A}^{\rho(Y_n)+\dim(Y_n)}$.
Since the morphism
${\rm Cox}^{(n)}(X,\Delta)\rightarrow Y_n$ is a quasi-\'etale abelian cover, we conclude that there is a commutative diagram:
\[
\xymatrix{
\mathbb{A}^{\rho(Y_n)+\dim(Y_n)} \ar[d]^-{\slash Q} \ar[rd] &  \\
{\rm Cox}^{(n)}(X,\Delta)\ar[r]^-{\slash \mathbb{T}_n} & (Y_n,\Delta_n). 
}
\]
Here, $Q$ is a quasi-torus acting on the affine space. 
Thus $\mathbb{A}^{\rho(Y_n)+\dim(Y_n)}\rightarrow {\rm Cox}^{(n)}(X,\Delta)$ is a quasi-\'etale abelian cover.
We conclude that
\[
{\rm Cox}^{(n+1)}(X,\Delta) \cong
\mathbb{A}^{\rho(Y_n)+\dim(Y_n)}.
\]
This concludes that
$(2)$ implies $(1)$.

Now, we prove that $(1)$ implies $(2)$
and we argue that $(1')$ and $(2')$ holds in this case.
We assume that the iteration of Cox rings is smooth.
Recall from Section~\ref{sec:scfc} that the simply connected factorial canonical cover is the Cox ring of the universal cover of the spectrum of the iteration of Cox rings.
This implies that $(1')$ holds in this case.
Without loss of generality, 
we may assume that the iteration of Cox rings is an affine space $\mathbb{A}^n$
and the reductive solvable group $G$ acting on it satisfies $G\leqslant {\rm GL}_n(\cc)$.
Indeed, this can be achieved by applying Luna \'etale slice theorem to the smooth $G$-fixed point on the iteration of Cox rings.
The connected component at the identity of the
reductive solvable group is a torus, we call it $\mathbb{T}$.
Hence, we have that $T:=\mathbb{A}^n/ \mathbb{T}$ is a projective toric variety with 
torsion-free class group.
This shows that $(2')$ holds in this case.
Finally, $X$ is the quotient of $Y$ by the component group of $G$ which is a finite solvable group. 
This finishes the second implication.
\end{proof}

We have the following corresponding statement for klt singularities.

\begin{theorem}\label{thm:smooth-it-local} 
Let $(X,\Delta;x)$ be a klt singularity.
Then, the following statements are equivalent:
\begin{enumerate}
\item The spectrum of the iteration of Cox rings ${\rm Cox}^{\rm it}(X,\Delta;x)$ is smooth, and
\item $(X,\Delta;x)$ is a finite quasi-\'etale solvable quotient of a toric singularity. 
\end{enumerate}
Furthermore, if any of the above conditions hold, then we have that
\begin{enumerate}
    \item[(1')] The simply connected factorial canonical cover coincides with the spectrum of the iteration of Cox rings, and
    \item[(2')] $(X,\Delta)$ is a finite quotient of a toric singularity with torsion-free class group.
\end{enumerate}
\end{theorem}

\subsection{Smoothness of the scfc cover}\label{subsec:smoothness-scfc}

In this subsection, we give a characterization of Fano type varieties with smooth scfc cover.

\begin{theorem}\label{thm:smooth-scfc-proj} 
Let $(X,\Delta)$ be a Fano type pair.
Then, the following statements are equivalent:
\begin{enumerate}
\item The simply connected factorial canonical cover of $(X,\Delta)$ is smooth, and 
\item $(X,\Delta)$ is a finite quasi-\'etale quotient of a projective toric pair. 
\end{enumerate}
\end{theorem}

\begin{proof}
Assume that $(X,\Delta)$ is a finite quasi-\'etale quotient of a projective toric pair $(T,\Delta_T)$.
Let ${\rm Cox}^{\rm it}(X,\Delta)$ be the iteration of Cox rings of $(X,\Delta)$.
Let $\mathbb{T}$ be the connected component of the reductive solvable group acting on ${\rm Cox}^{\rm it}(X,\Delta)$. 
We denote the quotient by 
$(Y,B):={\rm Cox}^{\rm it}(X,\Delta)\slash \mathbb{T}$.
Hence, as in the proof of Theorem~\ref{thm:smooth-it-proj}, we have a commutative diagram as follows
\[
\xymatrix{
(T,\Delta_T)\ar[d]^-{/S} & & \\
(T',\Delta_{T'}) & (Y,B)\ar[l]^-{/H} &
{\rm Cox}^{\rm it}(X,\Delta)\ar[l]^-{/\mathbb{T}}.
}
\]
Here, $(T',\Delta_{T'})$ is a quotient of the projective toric pair $(T,\Delta_T)$ by a finite perfect group $S$.
Furthermore, $H$ is a finite solvable group.
Hence, we obtain natural inclusions
$S\leqslant {\pi_1}^{\rm reg}(Y,B)$
and 
$S\leqslant {\pi_1}^{\rm reg}({\rm Cox}^{\rm it}(X,\Delta))$.
Thus, we conclude that the universal cover
of ${\rm Cox}^{\rm it}(X,\Delta)$
admits a quasi-\'etale abelian quotient to
$(T,\Delta_T)$.
Therefore, 
the universal cover of the iteration of Cox rings 
is a torus quotient of the Cox ring of $(T,\Delta_T)$ 
which is an affine space $\mathbb{A}^{\rho(T)+\dim(T)}$.
Hence, the simply connected factorial canonical cover of $(X,\Delta)$ is the affine space
$\mathbb{A}^{\rho(T)+\dim(T)}$ as claimed.
We conclude that $(2)$ implies $(1)$.

Now, we prove that $(1)$ implies $(2)$
We assume that the simply connected factorial canonical cover of $(X,\Delta)$ is smooth.
Without loss of generality, 
we may assume that the simply connected factorial canonical cover is the affine space $\mathbb{A}^n$
and the reductive solvable group $G$ acting on it satisfies $G\leqslant {\rm GL}_n(\cc)$.
Indeed, this can be achieved by applying Luna \'etale slice theorem to the smooth $G$-fixed point on the iteration of Cox rings.
The connected component at the identity of the
reductive solvable group is a torus, we call it $\mathbb{T}$.
Hence, we have that $T:=\mathbb{A}^n/ \mathbb{T}$ is a projective toric variety with 
torsion-free class group.
Finally, $X$ is the quotient of $Y$ by the component group of $G$ which is a finite group. 
This finishes the second implication.

\end{proof}

We have the following corresponding statement for klt singularities.

\begin{theorem}\label{thm:smooth-scfc-local} 
Let $(X,\Delta;x)$ be a klt singularity.
Then, the following statements are equivalent:
\begin{enumerate}
\item The simply connected factorial canonical cover of $(X,\Delta)$ is smooth, and 
\item $(X,\Delta)$ is a finite quasi-\'etale quotient of a toric singularity. 
\end{enumerate}
\end{theorem}

\begin{remark}
{\em
The singularities that appear in  Theorem~\ref{thm:smooth-scfc-local} are considered by the second author in~\cite{Mor20a,Mor20b}, where they are called toric quotient singularities. 
In~\cite{Mor20a,Mor20b}, it is shown that toric quotient singularities are the prototypes of klt singularities with large fundamental group.
Moreover, the
minimal log discrepancies of these singularities are described in~\cite{Mor20c}.
}
\end{remark}

\subsection{Iteration of Cox rings and scfc cover}
\label{subsec:iteration-vs-scfc}

In this subsection, we characterize when the iteration of Cox rings is isomorphic to the scfc cover.

\begin{theorem}
Let $(X,\Delta)$ be a Fano type variety.
Then, the following are equivalent:
\begin{enumerate}
    \item The spectrum of the iteration of Cox rings has trivial regional fundamental group,
    \item the spectrum of the iteration of Cox rings coincides with the simply connected factorial canonical cover, and
    \item the fundamental group $\pi_1^{\rm reg}(X,\Delta)$ is solvable.
\end{enumerate}
\end{theorem}

\begin{proof}
If the spectrum of the iteration of Cox rings has trivial regional fundamental group,
then it is factorial and simply connected.
Thus, we have that $(1)$ implies $(2)$.
The condition $(2)$ trivially implies $(1)$.

Assume that the spectrum of the iteration of Cox rings has trivial regional fundamental group.
Let $G$ be the solvable reductive group acting on the iteration of Cox rings ${\rm Cox}^{\rm it}(X,\Delta)$.
We know that $X$ is the quotient of ${\rm Cox}^{\rm it}(X,\Delta)$ by $G$.
Let $G^0$ be the connected compontent at the identity of $G$.
The quotient $X':={\rm Cox}^{\rm it}(X,\Delta)/G^0$ is a finite solvable cover of $(X,\Delta)$. 
Furthermore, the pull-back of $K_X+\Delta$ to $X'$ equals $K_{X'}+\Delta'$ where $\Delta'$ is an effective divisor.
Assume that $\pi_1^{\rm reg}(X,\Delta)$ is not solvable.
Then, $X'\rightarrow X$ is not the regional universal cover of $(X,\Delta)$.
Thus, we can take a non-trivial finite log quasi-\'etale Galois cover of $(X',\Delta')$.
We call this finite cover $Y\rightarrow X$.
By Proposition~\ref{prop:fiber-prod-fincov-and-quotpres}, this cover induces a non-trivial
finite log quasi-\'etale Galois cover
of the spectrum of the iteration of Cox rings.
This contradicts the fact that the spectrum of
the iteration of Cox rings has trivial
regional fundamental group.
We conclude that $(X',\Delta')$ is the universal cover of $(X,\Delta)$.
Thus, $\pi_1^{\rm reg}(X,\Delta)$ is a solvable group.
We conclude that $(2)$ implies $(3)$.

Now, assume that the regional fundamental group
$\pi_1^{\rm reg}(X,\Delta)$ is solvable.
We proceed as in the previous paragraph.
Let $G$ be the solvable reductive group acting on the iteration of Cox rings ${\rm Cox}^{\rm it}(X,\Delta)$.
Let $G^0$ be the connected component at the identity of $G$.
The quotient $X':={\rm Cox}^{\rm it}(X,\Delta)/G^0$ is a finite solvable cover of $(X,\Delta)$.
If the regional fundamental group of $(X,\Delta)$ is solvable,
then the regional fundamental group of $(X',\Delta')$ is solvable.
If $(X',\Delta')$ has simply connected log smooth locus, then
by the proof of Theorem~\ref{thm:scfc-cover},
we conclude that ${\rm Cox}^{\rm it}(X,\Delta)$ equals the scfc cover.
Indeed, in this case the Cox ring of $(X',\Delta')$ equals
${\rm Cox}^{\rm it}(X,\Delta)$.
On the other hand, assume that
$(X',\Delta')$ has non-trivial regional fundamental group $S$.
By assumption, $S$ is solvable.
In particular, its commutator is a proper subgroup.
By Proposition~\ref{prop:fiber-prod-fincov-and-quotpres}, there exists a finite log quasi-\'etale Galois cover of ${\rm Cox}^{\rm it}(X,\Delta)$ with acting group isomorphic to $S$.
Since $[S,S]\leqslant S$ is proper, by Corollary~\ref{cor:abelianization},
we conclude that ${\rm Cox}^{\rm it}(X,\Delta)$ is not factorial.
This leads to a contradiction.
We conclude that $(X',\Delta')$ is simply connected and its Cox ring is isomorphic to
${\rm Cox}^{\rm it}(X,\Delta)$.
Thus, the spectrum of the iteration equals the scfc of $(X,\Delta)$.
We have that $(3)$ implies $(2)$.
Hence, all the statements are equivalent.
\end{proof}

The following local version of the above theorem is proved analogously.

\begin{theorem}\label{thm:scfc=it-local} 
Let $(X,\Delta;x)$ be a klt singularity.
Then, the following are equivalent:
\begin{enumerate}
    \item The spectrum of the iteration of Cox rings has trivial regional fundamental group,
    \item the spectrum of the iteration of Cox rings coincides with the simply connected factorial canonical cover, and
    \item the regional fundamental group $\pi_1^{\rm reg}(X,\Delta;x)$ is solvable.
\end{enumerate}
\end{theorem}

\section{Examples and proofs of the theorems}\label{sec:ex}

In this section, we collect some examples that 
enlighten the techniques of the paper.
Then, we explain how the theorems of the introduction are implied by the theorems proved throughout the manuscript.
In subsection~\ref{subsec:p1-non-solvable}, we will give an example of a Fano type variety with non-solvable regional fundamental group.
We describe its scfc and iteration of Cox rings explicitly.
In Subsection~\ref{subsec:jordan-t-varieties}, we prove a special Jordan property for singularities with torus action. Finally, in Subsection~\ref{subsec:compl-one}, we give a detailed study of the iteration of Cox rings, regional fundamental groups, and scfc covers of klt singularities of complexity one.

\subsection{Examples with $\pi_1^{\rm reg}(X^{\rm reg})$ non-solvable}
\label{subsec:p1-non-solvable}

In this subsection, we give an example of a Fano type variety $X$ so that its regional fundamental group $\pi_1^{\rm reg}(X)$ is non-solvable.
We also explain how to obtain the simply connected factorial canonical cover of $X$.

\begin{example}
{\em
Let $X$ be the variety obtained from 
$(\pp^1)^n$ quotient by the action 
of $S_n$ permuting the coordinates.
We denote the quotient by
$\rho\colon (\pp^1)^n\rightarrow X$.
Then, $X$ is a projective variety of Fano type, which is not toric.
Furthermore, we have that $X$ satisfies that
$\pi_1^{\rm reg}(X)\cong S_n$.
Indeed, we have a natural \'etale Galois morphism 
$\rho\colon \rho^{-1}(X^{\rm reg})\rightarrow X^{\rm reg}$
with Galois group $S_n$.
Furthermore,
$\rho^{-1}(X^{\rm reg})$ has codimension at least two in $(\pp^1)^n$
from which we conclude that 
$\pi_1(\rho^{-1}(X^{\rm reg}))=0$.
This implies the claim.

We proceed to compute the iteration of Cox rings of $X$.
Note that $\rho(X)=1$.
Hence, its first Cox ring is just the ring over the Weil $\qq$-Cartier divisor $-K_X$.
This gives us a klt cone singularity
\[
({\rm Cox}(X), x_0) 
\]
where $x_0$ is the vertex for the action.
The singularity ${\rm Cox}(X)$ is factorial at $x_0$.
Hence, the Iteration of Cox rings coincides with the first Cox ring.
Furthermore, the regional fundamental group
of ${\rm Cox}(X)$ at $x_0$ is isomorphic to $S_n$
and its universal cover
is isomorphic to the
cone over $(\pp^1)^n$ with respect to $-K_{(\pp^1)^n}$.
We denote this variety 
by ${\rm Cone}((\pp^1)^n)$.
Note that the regional fundamental group 
of ${\rm Cone}((\pp^1)^n)$ is trivial
and its Cox ring is the affine space $\mathbb{A}^{2n}$.
Hence, the 
simply connected factorial canonical cover of $X$ is
the $2n$-dimensional affine space.
We obtain the following commutative diagram
\[
\xymatrix{
\mathbb{A}^{2n}\ar[r]^-{/\mathbb{T}_0} \ar[rd]^-{/\mathbb{T}_1} & {\rm Cone}((\pp^1)^n)\ar[d]^-{/\mathbb{G}_m}\ar[r]^-{/S_n} & {\rm Cox}(X) \ar[d]^-{/\mathbb{G}_m} \\
& (\pp^1)^n\ar[r]^-{/S_n} & X
}
\]
Here, the torus 
$\mathbb{T}_0$ acts on the affine space $\mathbb{A}^{2n}$ by
\[
(t_1,\dots,t_n)\cdot
(x_1,\dots,x_{2n})=
(t_1x_1,t_1x_2,
t_2x_3,t_2x_4,\dots,
t_nx_{2n-1},t_nx_{2n}).
\]
On the other hand, 
the torus $\mathbb{T}_1$ acts on the affine space $\mathbb{A}^{2n}$ by
\[
(t_1,\dots,t_{n-1})\cdot
(x_1,\dots,x_{2n})= 
(tx_1,tx_2,t_1x_3,t_1x_4,t_2x_5,t_2x_6,\dots,t_{n-1}x_{2n-1},t_{n-1}x_{2n}),
\]
where $t=(t_1\dots t_{n-1})^{-1}$.
Thus, we obtain a representation 
$X\cong \mathbb{A}^{2n}/(S_n\rtimes \mathbb{T}_0)$,
where $S_n$ is acting as permutation
on the components of $\mathbb{A}^{2n} \cong (\mathbb{A}^2)^n$
and $\mathbb{T}_0$ is acting as above.
The above example reflects two different ways in
which the simply connected factorial canonical ring can be obtained:
As the iteration of Cox rings of the universal cover, 
or 
as the Cox ring of the universal cover of the iteration of Cox ring.
This example is a particular case of Theorem~\ref{thm:smooth-scfc-proj}.
}
\end{example}

\subsection{Jordan property for $\mathbb{T}$-varieties}
\label{subsec:jordan-t-varieties}

In this subsection, we prove a strengthened version
of the Jordan property for the regional fundamental group
of affine klt $\mathbb{T}$-varieties of complexity $k$.
Then, we specialize this statement for $\mathbb{T}$-varieties
of complexity one.

\begin{theorem}\label{thm:jordan-comp-k}
Let $k$ be a positive integer.
There exists a constant $c(k)$, only depending on $k$,
satisfying the following.
Let $(X,\Delta;x)$ be a $n$-dimensional klt $\mathbb{T}$-singularity
of complexity $k$.
Then, there exists an exact sequence
\[
1\rightarrow A\rightarrow  
\pi_1^{\rm reg}(X,\Delta;x)
\rightarrow N\rightarrow 1,
\]
where $A$ is an abelian group of rank at most $n$
and index at most $c(k)$.
\end{theorem}

\begin{proof}
Let $(Y,B_Y)$ be the normalized Chow quotient of $(X,\Delta)$.
Then $(X,\Delta)$ is defined by a polyhedral divisor
$\mathcal{D}$ on $Y$ with tailcone $\sigma^\vee \in M_\qq$
(see, e.g.,~\cite[Theorem]{AH06}).
By~\cite[Theorem 4.9]{LS13}, we know that $(Y,B_Y)$ is a Fano type pair.
We may replace $Y$ with a small $\qq$-factorialization
to assume that $Y$ is $\qq$-factorial.
Let $\tilde{X}$ be relative spectrum over $Y$
of $\bigoplus_{u\in \sigma^\vee \cap M} \mathcal{O}_Y(\mathcal{D}(u))$.
Then, $\tilde{X}\rightarrow Y$ is an orbifold toric bundle.
Over the log smooth locus of $(Y,B_Y)$ the toric bundle is trivial.
Let $(\tilde{X},\Delta_{\tilde{X}})$ be the log pull-back of $(X,\Delta)$ to
$\tilde{X}$. Let $\Gamma$ be the boundary obtained from
$\Delta_{\tilde{X}}$ by increasing to one the coefficients of all
the torus invariant divisors which are horizontal over $Y$.
Hence, we conclude that there is an exact sequence
\[
1\rightarrow 
\zz^{n-k} \rightarrow
\pi_1^{\rm reg}(\tilde{X},\Gamma) 
\rightarrow
\pi_1^{\rm reg}(Y,B_Y)
\rightarrow 1.
\]
By construction, $\zz^{n-k}$ lies in
the center of $\pi_1^{\rm reg}(\tilde{X},\Gamma)$.
Furthermore, we have a surjection
$\pi_1^{\rm reg}(\tilde{X},\Gamma)
\rightarrow
\pi_1^{\rm reg}(X,\Delta;x)$.
Observe that $Y$ has dimension $k$.
By~\cite[Theorem 3]{BFMS20}, we can find an abelian normal subgroup
$A_Y \leqslant \pi_1^{\rm reg}(Y,B_Y)$ of rank at most $k$
and index at most $c(k)$, 
where $c(k)$ is a constant which only depends on $k$.
Hence, the pre-image $A_{\tilde{X}}$ of $A_Y$ in $\pi_1^{\rm reg}(\tilde{X},\Gamma) $ is a free finitely generated abelian group
of rank at most $n$ and index at most $c(k)$.
Let $A$ be the image of $A_{\tilde{X}}$ in $\pi_1^{\rm reg}(X,\Delta;x)$.
Since $\pi_1^{\rm reg}(X,\Delta;x)$ is finite, 
we conclude that $A$ is a finite abelian group 
of rank at most $n$ and index at most $c(k)$.
\end{proof}

\begin{remark}{\em 
Note that the size of the non-abelian part
$N$ of the regional fundamental group
only depends on the complexity
and not on the dimension of the germ
as in~\cite{BFMS20}.
This, of course, happens because the $(n-k)$-dimensional
torus action can not contribute to the non-abelian part
of the regional fundamental group.
If the complexity is zero, 
we can simply take $c(0)=0$ 
since the regional fundamental group
of a toric pair is always abelian.
The following corollary gives an explicit bound for $c(1)$.
}
\end{remark}

\begin{corollary}\label{cor:jordan-comp-1}
Let $(X,x)$ be a $n$-dimensional klt $\mathbb{T}$-singularity
of complexity one.
Then, there exists an exact sequence
\[
1\rightarrow A \rightarrow \pi_1^{\rm reg}(X,x)\rightarrow N \rightarrow 1,
\]
where $A$ is an abelian group of rank at most $n$ and index at most 60.
\end{corollary}

\begin{proof}
This follows from Theorem~\ref{thm:jordan-comp-k}
and the classification
of the regional fundamental groups
of log pair structures on $\pp^1$ with standard coefficients
(see, e.g.,~\cite[Example 5.1]{LLM19}). Indeed, the bound $60$ is obtained by the binary icosahedral group of order $120$ with the center $\mathbb{Z}_2$ being the only normal abelian subgroup.
\end{proof}

\subsection{Complexity one klt singularities}
\label{subsec:compl-one}

In the case of klt type singularities with a torus action of complexity one, we are able to explicitly determine all invariants defined so far: Cox rings, iterated Cox rings, regional fundamental groups,  associated universal covers, and the simply connected factorial canonical covers.
We start by recalling the construction of affine rational complexity one $\mathbb{T}$-varieties.

\begin{definition}
{\em 
Let $N$ be a free finitely generated abelian group of rank $r$.
Let $M$ be the dual of $N$.
We denote by $N_\qq$ and $M_\qq$ the corresponding $\qq$-vector spaces.
Given a polyhedron $\Delta\subset N_\qq$, we denote its {\em recession cone} to be the set of $v\in N_\qq$ so that
$v+\Delta \subset \Delta$.
The recession cone of a polyhedron is a strongly convex polyhedral cone.
It is denoted by ${\rm rec}(\Delta)$.
Let $\sigma$ be a strongly convex polyhedral cone in $N_\qq$.
We denote by ${\rm Pol}_\qq(N,\sigma)$
the semigroup of polyhedra $\Delta$ of $N_\qq$ for which 
${\rm rec}(\Delta)=\sigma$.
The additive structure of this semigroup is the Minkowski sum.
The elements of this group are called $\sigma$-polyhedra.

We denote by ${\rm CaDiv}_{\geq 0}(\pp^1)$ the semigroup of effective Cartier divisors on $\pp^1$.
A {\em polyhedral divisor} on $(\pp^1,N)$ with recession cone $\sigma$ is an element of 
\[
{\rm Pol}_{\qq}(N,\sigma)
\otimes_{\zz_{\geq 0}}
{\rm CaDiv}_{\geq 0}(\pp^1).
\]
Note that a polyhedral cone can be written as a formal finite sum
\[
\mathcal{D}=\sum_{i=1}^s \Delta_i \otimes \{p_i\},
\]
for a finite set of points $p_1,\dots,p_s$ in $\pp^1$ and $\sigma$-polyhedra $\Delta_1,\dots,\Delta_s$.
If we don't fix $N$ or the recession cone, then we just say that $\mathcal{D}$ is a {\em polyhedral divisor} on $\pp^1$.
}
\end{definition}

Let $\mathcal{D}$ be a polyhedral divisor on $\pp^1$.
We have a homomorphism of semigroups, called the evaluation homomorphism, defined as follows
\[
\mathcal{D}\colon \sigma^\vee \rightarrow {\rm CaDiv}_\qq(\pp^1) 
\]
\[
\mathcal{D}(u)=\sum_{i=1}^s \min \langle \Delta_i, u \rangle p_i.
\]
By abuse of notation, we are denoting the polyhedral divisor and the evaluation homomorphism by $\mathcal{D}$.

\begin{definition}
{\em 
A polyhedral divisor in $\pp^1$ is said to be a {\em proper polyhedral divisor} if 
$\mathcal{D}(u)$ is semiample for $u\in \sigma^\vee$ and
$\mathcal{D}(u)$ is big for 
$u\in {\rm relint}(\sigma^\vee)$.
For a proper polyhedral divisor
$\mathcal{D}$ on $\pp^1$,
we can define its {\em degree polyhedron} to be
\[
{\rm deg}(\mathcal{D}) = \sum_{i=1}^s \subset \sigma.
\]
}
\end{definition}

Given a proper polyedral divisor $\mathcal{D}$, we can associate to it a normal rational affine variety of dimension $r+1$ with an effective action of a $r$-dimensional torus.
We have a sheaf of $\mathcal{O}_{\mathbb{P}^1}$-algebras 
\[
\mathcal{A}(\mathcal{D}) =
\bigoplus_{u\in \sigma^\vee \cap M} \mathcal{O}_{\mathbb{P}^1}(\mathcal{D}(u)) \chi^u.
\]
We denote by $\widetilde{X}(\mathcal{D})$ the relative spectrum of $\mathcal{A}(\mathcal{D})$ over $\pp^1$.
We denote by $X(\mathcal{D})$ the ring of sections of $\mathcal{A}(\mathcal{D})$.
The variety $X(\mathcal{D})$ is a normal rational affine variety of dimension $r+1$ with an effective action of a $r$-dimensional torus.
Indeed, it admits an effective action of $\mathbb{T}:={\rm Spec}(\cc[M])$.
This means that $X(\mathcal{D})$ is a rational $\mathbb{T}$-variety of complexity one.
It is known that every rational $\mathbb{T}$-variety of complexity one is isomorphic to $X(\mathcal{D})$ for some polyhedral divisor $\mathcal{D}$ on $\pp^1$ (see, e.g.,~\cite[Theorem on p. 559]{AH06}).

\begin{notation}
{\em 
Let $\mathcal{D}$ be a proper polyhedral divisor on $\mathbb{P}^1$,
with recess cone $\sigma$,
and let $p\in \pp^1$.
We denote $\Delta_p =\Delta_i$ if $p=p_i$ or $\Delta_p=\sigma$ otherwise.
For every vertex $v\in \Delta_p$, we denote by $\mu(v)$ the smallest positive integer so that $\mu(v)v\in N$.
For every $p\in \pp^1$, we define
\[
\mu_p := {\rm max}\{ 
\mu(v) \mid v\in \Delta_p \}. 
\]
For every $p\in \pp^1$, we define $b_p:=(1-\mu_p^{-1})p$.
We define $B(\mathcal{D}):=\sum_{p\in \pp^1} b_p$.
Note that $B(\mathcal{D})$ is a divisor on $\pp^1$ with standard coefficients, i.e.,
$(\pp^1, B(\mathcal{D}))$
is a log pair with standard coefficients.
}
\end{notation}

From now on, we focus on complexity one $\mathbb{T}$-singularities.
This means, affine $\mathbb{T}$-varieties of complexity one $X$ with a distinguished point $x\in X$ which is a klt singularity.
We have the following theorem which characterizes the klt-ness of the complexity one $\mathbb{T}$-singularity

\begin{theorem}[Cf.~\cite{LS13}]
Let $\mathcal{D}$ be a polyhedral divisor on $\pp^1$.
Then, $(X(\mathcal{D}),x)$ is klt if and only if
$(\pp^1,B(\mathcal{D}))$ is a log Fano pair.
\end{theorem}

Note that this happens if and only if $\mu_p$ is non-trivial for at most three points in $\pp^1$, and, in addition, for these three points, the corresponding $\mu_p$ must form a platonic triple in the sense of~\cite[Example 4.1]{LLM19}.

From now on, we turn to describe the regional fundamental group of $X(\mathcal{D})$ at $x$.
To do so, first we need to understand the
$\mathbb{T}$-equivariant birational contraction $r\colon \widetilde{X}(\mathcal{D})\rightarrow X(\mathcal{D})$.
We proceed to explain which divisors are contracted by this birational contraction.
There are two types of $\mathbb{T}$-invariant divisors in $\widetilde{X}(\mathcal{D})$;
the divisors which are mapped to points in $\pp^1$ via the projection $\widetilde{X}(\mathcal{D})\rightarrow \pp^1$, which are called vertical invariant divisors.
Vertical invariant divisors are in bijection with pairs $(p,v)$ where $p\in \pp^1$ and $v$ is a vertex of the polyhedron $\Delta_p$.
Hence, we will denote the corresponding vertical divisor by
$\Delta_{(p,v)}$.
The invariant divisors which dominate $\pp^1$ are called horizontal divisors.
Horizontal divisors are in bijection with rays of the recession cone $\sigma$.
The contraction $r$ contracts exactly those horizontal divisors corresponding to rays of $\sigma$ which intersect ${\rm deg}(\mathcal{D})$ non-trivially 
(see, e.g.,~\cite[\S 10]{AH06}).

\begin{notation}
{\em
Let $\mathcal{D}$ be a proper polyhedral divisor on $\pp^1$ with recession cone $\sigma$.
Let $N_\mathcal{D} \subset N$ be the sub-lattice generated by elements of $N$ which belong to a regular sub-cone of $\sigma$ which does not intersect $\deg(\mathcal{D})$.
We introduce variables $t_1,\dots, t_r$, corresponding to a basis of $N$. For every $n\in N_{\mathcal{D}}$, 
we let $t^n:=t_1^{n_1}\dots t_r^{n_r}$.

Let $p\in \mathbb{P}^1$ and $\Delta_p$ the corresponding $\sigma$-polyhedra.
Consider the cone $\sigma(\mathcal{D},p)$ in $N_\qq \times \qq$. 
Let $N_{\sigma(\mathcal{D},p)} \subset N$ be the sub-lattice generated by elements of $N\times \zz$ which belong to a regular sub-cone of $\sigma(\mathcal{D},p)$ which does not intersect $\deg(\mathcal{D})$.
We denote by $\mathcal{B}(\mathcal{D},p)$ a basis
of $N_{\sigma(\mathcal{D},p)}$.
For every $v\in \mathcal{B}(\mathcal{D},p)$, 
we denote by $\pi_1(v)$
the projection in $N$
and by
$\pi_2(v)$ the projection in $\zz_{\geq 1}$.
}
\end{notation}

\begin{theorem}\label{thm:reg-compl-1}
Let $\mathcal{D}$ be a polyhedral divisor on $(\pp^1,N)$. Write $\mathcal{D}=\sum_{i=1}^s \Delta_i \otimes \{p_i\}$. Let $x\in X(\mathcal{D})$ be the vertex of the torus action.
Then, $\pi_1^{\rm reg}(X(\mathcal{D}),x)$ is isomorphic
to the group generated by
\[
t_1,\dots, t_r, b_1,\dots, b_s
\]
with the relations
\begin{itemize}
\item $b_1\cdots b_s$, 
\item $[t_i,t_j]$ for every $1\leq i\leq j \leq r$,
\item $[t_i,b_j]$ for every $i\in \{1,\dots, r\}$ and $j\in \{1,\dots, s\}$, 
\item $t^n$ for every $n\in N_{\mathcal{D}}$, and
\item $t^{\pi_1(v)}b_j^{\pi_2(v)}$ for every $v\in \mathcal{B}(\mathcal{D},p)$.
\end{itemize}
\end{theorem}

\begin{proof}
We have a good quotient 
$\widetilde{X}(\mathcal{D})\rightarrow \pp^1$.
This good quotient is trivial with fiber $X(\sigma)={\rm Spec}(\cc[\sigma^\vee \cap M])$
over $\pp^1\setminus \{p_1,\dots,p_s\}$.
Around each $p_i$ the variety $\widetilde{X}(\mathcal{D})$ has an analytic neighborhood diffeomorphic to
an analytic neighborhood of the fiber of
$X(\sigma(\mathcal{D},p))\rightarrow \mathbb{A}^1$ around zero (see, e.g.,~\cite[Example 2.5]{LS13}).
The equivariant birational contraction
$r\colon \widetilde{X}(\mathcal{D})\rightarrow X(\mathcal{D})$ contract the closure of the $\mathbb{T}$-invariant cycles of the form $X(\tau)\times (\pp^1 \setminus \{p_1,\dots,p_s\})$,
where $\tau\leqslant \sigma$ is a cone intersecting $\deg(\mathcal{D})$
(see, e.g.,~\cite[\S 5]{AIPSV12}).
In order to compute the regional fundamental group of $X(\mathcal{D})$ at $x$, it suffices to compute the regional fundamental group of $\widetilde{X}(\mathcal{D})\setminus {\rm Ex}(r)$.
Indeed, the image of every prime component of ${\rm Ex}(r)$ has codimension at least two in $X(\mathcal{D})$.
If the image of such component is contained in the singular locus of $X(\mathcal{D})$, then it does not contribute to the regional fundamental group.
Thus, it suffices to compute
$\pi_1^{\rm reg}(\widetilde{X}(\mathcal{D})\setminus {\rm Ex}(r))$.

A general fiber of 
$X(\mathcal{D})^{\rm reg}\setminus {\rm Ex}(r) \rightarrow \pp^1$ is isomorphic to the open subvariety of $X(\sigma)$ which corresponds to the regular sub-cones of $\sigma$ that does not intersect $\deg(\mathcal{D})$.
Around each $p_i$, the variety 
$\widetilde{X}(\mathcal{D})^{\rm reg} \setminus {\rm Ex}(r)$ is diffeomorphic to the open subvariety of $X(\sigma(\mathcal{D},p))$ corresponding to the regular sub-cones
of $\sigma(\mathcal{D},p)$ not intersecting $\deg(\mathcal{D})$.
Thus, we have a formally toric description of $\widetilde{X}(\mathcal{D})^{\rm reg} \setminus {\rm Ex}(r)$ 
over $\pp^1$.
Then, the rest of the description follows from applying Van Kampen Theorem to glue the fundamental group
of $X(\sigma)^{\rm reg} \times (\pp^1\setminus \{p_1,\dots,p_s\})$
with those of the analytic neighborhoods of the fibers of the $p_i$'s.
The proof proceeds similarly as in~\cite[Theorem 3.4]{LLM19}.
\end{proof}

In the above theorem, the loop $t_i$ corresponds to a loop around the $i$-th factor of the $r$-dimensional torus $\mathbb{T} \cong (\cc^*)^r$ of a general fiber of $\widetilde{X}(\mathcal{D})\rightarrow \pp^1$.
On the other hand, the loops $b_j$ correspond to liftings to $\mathcal{X}(\mathcal{D})$ of the loops around the points $p_j$ in $\pp^1$.

Note that the above description gives an explicit version of Corollary~\ref{cor:jordan-comp-1}.
Indeed, for the group $A$, we can consider the normal abelian group generated by the $t_i$'s.
Since we have at most three points for which $\mu_p$ is non-trivial,
we conclude that the quotient $\pi_1^{\rm reg}(X(\mathcal{D}),x)$ has order at most $60M$.
Indeed, the quotient $\pi_1^{\rm reg}(X(\mathcal{D}),x)/A$ admits a surjection from
$\pi_1^{\rm reg}(\pp^1,B(\mathcal{D}))$.

Theorem~\ref{thm:reg-compl-1}, gives a simple way to construct the universal cover of a complexity-one $\mathbb{T}$-singularity.
Let $\mathcal{D}$ be a proper polyhedral divisor on $(\pp^1,N)$.
Let $(\pp^1,B(\mathcal{D}))$ be the associated log Fano pair.
Let $p\colon (\pp^1, B')\rightarrow (\pp^1,B(\mathcal{D}))$ be the universal cover of 
$\pi_1(\pp^1, B(\mathcal{D}))$.
Then, $p^ *\mathcal{D}$ is a proper polyhedral divisor on $\pp^1$ and we have a finite quasi-\'etale Galois morphism
\[
p\colon (X(p^*\mathcal{D}), x')
\rightarrow (X(\mathcal{D}),x).
\]
We denote it by $p$ by abuse of notation.
Here, $x'$ is the unique pre-image of $x$.
We are considering the pull-back of proper polyhedral divisors as defined in~\cite[\S 8]{AH06}.
By Theorem~\ref{thm:reg-compl-1}, the regional fundamental group of $(X(p^*\mathcal{D}),x')$ is generated by the loops $t_i$.
In particular, it is abelian.
Hence, its universal cover is nothing else than an isogeny of the torus given by a lattice extension $N \hookrightarrow N'$.

Now, we turn to describe the Cox ring of an affine $\mathbb{T}$-variety of complexity one.
We restrict ourselves to the klt case, so the singularities will impose some restriction on the structure of the Cox ring.

\begin{definition}[Cf.~\cite{ABHW18}]
{\em
Let $\mathcal{D}$ be a proper polyhedral divisor 
on $(\pp^1,N)$ which defines a klt complexity one affine variety $X(\mathcal{D})$.

Fix integers $m \geq 0$, $n,r > 0$, and a partition $n=n_0 + \ldots + n_r$. For every $i=0,\ldots,r$, let $l_i=(l_{i1},\ldots,l_{in_i}) \in \zz^{n_i}$ with $l_{i1} \geq \ldots \geq l_{in_i}>0$ and $l_{i1} \geq \ldots \geq  l_{r1}$. Define monomials $T_i^{l_i}:=T_{i1}^{l_{i1}} \cdots T_{in_i}^{l_{in_i}}$ in the polynomial ring
\[
\cc[T_{ij},S_k]:=\cc[T_{ij},S_k; i=1,\ldots,r, j=1,\ldots,n_i, k=1,\ldots,m].
\]
Now, define pairwise different scalars $\theta_0 =1,\theta_1,\ldots,\theta_{r-a} \in \cc^*$ and for $i=0,\ldots, r-2$ a trinomial
\[
g_i:= \theta_iT_i^{l_i}+ T_{i+1}^{l_{i+1}}+ T_{i+2}^{l_{i+2}}.
\]
If the $\mathfrak{l}_i:=\max(l_{i1},\ldots,l_{in_i})$ are {\em platonic tuples}, i.e. of the form
\[
(5,3,2,1\ldots,1), (4,3,2,1\ldots,1), (3,3,2,1\ldots,1), (k,2,2,1\ldots,1), (k,l,1,1\ldots,1),
\]
then, we call the factor ring $R:=\cc[T_{ij},S_{k}]/\langle g_0,\ldots,g_{r-2}\rangle$ a {\em platonic ring}.
}
\end{definition}

Now, we have the following slight generalization of~\cite[Theorem 1.3]{ABHW18} that was first stated in~\cite[Theorem 5]{BraThesis}.

\begin{theorem}
Let $\mathcal{D}$ be a proper polyhedral divisor 
on $(\pp^1,N)$ which defines a klt complexity one affine variety $X(\mathcal{D})$.
Let $x\in X(\mathcal{D})$ be the vertex of the torus action.
Write $\mathcal{D}=\sum_{i=1}^s \Delta_i \otimes \{p_i\}$ and assume $\mu(p_i)=1$ for $i\geq 4$.
Then, the Cox ring of $(X(\mathcal{D}),x)$ is a platonic ring with associated tuple $(\mu(p_1),\mu(p_2),\mu(p_3))$.
\end{theorem}

Now, we turn to explicitly describe the possible Cox ring iterations in terms of the platonic Cox ring of a klt singularity of complexity one. The original reference is~\cite[Rem 6.7]{ABHW18}.

\begin{theorem}[Cf.~\cite{HW18}]
Let $\mathcal{D}$ be a proper polyhedral divisor on $(\pp^1,N)$ which defines a klt complexity one affine variety $X(\mathcal{D})$.
Then, the possible sequences of platonic triples arising from Cox ring iterations of $X(\mathcal{D})$ are the following:
\begin{itemize}
\item $(1,1,1)\rightarrow (2,2,2)\rightarrow (3,3,2)\rightarrow (4,3,2)$,
\item $(1,1,1)\rightarrow (x,x,1)\rightarrow (2x,2,2)$,
\item $(1,1,1)\rightarrow (x,x,1)\rightarrow (x,2,2)$, and 
\item $(l^{-1}l_0,l^{-1}l_1,1) \rightarrow (l_0,l_1,1)$ where $l:={\rm gcd}(l_0,l_1)>1$.
\end{itemize}
\end{theorem}

Finally, the following theorem describes the simply connected factorial canonical cover of a klt singularity of complexity one.

\begin{theorem}
Let $\mathcal{D}$ be a proper polyhedral divisor on $(\pp^1,N)$ which defines a klt complexity one affine variety $X(\mathcal{D})$.
Let $p\colon \pp^1\rightarrow \pp^1$ be the universal cover of $(\pp^1,B(\mathcal{D}))$.
Then, the scfc cover is ${\rm Cox}(X(p^*\mathcal{D}))$.
\end{theorem}

\begin{proof}
We have a Galois quasi-\'etale finite morphism 
$X(p^*\mathcal{D})\rightarrow X(\mathcal{D})$.
By Theorem~\ref{thm:reg-compl-1}, we know that the regional fundamental group of 
$X(p^*\mathcal{D})$ is abelian and generated by the loops $t_1,\dots, t_r$.
Hence, by Theorem~\ref{thm:scfc=it-local}, we conclude that the scfc cover of $X(p^*\mathcal{D})$, which coincides with the scfc cover of $X(\mathcal{D})$, is isomorphic to
${\rm Cox}(X(p^*\mathcal{D}))$.
\end{proof}

\subsection{Proof of the theorems}
\label{subsec:proofs}
In this subsection, we explain how the theorems in the introduction follow from the theorems proved throughout the manuscript.

\begin{proof}[Proof of Theorem~\ref{introthm2-existence-iteration-local}]
Follows from Theorem~\ref{thm:bounded-iteration}.
\end{proof}

\begin{proof}[Proof of Theorem~\ref{introthm3-bounded-iteration-local}]
Follows from Theorem~\ref{thm:bounded-iteration}.
\end{proof}

\begin{proof}[Proof of Theorem~\ref{introthm4-bounded-dim-it-local}]
Follows from Theorem~\ref{thm:dim-bound-2-homotopy}.
\end{proof} 

\begin{proof}[Proof of Theorem~\ref{introthm-5-existence-scf-cover}]
Follows from Theorem~\ref{thm:scfc-cover}.
\end{proof} 

\begin{proof}[Proof of Theorem~\ref{introthm-6-univ-scf-cover}]
Follows from Theorem~\ref{thm-univ-scfc}.
\end{proof} 

\begin{proof}[Proof of Theorem~\ref{introthm7-smooth-it}]
Follows from Theorem~\ref{thm:smooth-it-local}.
\end{proof}

\begin{proof}[Proof of Theorem~\ref{introthm8-smooth-scfc}]
Follows from Theorem~\ref{thm:smooth-scfc-local}.
\end{proof}

\begin{proof}[Proof of Theorem~\ref{introthm9-equal-it-scfc}]
Follows from Theorem~\ref{thm:scfc=it-local}.
\end{proof} 

\begin{proof}[Proof of Theorem~\ref{introthm10-jordan-relative}]
Follows from Theorem~\ref{thm:rel-finiteness}.
\end{proof}

\begin{proof}[Proof of Theorem~\ref{introthm11-jordan-t-var}]
Follows from Theorem~\ref{thm:jordan-comp-k}.
\end{proof}

\begin{proof}[Proof of Theorem~\ref{introthm12-it-t-var}]
Follows from Theorem~\ref{thm:jordan-comp-k}
and the proof of Theorem~\ref{thm:bounded-iteration}.
\end{proof}

\section{Appendix: Table of covers} 
\label{appendix}

In this appendix, we summarize all the different categories of covers of klt singularities (or Fano type varieties) that we consider throughout this article.
We describe the category of covers over $X$, 
the group that acts on such covers, 
the inverse limit, 
and the main property of the inverse limit.\\

\begin{center} 
\textbf{Table 1.} Covers of klt singularities.\\
\vspace{0.5cm}
\begin{tabularx}
{0.8\textwidth}
{ 
  | >{\centering\arraybackslash}X 
  | >{\centering\arraybackslash}X 
  | >{\centering\arraybackslash}X 
  | >{\centering\arraybackslash}X |}

 \hline
 \begin{center}\textbf{   Category}\end{center} & 
 \begin{center}\textbf{Acting group}\end{center} & 
 \begin{center}\textbf{ Inverse limit }\end{center} &
 \begin{center}\textbf{ Main property}\end{center} \\
 \hline
 \begin{center} 
 Finite Galois quasi-\'etale covers 
 \end{center}& 
 \begin{center} 
 Finite group
 \end{center} & 
 \begin{center}
 Universal cover
 \end{center} &
 \begin{center} 
 Simply connectedness
 \end{center} \\ 
\hline
\begin{center} 
Abelian reductive quasi-\'etale covers
\end{center}
& 
\begin{center} 
Quasi-torus
\end{center} & 
\begin{center}
Cox ring
\end{center}
& 
\begin{center} 
$\mathbb{T}$-factoriality
\end{center} 
\\
\hline 
\begin{center} 
Solvable reductive quasi-\'etale covers
\end{center} &
\begin{center} 
Solvable reductive group
\end{center}&
\begin{center} 
Iteration of Cox ring
\end{center}&
\begin{center}
Factoriality
\end{center} \\ 
\hline 
\begin{center}
Finite-solvable quasi-\'etale covers
\end{center} &
\begin{center} 
Finite extensions of solvable reductive group
\end{center}&
\begin{center} 
Simply connected factorial canonical cover 
(scfc cover)\end{center} &
\begin{center} 
Simply connectedness and factoriality
\end{center} \\ 
\hline 
\end{tabularx}
\end{center} 

\vspace{0.5cm}

Note that all the group isomorphism classes
considered in the above table are closed under extensions.
The following diagram shows the natural morphisms between the different covers in the above table.
    \[
    \xymatrix{
    & {\rm Cox}^{\rm it}(X,\Delta;x)\ar[ld]^-{\phi_3} & (Y,\Delta;y)\ar[l]^-{\phi_1}\ar[dd]^-{\phi_2} \\
    {\rm Cox}(X,\Delta;x)\ar[d]^-{\phi_5} & & \\ 
    (X,\Delta;x) & & (\tilde{X},\tilde{\Delta},\tilde{x}) \ar[ll]^-{\phi_4}
    }
    \]
Here, $(\tilde{X},\tilde{\Delta};\tilde{x})$
is the universal cover of $(X,\Delta;x)$ and $(Y,\Delta_Y;y)$ is the scfc cover of $(X,\Delta;x)$.
We finish the appendix explaining when the morphisms in the above diagram are isomorphisms:
\begin{enumerate}
    \item By Theorem~\ref{thm:scfc=it-local}, $\phi_1$ is an isomorphism if and only if 
    $\pi_1^{\rm reg}(X,\Delta;x)$ is solvable.
    \item $\phi_{i}$ is an isomorphism if and only if the target is factorial for $i\in \{2,3,5\}$.
    \item $\phi_4$ is an isomorphism if and only if $\pi_1^{\rm reg}(X,\Delta;x)$ is trivial.
\end{enumerate}

\bibliographystyle{habbrv}
\bibliography{bib}

\begin{thebibliography}{10}
\expandafter\ifx\csname url\endcsname\relax
  \def\url#1{\texttt{#1}}\fi
\expandafter\ifx\csname doi\endcsname\relax
  \def\doi#1{\burlalt{doi:#1}{http://dx.doi.org/#1}}\fi
\expandafter\ifx\csname urlprefix\endcsname\relax\def\urlprefix{URL }\fi
\expandafter\ifx\csname href\endcsname\relax
  \def\href#1#2{#2}\fi
\expandafter\ifx\csname burlalt\endcsname\relax
  \def\burlalt#1#2{\href{#2}{#1}}\fi

\bibitem{Ale93}
V.~Alexeev.
\newblock Two two-dimensional terminations.
\newblock {\em Duke Math. J.}, 69(3):527--545, 1993.
\newblock \doi{10.1215/S0012-7094-93-06922-0}.

\bibitem{AH06}
K.~Altmann and J.~Hausen.
\newblock Polyhedral divisors and algebraic torus actions.
\newblock {\em Math. Ann.}, 334(3):557--607, 2006.
\newblock \doi{10.1007/s00208-005-0705-8}.

\bibitem{AIPSV12}
K.~Altmann, N.~O. Ilten, L.~Petersen, H.~S\"{u}\ss, and R.~Vollmert.
\newblock The geometry of {$T$}-varieties.
\newblock In {\em Contributions to algebraic geometry}, EMS Ser. Congr. Rep.,
  pages 17--69. Eur. Math. Soc., Z\"{u}rich, 2012.
\newblock \doi{10.4171/114-1/2}.

\bibitem{AP12}
K.~Altmann and L.~Petersen.
\newblock Cox rings of rational complexity-one {$T$}-varieties.
\newblock {\em J. Pure Appl. Algebra}, 216(5):1146--1159, 2012.
\newblock \doi{10.1016/j.jpaa.2011.12.018}.

\bibitem{ACL21}
M.~Artebani, C.~Correa~Deisler, and A.~Laface.
\newblock Cox rings of {K}3 surfaces of {P}icard number three.
\newblock {\em J. Algebra}, 565:598--626, 2021.
\newblock \doi{10.1016/j.jalgebra.2020.08.016}.

\bibitem{AHL10}
M.~Artebani, J.~Hausen, and A.~Laface.
\newblock On {C}ox rings of {K}3 surfaces.
\newblock {\em Compos. Math.}, 146(4):964--998, 2010.
\newblock \doi{10.1112/S0010437X09004576}.

\bibitem{ABHW18}
I.~Arzhantsev, L.~Braun, J.~Hausen, and M.~Wrobel.
\newblock Log terminal singularities, platonic tuples and iteration of {C}ox
  rings.
\newblock {\em Eur. J. Math.}, 4(1):242--312, 2018.
\newblock \doi{10.1007/s40879-017-0179-8}.

\bibitem{ADHL15}
I.~Arzhantsev, U.~Derenthal, J.~Hausen, and A.~Laface.
\newblock {\em Cox rings}, volume 144 of {\em Cambridge Studies in Advanced
  Mathematics}.
\newblock Cambridge University Press, Cambridge, 2015.

\bibitem{AG10}
I.~V. Arzhantsev and S.~A. Ga\u{\i}fullin.
\newblock Cox rings, semigroups, and automorphisms of affine varieties.
\newblock {\em Mat. Sb.}, 201(1):3--24, 2010.
\newblock \doi{10.1070/SM2010v201n01ABEH004063}.

\bibitem{AM69}
M.~F. Atiyah and I.~G. Macdonald.
\newblock {\em Introduction to commutative algebra}.
\newblock Addison-Wesley Publishing Co., Reading, Mass.-London-Don Mills, Ont.,
  1969.

\bibitem{BP02}
V.~V. Batyrev and O.~N. Popov.
\newblock The {C}ox ring of a del {P}ezzo surface.
\newblock In {\em Arithmetic of higher-dimensional algebraic varieties ({P}alo
  {A}lto, {CA}, 2002)}, volume 226 of {\em Progr. Math.}, pages 85--103.
  Birkh\"{a}user Boston, Boston, MA, 2004.
\newblock \doi{10.1007/978-0-8176-8170-8\_5}.

\bibitem{Bech12}
B.~Bechtold.
\newblock Factorially graded rings and {C}ox rings.
\newblock {\em J. Algebra}, 369:351--359, 2012.
\newblock \doi{10.1016/j.jalgebra.2012.07.030}.

\bibitem{BHHN16}
B.~Bechtold, J.~Hausen, E.~Huggenberger, and M.~Nicolussi.
\newblock On terminal {F}ano 3-folds with 2-torus action.
\newblock {\em Int. Math. Res. Not. IMRN}, 1(5):1563--1602, 2016.
\newblock \doi{10.1093/imrn/rnv190}.

\bibitem{BF84}
J.~Bingener and H.~Flenner.
\newblock Variation of the divisor class group.
\newblock {\em J. Reine Angew. Math.}, 351:20--41, 1984.

\bibitem{BCHM10}
C.~Birkar, P.~Cascini, C.~D. Hacon, and J.~McKernan.
\newblock Existence of minimal models for varieties of log general type.
\newblock {\em J. Amer. Math. Soc.}, 23(2):405--468, 2010.
\newblock \doi{10.1090/S0894-0347-09-00649-3}.

\bibitem{Bou78}
J.-F. Boutot.
\newblock {\em Sch\'{e}ma de {P}icard local}, volume 632 of {\em Lecture Notes
  in Mathematics}.
\newblock Springer, Berlin, 1978.

\bibitem{BG08}
C.~P. Boyer and K.~Galicki.
\newblock {\em Sasakian geometry}.
\newblock Oxford Mathematical Monographs. Oxford University Press, Oxford,
  2008.

\bibitem{Bra19}
L.~Braun.
\newblock Gorensteinness and iteration of cox rings for fano type varieties,
  2019, \burlalt{arXiv:1903.07996}{http://arxiv.org/abs/arXiv:1903.07996}.

\bibitem{BraThesis}
L.~Braun.
\newblock Quotient presentations of mori dream spaces, 2019,
  \burlalt{http://dx.doi.org/10.15496/publikation-39049}{http://arxiv.org/abs/http://dx.doi.org/10.15496/publikation-39049}.
\newblock PhD thesis, Universit\"at T\"ubingen.

\bibitem{Bra20}
L.~Braun.
\newblock The local fundamental group of a kawamata log terminal singularity is
  finite, 2020,
  \burlalt{arXiv:2004.00522}{http://arxiv.org/abs/arXiv:2004.00522}.

\bibitem{BFMS20}
L.~Braun, S.~Filipazzi, J.~Moraga, and R.~Svaldi.
\newblock The jordan property for local fundamental groups, 2020,
  \burlalt{arXiv:2006.01253}{http://arxiv.org/abs/arXiv:2006.01253}.

\bibitem{Bri19}
M.~Brion.
\newblock Notes on automorphism groups of projective varieties, 2019.

\bibitem{Bro13}
M.~Brown.
\newblock Singularities of {C}ox rings of {F}ano varieties.
\newblock {\em J. Math. Pures Appl. (9)}, 99(6):655--667, 2013.
\newblock \doi{10.1016/j.matpur.2012.10.003}.

\bibitem{BH93}
W.~Bruns and J.~Herzog.
\newblock {\em Cohen-{M}acaulay rings}, volume~39 of {\em Cambridge Studies in
  Advanced Mathematics}.
\newblock Cambridge University Press, Cambridge, 1993.

\bibitem{Cae83}
S.~Caenepeel.
\newblock Graded complete and graded {H}enselian rings.
\newblock In {\em Methods in ring theory ({A}ntwerp, 1983)}, volume 129 of {\em
  NATO Adv. Sci. Inst. Ser. C: Math. Phys. Sci.}, pages 67--80. Reidel,
  Dordrecht, 1984.

\bibitem{Cam11}
F.~Campana.
\newblock Orbifoldes g\'{e}om\'{e}triques sp\'{e}ciales et classification
  bim\'{e}romorphe des vari\'{e}t\'{e}s k\"{a}hl\'{e}riennes compactes.
\newblock {\em J. Inst. Math. Jussieu}, 10(4):809--934, 2011.
\newblock \doi{10.1017/S1474748010000101}.

\bibitem{Cas09}
A.-M. Castravet.
\newblock The {C}ox ring of {$\overline M_{0,6}$}.
\newblock {\em Trans. Amer. Math. Soc.}, 361(7):3851--3878, 2009.
\newblock \doi{10.1090/S0002-9947-09-04641-8}.

\bibitem{Chen01}
W.~Chen.
\newblock A homotopy theory of orbispaces, 2001,
  \burlalt{ArXiv:0102020}{http://arxiv.org/abs/ArXiv:0102020}.

\bibitem{CTS76}
J.-L. Colliot-Th\'{e}l\`ene and J.-J. Sansuc.
\newblock Torseurs sous des groupes de type multiplicatif; applications \`a
  l'\'{e}tude des points rationnels de certaines vari\'{e}t\'{e}s
  alg\'{e}briques.
\newblock {\em C. R. Acad. Sci. Paris S\'{e}r. A-B}, 282(18):Aii, A1113--A1116,
  1976.

\bibitem{CTS77}
J.-L. Colliot-Th\'{e}l\`ene and J.-J. Sansuc.
\newblock Vari\'{e}t\'{e}s de premi\`ere descente attach\'{e}es aux
  vari\'{e}t\'{e}s rationnelles.
\newblock {\em C. R. Acad. Sci. Paris S\'{e}r. A-B}, 284(16):A967--A970, 1977.

\bibitem{Cox95}
D.~A. Cox.
\newblock The homogeneous coordinate ring of a toric variety.
\newblock {\em J. Algebraic Geom.}, 4(1):17--50, 1995.

\bibitem{DHHKL15}
U.~Derenthal, J.~Hausen, A.~Heim, S.~Keicher, and A.~Laface.
\newblock Cox rings of cubic surfaces and {F}ano threefolds.
\newblock {\em J. Algebra}, 436:228--276, 2015.
\newblock \doi{10.1016/j.jalgebra.2015.04.028}.

\bibitem{Don16}
M.~Donten-Bury.
\newblock Cox rings of minimal resolutions of surface quotient singularities.
\newblock {\em Glasg. Math. J.}, 58(2):325--355, 2016.
\newblock \doi{10.1017/S0017089515000221}.

\bibitem{DG17}
M.~Donten-Bury and M.~Grab.
\newblock Crepant resolutions of 3-dimensional quotient singularities via cox
  rings, 2017,
  \burlalt{arXiv:1701.09149}{http://arxiv.org/abs/arXiv:1701.09149}.

\bibitem{DK17}
M.~Donten-Bury and S.~Keicher.
\newblock Computing resolutions of quotient singularities.
\newblock {\em J. Algebra}, 472:546--572, 2017.
\newblock \doi{10.1016/j.jalgebra.2016.10.042}.

\bibitem{Elk73}
R.~Elkik.
\newblock Solutions d'\'{e}quations \`a coefficients dans un anneau
  hens\'{e}lien.
\newblock {\em Ann. Sci. \'{E}cole Norm. Sup. (4)}, 6:553--603 (1974), 1973.
\newblock \urlprefix\url{http://www.numdam.org/item?id=ASENS_1973_4_6_4_553_0}.

\bibitem{FGL11}
L.~Facchini, V.~Gonz\'{a}lez-Alonso, and M.~Laso\'{n}.
\newblock Cox rings of du {V}al singularities.
\newblock {\em Matematiche (Catania)}, 66(2):115--136, 2011.
\newblock \doi{10.4418/2011.66.2.11}.

\bibitem{Fl81}
H.~Flenner.
\newblock Divisorenklassengruppen quasihomogener {S}ingularit\"{a}ten.
\newblock {\em J. Reine Angew. Math.}, 328:128--160, 1981.
\newblock \doi{10.1515/crll.1981.328.128}.

\bibitem{Gag19}
G.~Gagliardi.
\newblock Luna-vust invariants of cox rings and skeletons of spherical
  varieties, 2019,
  \burlalt{arXiv:1608.08151}{http://arxiv.org/abs/arXiv:1608.08151}.

\bibitem{GOST15}
Y.~Gongyo, S.~Okawa, A.~Sannai, and S.~Takagi.
\newblock Characterization of varieties of {F}ano type via singularities of
  {C}ox rings.
\newblock {\em J. Algebraic Geom.}, 24(1):159--182, 2015.
\newblock \doi{10.1090/S1056-3911-2014-00641-X}.

\bibitem{GW78}
S.~Goto and K.~Watanabe.
\newblock On graded rings. {II}. ({${\bf Z}^{n}$}-graded rings).
\newblock {\em Tokyo J. Math.}, 1(2):237--261, 1978.
\newblock \doi{10.3836/tjm/1270216496}.

\bibitem{Gra18}
M.~Grab.
\newblock Cox rings and symplectic quotient singularities with torus action,
  2020, \burlalt{arXiv:1807.11438}{http://arxiv.org/abs/arXiv:1807.11438}.

\bibitem{EGA4}
A.~Grothendieck.
\newblock \'{E}l\'{e}ments de g\'{e}om\'{e}trie alg\'{e}brique. {IV}. \'{E}tude
  locale des sch\'{e}mas et des morphismes de sch\'{e}mas. {II}.
\newblock {\em Inst. Hautes \'{E}tudes Sci. Publ. Math.}, 1(24):231, 1965.
\newblock \urlprefix\url{http://www.numdam.org/item?id=PMIHES_1965__24__231_0}.

\bibitem{BM17}
M.~B. Guillén and D.~Maclagan.
\newblock A presentation for the cox ring of $\overline{M}_{0,6}$, 2017,
  \burlalt{arXiv:1712.08193}{http://arxiv.org/abs/arXiv:1712.08193}.

\bibitem{HT02}
B.~Hassett and Y.~Tschinkel.
\newblock Universal torsors and {C}ox rings.
\newblock In {\em Arithmetic of higher-dimensional algebraic varieties ({P}alo
  {A}lto, {CA}, 2002)}, volume 226 of {\em Progr. Math.}, pages 149--173.
  Birkh\"{a}user Boston, Boston, MA, 2004.
\newblock \doi{10.1007/978-0-8176-8170-8\_10}.

\bibitem{HKL16}
J.~Hausen, S.~Keicher, and A.~Laface.
\newblock Computing {C}ox rings.
\newblock {\em Math. Comp.}, 85(297):467--502, 2016.
\newblock \doi{10.1090/mcom/2989}.

\bibitem{HLM19}
J.~Hausen, A.~Laface, and C.~Mauz.
\newblock On smooth fano fourfolds of picard number two, 2021,
  \burlalt{arXiv:1907.08000}{http://arxiv.org/abs/arXiv:1907.08000}.

\bibitem{HaSu10}
J.~Hausen and H.~S\"{u}\ss.
\newblock The {C}ox ring of an algebraic variety with torus action.
\newblock {\em Adv. Math.}, 225(2):977--1012, 2010.
\newblock \doi{10.1016/j.aim.2010.03.010}.

\bibitem{HW18}
J.~Hausen and M.~Wrobel.
\newblock On iteration of {C}ox rings.
\newblock {\em J. Pure Appl. Algebra}, 222(9):2737--2745, 2018.
\newblock \doi{10.1016/j.jpaa.2017.10.017}.

\bibitem{HM15}
A.~Hochenegger and E.~Martinengo.
\newblock Mori dream stacks.
\newblock {\em Math. Z.}, 280(3-4):1185--1202, 2015.
\newblock \doi{10.1007/s00209-015-1472-1}.

\bibitem{HMT20}
A.~Hochenegger, E.~Martinengo, and F.~Tonini.
\newblock Cox ring of an algebraic stack, 2020,
  \burlalt{arXiv:2004.01445}{http://arxiv.org/abs/arXiv:2004.01445}.

\bibitem{HK00}
Y.~Hu and S.~Keel.
\newblock Mori dream spaces and {GIT}.
\newblock In {\em Mori dream spaces and {GIT}}, volume~48, pages 331--348. MMJ,
  2000.
\newblock \doi{10.1307/mmj/1030132722}.
\newblock Dedicated to William Fulton on the occasion of his 60th birthday.

\bibitem{LLM19}
A.~Laface, A.~Liendo, and J.~Moraga.
\newblock The fundamental group of a log terminal {$\Bbb T$}-variety.
\newblock {\em Eur. J. Math.}, 5(3):937--957, 2019.
\newblock \doi{10.1007/s40879-018-0296-z}.

\bibitem{LV09}
A.~Laface and M.~Velasco.
\newblock A survey on {C}ox rings.
\newblock {\em Geom. Dedicata}, 139:269--287, 2009.
\newblock \doi{10.1007/s10711-008-9329-y}.

\bibitem{LT17}
K.~Langlois and R.~Terpereau.
\newblock The {C}ox ring of a complexity-one horospherical variety.
\newblock {\em Arch. Math. (Basel)}, 108(1):17--27, 2017.
\newblock \doi{10.1007/s00013-016-0979-y}.

\bibitem{Hui12}
H.~Li.
\newblock On monoid graded local rings.
\newblock {\em J. Pure Appl. Algebra}, 216(12):2697--2708, 2012.
\newblock \doi{10.1016/j.jpaa.2012.03.031}.

\bibitem{LS13}
A.~Liendo and H.~S\"{u}ss.
\newblock Normal singularities with torus actions.
\newblock {\em Tohoku Math. J. (2)}, 65(1):105--130, 2013.
\newblock \doi{10.2748/tmj/1365452628}.

\bibitem{MR17}
L.~Monin and J.~Rana.
\newblock Equations of {$\overline {\rm M}_{0,n}$}.
\newblock In {\em Combinatorial algebraic geometry}, volume~80 of {\em Fields
  Inst. Commun.}, pages 113--132. Fields Inst. Res. Math. Sci., Toronto, ON,
  2017.

\bibitem{Mor20b}
J.~Moraga.
\newblock Fano type surfaces with large cyclic automorphisms, 2020,
  \burlalt{arXiv:2001.03797}{http://arxiv.org/abs/arXiv:2001.03797}.

\bibitem{Mor20a}
J.~Moraga.
\newblock Fano-type surfaces with large cyclic automorphisms.
\newblock {\em Forum Math. Sigma}, 9:Paper No. e54, 2021.
\newblock \doi{10.1017/fms.2021.44}.

\bibitem{Mor20c}
J.~Moraga.
\newblock Kawamata log terminal singularities of full rank, 2021,
  \burlalt{arXiv:2007.10322}{http://arxiv.org/abs/arXiv:2007.10322}.

\bibitem{Muk04}
S.~Mukai.
\newblock Geometric realization of {$T$}-shaped root systems and
  counterexamples to {H}ilbert's fourteenth problem.
\newblock In {\em Algebraic transformation groups and algebraic varieties},
  volume 132 of {\em Encyclopaedia Math. Sci.}, pages 123--129. Springer,
  Berlin, 2004.
\newblock \doi{10.1007/978-3-662-05652-3\_7}.

\bibitem{MFK94}
D.~Mumford, J.~Fogarty, and F.~Kirwan.
\newblock {\em Geometric invariant theory}, volume~34 of {\em Ergebnisse der
  Mathematik und ihrer Grenzgebiete (2) [Results in Mathematics and Related
  Areas (2)]}.
\newblock Springer-Verlag, Berlin, third edition, 1994.
\newblock \doi{10.1007/978-3-642-57916-5}.

\bibitem{Mur69}
M.~P. Murthy.
\newblock Vector bundles over affine surfaces birationally equivalent to a
  ruled surface.
\newblock {\em Ann. of Math. (2)}, 89:242--253, 1969.
\newblock \doi{10.2307/1970667}.

\bibitem{Sam64}
P.~Samuel.
\newblock {\em Lectures on unique factorization domains}.
\newblock Tata Institute of Fundamental Research Lectures on Mathematics, No.
  30. Tata Institute of Fundamental Research, Bombay, 1964.
\newblock Notes by M. Pavman Murthy.

\bibitem{Vez20}
A.~Vezier.
\newblock Cox rings of almost homogeneous sl2-threefolds, 2020,
  \burlalt{arXiv:2009.08676}{http://arxiv.org/abs/arXiv:2009.08676}.

\bibitem{Vez20a}
A.~Vezier.
\newblock Equivariant cox ring, 2020,
  \burlalt{arXiv:2009.08675}{http://arxiv.org/abs/arXiv:2009.08675}.

\end{thebibliography}

\end{document}